\documentclass[a4paper,reqno,11pt]{amsart}

\usepackage{amsmath,amsfonts,amssymb,amsthm,amscd}

\usepackage{graphicx,epsfig}
\usepackage{psfrag}
\usepackage{perpage}
\usepackage{url}
\usepackage{color}
\usepackage[english]{babel}
\usepackage{bbm}
\usepackage{epstopdf}
\usepackage[utf8]{inputenc}
\usepackage[T1]{fontenc}
\usepackage{microtype}
\usepackage{hyperref}
\usepackage{mathabx}

\usepackage[a4paper,scale={0.72,0.74},marginratio={1:1},footskip=7mm,headsep=10mm]{geometry}

\numberwithin{equation}{section}

\newtheorem{theorem}{Theorem}[section]
\newtheorem{lemma}[theorem]{Lemma}
\newtheorem{proposition}[theorem]{Proposition}

\newtheorem{assumption}[theorem]{Assumption}

\usepackage{tikz}
\usetikzlibrary{arrows,automata,positioning,calc,shapes,decorations.pathreplacing,
decorations.markings,shapes.misc,petri,topaths}
\usepackage{pgfplots}
\pgfplotsset{compat=newest}
\usetikzlibrary{plotmarks}
\usepackage{grffile}
\newlength\figureheight
\newlength\figurewidth
\pgfplotsset{
tick label style={font=\scriptsize},
label style={font=\footnotesize},
legend style={font=\footnotesize},
every axis plot/.append style={very thick}
}

\usepackage{rotating}
\usepackage{amsbsy,enumerate}
\usepackage{graphicx}
\usepackage{ccaption}
\usepackage{comment}
\usepackage{mathrsfs} 
\usepackage{bbm}

\newcommand{\bs}{\boldsymbol}

\renewcommand{\hat}{\widehat}

\newcommand{{\paa}[1]}{p_{00,#1}}
\newcommand{{\pab}[1]}{p_{01,#1}}
\newcommand{{\pba}[1]}{p_{10,#1}}
\newcommand{{\pbb}[1]}{p_{11,#1}}
\newcommand{\xin}{X_i^{(n)}}

\newcommand{\eee}{\mathrm{e}}

\newcommand{\ddd}{\mathrm{d}}

\makeatletter
\renewcommand{\fnum@figure}[1]{\textbf{\figurename~\thefigure}. }
\renewcommand{\fnum@table}[1]{\textbf{\tablename~\thetable}. }
\makeatother

\allowdisplaybreaks

\begin{document}

\title[Graphon-valued processes]{Graphon-valued processes with\\ 
vertex-level fluctuations}

\author{Peter Braunsteins}
\address{Korteweg-de-Vries Instituut, Universiteit van Amsterdam, PO Box 94248, 1090 GE Amsterdam, The Netherlands}
\email{pbraunsteins@gmail.com}

\author{Frank den Hollander}
\address{Mathematisch Instituut, Universiteit Leiden, PO Box 9512, 2300 RA Leiden, The Netherlands}
\email{denholla@math.leidenuniv.nl}

\author{Michel Mandjes}
\address{Korteweg-de Vries Instituut, Universiteit van Amsterdam, PO Box 94248, 1090 GE Amsterdam, The Netherlands}
\email{M.H.R.Mandjes@uva.nl}

\date{\today}

\begin{abstract}
We consider a class of graph-valued stochastic processes in which each vertex has a type that fluctuates randomly over time. Collectively, the paths of the vertex types up to a given time determine the probabilities that the edges are active or inactive at that time. Our focus is on the evolution of the associated empirical graphon in the limit as the number of vertices tends to infinity, in the setting where fluctuations in the graph-valued process are more likely to be caused by fluctuations in the vertex types than by fluctuations in the states of the edges given these types. We derive both sample-path large deviation principles and convergence of stochastic processes. We demonstrate the flexibility of our approach by treating a class of stochastic processes where the edge probabilities depend not only on the fluctuations in the vertex types but also on the state of the graph itself.

\vspace{0.5cm}
\noindent

\\
\emph{Acknowledgment.}
The work in this paper was supported by the Netherlands Organisation for Scientific Research (NWO) through Gravitation-grant NETWORKS-024.002.003.

\end{abstract}

\maketitle


\section{Introduction}

\subsection{Background}

Graphons arise as a powerful tool for characterising the limit of a sequence of \emph{dense graphs}, i.e., graphs in which the number of edges scales as the square of the number of vertices. The theory describing these graphons (see e.g.\ \cite{LS06}, \cite{LSb}, \cite{BCLSVa}, \cite{BCLSVb}, \cite{L12}) focuses on the limiting properties of large dense graphs in terms of their subgraph densities. The literature covers both typical and atypical behaviour, a notable result being the large deviation principle (LDP) for homogeneous Erd\H{o}s-R\'enyi random graphs and associated graphons \cite{CV11}, and their inhomogeneous counterparts \cite{DS19}. 

While most of the existing theory focuses on \emph{static} random graphons, the attention has gradually shifted to \emph{dynamic} random graphons (see e.g.\ \cite{R}, \cite{CK20}, \cite{AdHR19}). Two notable contributions in this area are \cite{AdHR19}, which presents a stochastic process limit in the space of graphons for a class of processes where the edges evolve in a dependent manner, and \cite{BdHM20}, which extends the LDP of \cite{CV11} to a sample-path LDP for a dynamic random graph in which the edges switch on and off independently in a random fashion. In \cite{BdHM20} the authors leave open the question whether a sample-path large deviation principle can be established for processes where the edges switch on and off in a dependent manner, such as in \cite{AdHR19}.

\subsection{Motivation}

The goal of the present paper is two-fold: (1) to answer the open question raised in \cite{BdHM20} by establishing a sample-path large deviation principle for a class of processes in which edges evolve in a dependent manner; (2) to strengthen the results in \cite{AdHR19} while working in a more general framework.

In the class of processes we consider, each vertex is assigned a {\it type} that changes randomly over time, and fluctuations in the types of the vertices determine how the edges interact with each other while switching on and off. Specifically, the paths of the types of all the vertices up to time $t$ determine the probability that the edges in the random graph are active at time $t$. Collectively, these paths are called the \emph{driving process}. 

Our results generalise those of \cite{AdHR19} in a number of directions:
\begin{itemize}
\item[(i)]
We establish sample-path large deviations (Theorem \ref{thm:LDPmain}), whereas \cite{AdHR19} restricts attention to diffusion limits.
\item[(ii)] 
We consider a general driving process and a general edge-switching dynamics, whereas \cite{AdHR19} restricts attention to a specific driving process (the multi-type Moran model) and to a specific edge-switching dynamics (modulated by a fitness function). 
\item[(iii)] 
We establish stochastic process convergence in the space of $({\mathscr{W}}, d_\square)$-valued c\`adl\`ag paths (Theorem \ref{thm:CD}), whereas \cite{AdHR19} works in the space of $(\tilde{\mathscr{W}}, \delta_\square)$-valued Skorokhod paths. (For the definition of these two spaces, see Section \ref{sec:defs} below.)
\item[(iv)] 
We allow for processes in which the probabilities that edges are active depend not only on the fluctuations in the types but also on the state of the graph itself, i.e., on which edges are active or not (Section \ref{Sec:APP1}).
\end{itemize}
On the way to proving our results for graphon-valued processes, we also prove a new large deviation principle for static random graphs (Theorem \ref{thm:LDPTP}). This result can be viewed as a generalisation of the large deviation principle for inhomogeneous Erd\H{o}s-R\'enyi random graphs in which each vertex is assigned a random type. Our proofs rely on concentrations estimates, coupling arguments, and continuous mapping. Along the way, several examples are presented. 

The models analysed in this paper have a \emph{one-way dependence}: the states of the edges depend on the types of the vertices, but the types of the vertices do not depend on the states of the edges. There are many natural models that fall into this framework, coming from statistical physics, population genetics and the social sciences, where the strengths of the interactions between particles, alleles or individuals generally depend on the type they carry. It is much harder to analyse models that exhibit a \emph{two-way dependence}. Such models capture the evolution of spins, infections or opinions on dynamic random networks with mutual feedback, an area that so far remains largely unexplored.

\subsection{Outline}

In Section \ref{sec:IRG} we recall basic LDPs for graphons, and present three LDPs for what we call \emph{inhomogeneous random graphs with type dependence} (IRGTs), which are static random objects. In Section \ref{sec:GVP} we look at their dynamic counterparts, which are graph-valued processes, the main result being a \emph{sample-path LDP} in graphon space. We illustrate our results via a running example, and derive \emph{convergence} of the graph-valued process to a graphon process. In Section \ref{sec:ExSD} we describe various applications, and discuss possible extensions. Section~\ref{sec:proof} contains the proofs of our main theorems. Appendix \ref{appA} identifies the rate function in the LDP of the underlying driving process.


\section{Large deviations for static random graphs}
\label{sec:IRG}

While the goal of the present paper is to study a specific class of \emph{dynamic} random graphs, we begin by analyzing their \emph{static} counterparts. The reason is that the marginal distributions of the dynamic random graphs to be considered (introduced in Section \ref{sec:GVP}) at any given time $t \geq 0$ corresponds to a random graph with type dependence (introduced in Section \ref{sec:IRGRT}).

In Section~\ref{sec:defs} we recall a few basic definitions related to graphons. In Section \ref{sec:IHRG} we introduce inhomogeneous Erd\H{o}s-R\'enyi random graphs and recall the large deviation principle for their associated empirical graphons. In Section \ref{sec:IRGRT} we describe a generalisation of inhomogeneous Erd\H{o}s-R\'enyi random graphs, referred to as inhomogeneous random graphs with type dependence (IRGT), which motivate the definition of the class of graph-valued stochastic processes that we will be working with from Section \ref{sec:GVP} onwards. In Section~\ref{sec:assumptions} we state a number of key assumptions that are needed along the way. In Section~\ref{sec:LDP-IRGT} we establish the large deviation principle for the associated empirical graphon processes under the assumption that the driving process satisfies the LDP. The latter assumption is investigated in detail in Appendix \ref{appA}.  


\subsection{Graphs and graphons}
\label{sec:defs}

Let $\mathscr{W}$ be the space of functions $h\colon\,[0,1]^2 \to [0,1]$ such that $h(x,y)=h(y,x)$ for all $(x,y) \in [0,1]^2$, formed after taking the quotient with respect to the equivalence relation of almost everywhere equality. A finite simple graph $G$ on $n$ vertices can be represented as a graphon $h^G \in \mathscr{W}$ by setting
\begin{equation}
h^G(x,y) := \begin{cases}
1 &\quad \text{if there is an edge between vertex }\lceil nx \rceil\text{ and vertex }\lceil ny \rceil, \\
0 &\quad \text{otherwise.}
\end{cases} 
\end{equation}
This object is referred to as an \emph{empirical graphon} and has a block structure. The space of graphons $\mathscr{W}$ is endowed with the \emph{cut distance}
\begin{equation}
\label{cutdist}
d_\square(h_1, h_2) := \sup_{S,T \subseteq [0,1]} \left| \int_{S \times T} \ddd x\, \ddd y\, 
[h_1(x,y)- h_2(x,y)] \right|, \quad h_1,h_2 \in \mathscr{W}.
\end{equation}
It is noted that the space $(\mathscr{W}, d_\square)$ is not compact \cite[Example F.6]{J}.

On $\mathscr{W}$ there is a natural equivalence relation, referred to as `$\sim$'. Letting $\mathscr{M}$ denote the set of measure-preserving bijections $\sigma\colon\, [0,1] \to [0,1]$, we write $h_1 \sim h_2$ when there exists a $\sigma \in \mathscr{M}$ such that $h_1(x,y) = h_2(\sigma(x), \sigma(y))$ for all $(x,y) \in [0,1]^2$. This equivalence relation induces the quotient space $(\tilde{\mathscr{W}}, \delta_\square)$, where $\delta_\square$ is the \emph{cut metric} defined by 
\begin{equation}
\delta_\square(\tilde h_1, \tilde h_2) := \inf_{\sigma_1, \sigma_2 \in \mathscr{M}} d_\square (h_1^{\sigma_1}, h_2^{\sigma_2}), \quad \tilde h_1, \tilde h_2 \in \tilde{\mathscr{W}} . 
\end{equation}
Notably, the space $(\tilde{\mathscr{W}}, \delta_\square)$ \emph{is} compact \cite[Lemma 8]{LS06}.


\subsection{Inhomogeneous Erd\H{o}s-R\'enyi random graph}
\label{sec:IHRG}

Let $r \in \mathscr{W}$ be a \emph{reference graphon}. Fix $n \in \mathbb{N}$ and consider a random graph $\hat G_n$ with vertex set $[n]:=\{1, \dots, n\}$, where the pair of vertices $i,j \in [n]$, $i \neq j$, is connected by an edge with probability $r(\frac{i}{n},\frac{j}{n})$, independently of other pairs of vertices. Write $\mathbb{P}_n$ to denote the law of $\hat G_n$. Use the same symbol to denote the law on $\mathscr{W}$ induced by the map that associates with the graph $\hat G_n$ its graphon $h^{\hat G_n}$. Write $\tilde{\mathbb{P}}_n$ to denote the law of $\tilde h^{\hat G_n}$, the equivalence class associated with $h^{\hat G_n}$.

The following theorem is an extension of the LDP for homogeneous Erd\H{o}s-R\'enyi random graphs established in \cite{CV11}. It was first stated in \cite{DS19} under additional assumptions. These assumptions were subsequently relaxed in \cite{M20}, \cite{BCGPS20}, \cite{DM20}. The following version of the LDP corresponds to \cite[Theorem 4.1]{DM20}.

\begin{theorem}\label{Thm:LDPIRG}
Suppose that $r\log r$ and $(1-r)\log(1-r)$ are integrable. Then the sequence of probability measures $(\tilde{\mathbb{P}}_n)_{n\in \mathbb{N}}$ satisfies the LDP on $(\tilde{\mathscr{W}}, \delta_\square)$ with rate ${n \choose 2}$ and with rate function $\tilde I_r$, i.e.,  
\begin{equation}
\begin{aligned}
&\limsup_{n \to \infty} \frac{1}{{n \choose 2}} \log \tilde{\mathbb{P}}_n(\mathcal{C}) 
\leq -\inf_{\tilde h \in \mathcal{C}} \tilde I_r(\tilde h) \qquad \forall\,\,\mathcal{C} \subseteq \tilde{\mathscr{W}} \text{ closed,}\\
&\liminf_{n \to \infty} \frac{1}{{n \choose 2}} \log \tilde{\mathbb{P}}_n(\mathcal{O}) 
\geq -\inf_{\tilde h \in \mathcal{O}} \tilde I_r(\tilde h) \qquad \forall\,\,\mathcal{O} \subseteq \tilde{\mathscr{W}} \text{ open,}
\end{aligned}
\end{equation}
where
\begin{equation}
\tilde I_r(\tilde h) = \inf_{\sigma \in \mathscr{M}} I_r(h^\sigma),
\end{equation} 
$h$ is any representative of $\tilde h$, and 
\begin{equation}
I_r(h):= \int_{[0,1^2]} {\rm d}x\,{\rm d}y\,\,\mathcal{R}(h(x,y)\,|\, r(x,y)),
\end{equation}
with
\begin{equation}
\mathcal{R}(a\,|\, b) := a \log \frac{a}{b} + (1-a) \log \frac{1-a}{1-b}.
\end{equation}
\end{theorem}


\subsection{Inhomogeneous random graphs with type dependence}
\label{sec:IRGRT}

Consider the following generalisation of the inhomogeneous Erd\H{o}s-R\'enyi random graph defined in Section \ref{sec:IHRG}. Suppose that each vertex $i \in [n]$ is assigned a (possibly random) {\it type} $\xin \in [0,1]$. Denote the \emph{empirical type measure} by
\begin{equation}\label{min}
\mu_n = \frac{1}{n} \sum_{i=1}^n \delta_{\xin}
\end{equation}
and the \emph{empirical type distribution}
\begin{equation}\label{Fn}
F_n({\bs X}^{(n)},x) = \frac{1}{n} \sum_{i=1}^n \mathbbm{1}\{\xin \leq x\},
\qquad x \in [0,1],
\end{equation}
where $\mathbbm{1}\{A\}$ is the indicator function of the event $A$ and ${\bs X}^{(n)}\equiv (X_1^{(n)},\ldots,X_n^{(n)})$. Let $\mathcal{M}([0,1])$ denote the space of measures on $[0,1]$ endowed with the topology of weak convergence. 

The way the graph is constructed is as follows. Whether or not an edge $(i,j)$ is active depends on {\it local} properties, namely, the types $X_i^{(n)}$ and $X_j^{(n)}$, as well as {\it global} properties, namely, the empirical type distribution $F_n({\bs X}^{(n)},\cdot\,)$. Concretely, we let edge $ij$ be active with probability 
\begin{equation}
H\Big(X_i^{(n)},X_j^{(n)},F_n\Big)\equiv H\Big(X_i^{(n)},X_j^{(n)},F_n({\bs X}^{(n)},\cdot\,)\Big),
\end{equation}
where $H\colon\, [0,1]^2\times \mathcal{M}([0,1]) \to [0,1]$ is symmetric in its first two inputs. Given ${\bs X}^{(n)}$, the edge placement is independent for all vertex pairs. We label the resulting sequence of random graphs as $\{G_n \}_{n \in \mathbb{N}}$, and refer to them as \emph{inhomogeneous random graphs with type dependence} (IRGT).

Two relations with inhomogeneous Erd\H{o}s-R\'enyi random graphs are worth mentioning:
\begin{itemize}
\item[$\circ$]
Observe that if 
\begin{equation}
\label{eq:ERcon}
X_i^{(n)} = \frac{i}{n} \quad \forall\,n \in \mathbb{N},\,i \in [n], 
\qquad H(x,y,F)=r(x,y) \quad \forall\,x,y \in [0,1],
\end{equation}
then the IRGT is equivalent to the inhomogeneous Erd\H{o}s-R\'{e}nyi random graph defined in Section \ref{sec:IHRG}. 
\item[$\circ$]
Let $\bar F$ denote the right-continuous generalised inverse of a distribution function $F$ with support $[0,1]$, which is defined in the usual way as 
\begin{equation}
\bar F(u) := \inf \{ x \in [0,1]\colon\,F(x) >u \}, \qquad u \in [0,1).
\end{equation}
For $F \in \mathcal{M}([0,1])$, define the \emph{induced reference graphon} $g^{[F]} \in \mathscr{W}$ by setting 
\begin{equation}
\label{eqn:SRGd}
g^{[F]}(x,y) = H\big(\bar F(x), \bar F(y),F\big).
\end{equation}
Given the type distribution $F_n$, $\tilde h^{G_n}$ has the same distribution as the inhomogeneous Erd\H{o}s-R\'{e}nyi random graph with reference graphon $g^{[F_n]}$. In other words, we have 
\begin{equation}
\label{ob:SS}
\tilde h^{G_n}\,|\, F_n \stackrel{\rm d}{=} \tilde h^{\hat G_n},\,r=g^{[F_n]}.
\end{equation}
This observation is central to the LDP for IRGTs stated in Theorem \ref{thm:LDPTP} below. 
\end{itemize}


\subsection{Key assumptions}
\label{sec:assumptions}

Before stating Theorem \ref{thm:LDPTP}, we make a number of assumptions. 

\begin{assumption}
\label{ass:LDPtypes}
{\rm The sequence of type distributions $(F_n({\bs X}^{(n)},\cdot\,))_{n \in \mathbb{N}}$ satisfies the LDP on $\mathcal{M}([0,1])$ with rate $\ell(n)$ and with rate function $K$.} \hfill$\diamondsuit$
\end{assumption}

When $X^{(n)}_i$, $i \in \mathbb{N}$, are fixed and $\mu_n \to \mu$ in $\mathcal{M}([0,1])$ as $n\to\infty$, as in \eqref{eq:ERcon}, then Assumption \ref{ass:LDPtypes} is satisfied with $\ell(n)=\infty$. Assumption \ref{ass:LDPtypes} holds, for instance, when $\{X^{(n)}_i\}_{n \in \mathbb{N}, i \in [n]}$ are i.i.d.\ random variables with distribution $f$, in which case $\ell(n)=n$ and $K(f^\circ) = {\mathscr H}(f^\circ\,|\, f)$ is the relative entropy (or Kullback-Leibler divergence) of $f^\circ$ with respect to $f$. Assumption \ref{ass:LDPtypes} may also hold with $\ell(n)$ not scaling linearly in $n$. 

To provide an example, we extend the setup in \cite[Example 2.5]{DM20}. Let $p>0$, $f$ be a probability distribution on $\mathbb{R}$ with bounded support, and $\{ Y_i^{(n)}\}_{i \in [\lfloor n^p \rfloor]}$ be i.i.d.\ random variables with distribution $f$. Let $s_n$ be such that 
\begin{equation}
s_n(x)=Y_i^{(n)}, \qquad x \in \left( \frac{i-1}{\lfloor n^p \rfloor}, \frac{i}{\lfloor n^p \rfloor} \right),
\end{equation} 
and identify $s_n$ with its periodic extension to $\mathbb{R}$.
Let $\rho$ be a smooth convolution kernel with compact support, and define 
\begin{equation}
X_i^{(n)} = \int_{\mathbb{R}} {\rm d}y\,\rho\left(\frac{i}{\lfloor n^p \rfloor} - y\right) s_n(y).
\end{equation}
In this case $\ell(n)=n^p$ and 
\begin{equation}
K(v)= \inf_{f^\circ} \left\{{\mathscr H}(f^\circ\,|\,f): v(x)=\int_{\mathbb{R}} {\rm d}y\, \rho(x-y) f^\circ(y), \, x \in \mathbb{R}\right\}.
\end{equation}
We refer the reader to \cite[Example 2.5]{DM20} for the arguments underlying this result.

\begin{assumption}
\label{ass:RG}
{\rm The function $F \mapsto g^{[F]}$ defined in \eqref{eqn:SRGd} is a continuous mapping from $\mathcal{M}([0,1])$ to $(\mathscr{W}, \lVert \cdot \rVert_{L_1})$.} \hfill $\diamondsuit$ 
\end{assumption}

Assumption \ref{ass:RG} holds, for example, when $H(x,y,F) \equiv H^*(x,y)$ (i.e., there is no dependence on the type distribution) and $H^*\colon\, [0,1]^2 \to [0,1]$ is a continuous function. Assumption \ref{ass:RG} also holds when, in addition, $f\colon\, \mathcal{M}([0,1]) \to [0,1]$ and $h\colon\, [0,1]^2 \to [0,1]$ are continuous functions, and 
\begin{equation}
H(x,y;F) = h( H^*(x,y), f(F)) \qquad \forall\,\, [x,y] \in [0,1]^2,\,F\in \mathcal{M}([0,1]).
\end{equation}

In certain settings we require two further assumptions that are of a more technical nature.

\begin{assumption}
\label{ass:AB}
{\rm For all $F \in \mathcal{M}([0,1])$, the induced graphon $g^{[F]}$ is \emph{away from the boundary}, i.e., 
\begin{equation}
\eta \leq g^{[F]}(x,y) \leq 1- \eta \qquad \forall\,\, (x,y) \in [0,1]^2,
\end{equation}
for some $\eta >0$.}\hfill$\diamondsuit$
\end{assumption}

\begin{assumption}
\label{ass:UZ}
{\rm The rate function $K$ has a unique zero, labelled $F^*$.}\hfill$\diamondsuit$
\end{assumption}


\subsection{LDP for IRGTs}
\label{sec:LDP-IRGT}

Let 
\begin{equation}
J(\tilde h) = \inf_{F \in \mathcal{M}([0,1]): \tilde g^{[F]}=\tilde h} K(F)
\end{equation}
and recall that $(r, \tilde h) \mapsto I_{r}(\tilde h)$ is a function from $\mathscr{W} \times \tilde{\mathscr{W}}$ to $\mathbb{R}_+$.
We are now ready to state our LDP for IRGTs.

\begin{theorem}
\label{thm:LDPTP}
Subject to Assumptions \ref{ass:LDPtypes} and \ref{ass:RG} the following hold:
\begin{itemize}
\item[(i)] 
If $\ell(n)=o\left( {n \choose 2} \right)$, then $\{ \tilde h^{\hat G_n} \}$ satisfies the LDP with rate $\ell(n)$ and with rate function $I^*(\tilde h)=J(\tilde h)$. 
\item[(ii)] 
If $\lim_{n \to \infty} \ell(n)/{n \choose 2} =c$ and Assumption \ref{ass:AB} holds as well, then $\{ \tilde h^{\hat G_n} \}$ satisfies the LDP with rate ${n \choose 2}$ and with rate function $I^*(\tilde h)=\inf_{g \in \tilde{\mathscr{W}}}[c J(\tilde g)+ I_g(\tilde h)]$, where $g$ is any representative of $\tilde g$.
\item[(iii)] 
If ${n \choose 2}=o(\ell(n))$ and Assumptions \ref{ass:AB} and \ref{ass:UZ} hold as well, then $\{ \tilde h^{\hat G_n} \}$ satisfies the LDP with rate ${n \choose 2}$ and with rate function $I^*(\tilde h)=I_{g^{[F^*]}}(\tilde h)$. 
\end{itemize}
\end{theorem}

To understand where Theorem \ref{thm:LDPTP} comes from, it is instructive to realize that two random mechanisms play a role, and that the dominant mechanism determines the rare event behavior. Concretely, think of simulating outcomes of $\tilde h^{\hat G_n}$ in two steps:
\begin{itemize}
\item[$\circ$] 
Simulate the types of the vertices, i.e., simulate the type distribution $F_n$.
\item[$\circ$] 
Simulate the edges given $F_n$, i.e., simulate $\tilde h^{\hat G_n}$ given the induced reference graphon $g^{[F_n]}$.
\end{itemize}
Due to Assumption \ref{ass:LDPtypes}, large fluctuations in Step 1 are governed by the LDP with rate $\ell(n)$ and with rate function $K(\cdot)$, whereas due to \eqref{ob:SS} large fluctuations in Step 2 are governed by the LDP with rate ${n \choose 2}$ and with rate function $I_{g^{[F_n]}}$. In particular, this implies that when $\ell(n)=o\left({n \choose 2}\right)$ large fluctuations in $\tilde h^{\hat G_n}$ are most likely to be caused by a rare event in Step~1, whereas when ${n \choose 2}=o(\ell(n))$ large fluctuations in $\tilde h^{\hat G_n}$ are most likely to be caused by a rare event in Step 2. The regime $\ell(n) \asymp {n \choose 2}$ can be viewed as `balanced', in the sense that large fluctuations in $\tilde h^{\hat G_n}$ are most likely to be caused by a combination of rare events in both Steps 1 and 2.
When $\ell(n) = o\left({n \choose 2}\right)$ we say that the IRGT exhibits \emph{vertex-level fluctuations}, whereas when ${n \choose 2}=o(\ell(n))$ we say that it exhibits \emph{edge-level fluctuations}.

The IRGT is of interest in its own right. However, our primary motivation for introducing the IRGT is that it can be generalized in a natural way to a stochastic process. The rough idea behind its dynamic counterpart is that at each point in time the distribution of the process corresponds to an IRGT. We will focus primarily on processes that exhibit vertex-level fluctuations.


\section{Graphon-valued processes}
\label{sec:GVP}

In Section \ref{sec:mod} we first introduce the graph-valued process of interest, which can be viewed as the dynamic counterpart of the model discussed in Section \ref{sec:IRG}. Section \ref{sec:example} describes an illustrative example. Section \ref{sec:GFvc} states the sample-path LDP for the graph-valued process under the assumption that the driving process satisfies an LDP. Section \ref{sec:GFvcalt} states the stochastic process limit for the graph-valued process under the assumption that the driving process satisfies a stochastic process limit.  The latter assumption is investigated in Appendix~\ref{appA}.  
 

\subsection{The model}
\label{sec:mod}

For a given time horizon $T$, let $(G_n(t))_{t \in [0,T]}$ denote our graph-valued process. This process is constructed as follows. Suppose that each vertex $i \in [n]$ has a type $X^{(n)}_i(t)$ that may fluctuate randomly over time. Let $(\mu_n(t))_{t \in [0,T]}$ denote the process of \emph{empirical type measures} defined by 
\begin{equation}
\mu_n(t) = \frac{1}{n} \sum_{i=1}^n \delta_{X_i^{(n)}(t)},
\end{equation}
i.e., the dynamic version of \eqref{min}. This process evolves autonomously, i.e., independently of the graph-valued process $(G_n(t))_{t \in [0,T]}$.

In addition, let  $(F_n(t; \cdot))_{t \in [0,T]}$ denote the process of \emph{empirical type distribution} defined by 
\begin{equation}
F_n(t;x) = \frac{1}{n} \sum_{i=1}^n \mathbbm{1}\{X_i^{(n)}(t) \leq x\}, \qquad x \in [0,1],
\end{equation}
i.e., the dynamic counterpart of \eqref{Fn}. The process $(\mu_n(t))_{t \in [0,T]}$ lives on $D(\mathcal{M}([0,1]),[0,T])$, the space of $\mathcal{M}([0,1])$-valued c\`adl\`ag paths. We suppose that, at any given time $t$, edge $ij$ is active with probability 
\begin{equation}
H\big(t; X^{(n)}_i(t), X^{(n)}_j(t), (F_n(t;\cdot))_{t \in [0,T]}\big),
\end{equation}
independently of all other edges at time $t$, of $X^{(n)}_i(t)$, $X^{(n)}_j(t)$, and of $(F_n(t;\cdot))_{t \in [0,T]}$, where 
\begin{equation}
H\colon\, [0,T] \times [0,1]^2 \times D(\mathcal{M}([0,1]),[0,T] ) \mapsto [0,1]. 
\end{equation}
This function $H$ gives rise to the \emph{induced reference graphon process} $g^{[F]}$, which, for $F \in D( \mathcal{M}([0,1]),[0,T])$, is characterised by 
\begin{equation}
\label{eq:Gtdef}
g^{[F]}(t;x,y) = H\big(t; \bar F(t;x), \bar F(t;y), (F(t; \cdot)_{t \in [0,T]})\big).
\end{equation}
Observe that, for any $t \in [0,T]$, given the outcome of the empirical type distribution $F_n(t,\cdot\,)$, the distribution of $\tilde h^{G_n(t)}$ corresponds to that of an inhomogeneous Erd\H{o}s-R\'{e}nyi random graph with reference graphon $g^{[F_n]}(t;\cdot, \cdot)$. In other words, for any $t \in [0,T]$,
\begin{equation}
\label{Ob:spIRG}
h^{G_n(t)} | F_n \stackrel{\rm d}{=} h^{\hat G_n}, \qquad r = g^{[F_n]}(t;\cdot),
\end{equation}
where $\hat G_n$ is the inhomogeneous Erd\H{o}s-R\'{e}nyi random graph defined in Section \ref{sec:IHRG}. We make the following assumption on the function $F \mapsto g^{[F]}$, which due to \eqref{eq:Gtdef} is an assumption on $H$.

\begin{assumption}
\label{ass:CP}
{\rm The map $F \mapsto g^{[F]}$ from $D(\mathcal{M}([0,1]),[0,T],)$ to $D((\mathscr{W},\lVert \cdot \rVert_{L_1}),[0,T])$ is continuous.}\hfill$\diamondsuit$
\end{assumption}


\subsection{An illustrative example}
\label{sec:example}

Suppose that $(G_n(t))_{t \in [0,T]}$ is characterised by the following dynamics:
\begin{itemize}
\item 
$G_n(0)$ is the empty graph.
\item 
Each vertex is assigned an independent Poisson clock with rate $\gamma$, i.e., the time intervals between two consecutive ring times are exponentially distributed with parameter $\gamma$. Each time the clock attached to vertex $v$ rings, all the edges are adjacent to $v$ become inactive.
\item 
If the edge $ij$ is inactive, then it becomes active at rate $\lambda$, independently of anything else.
\end{itemize}

We first describe the associated \emph{driving process}. Let $\{\tau_k(v)\}_{k \in \mathbb{N}}$ denote the sequence of times at which the Poisson clock attached to vertex $v$ rings, and let 
\begin{equation}
Y_v(t) := t -  \max_k \{ \tau_k (v)\colon\, \tau_k(v) \leq t \}\vee 0
\end{equation}  
denote the time since the clock last rung. The value of $Y_v(t)$ can be thought of as the \emph{age of vertex $v$ at time $t$}: each time the clock associated with $v$ rings, it dies and all its adjacent edges are lost. Recalling that we assumed that types take values in $[0,1]$, we write 
\begin{equation}\label{eq:XvD}
X_v(t) := F^{{\rm exp}}(Y_v(t))=1 - \eee^{-\gamma Y_v(t)}
\end{equation}
to denote the \emph{type of vertex $v$ at time $t$}, where $F^{{\rm exp}}(\cdot)$ can be interpreted as the distribution function of an exponential random variable with rate $\gamma$.

The function $H(t;u,v,F)$ can now also be identified. The probability that there is an active edge between vertices of ages $\bar u$ and $\bar v$ is $1- \exp\{ -\lambda (\bar u \wedge \bar v) \}$. Putting $u = F^{{\rm exp}}(\bar u)$ and $v = F^{{\rm exp}}(\bar v)$, we obtain, using that $\bar u=-\log(1-u)/\gamma$ and $\bar v=-\log(1-v)/\gamma$,
\begin{equation}\label{eq:IEH}
H(t;u,v,F) =1- \exp\left(\lambda \left(\frac{1}{\gamma}\log(1-u \wedge v\right)\right)= 1 - (1 - u \wedge v)^{\lambda/\gamma}.
\end{equation}
Because $H(t;u,v,F)$ is a continuous function of $u$ and $v$, and is independent of $t$ and $F$, it is straightforward to verify that Assumption \ref{ass:CP} holds. A more involved example is given in Section \ref{Sec:APP1}.


\subsection{Sample-path large deviations}
\label{sec:GFvc}

Similarly as in Section \ref{sec:LDP-IRGT}, we assume that the driving process satisfies the LDP (which for the above illustrative example is established in Lemma \ref{lem:LDPDP}). 

\begin{assumption}
\label{ass:DP}
{\rm $\{F_n\}_{n \in \mathbb{N}}$ satisfies the LDP on $D( \mathcal{M}([0,1]),[0,T])$ with rate $\ell(n)=o({n \choose 2})$ and with rate function $K$.}\hfill$\diamondsuit$
\end{assumption}

To establish the sample-path LDP for the graphon-valued process $\{(\tilde h^{\hat G_n(t)})_{t \geq 0}\}_{n \in \mathbb{N}}$, we need to: 
\begin{itemize}
\item[(I)] 
establish the LDP in the pointwise topology; 
\item[(II)] 
strengthen this topology by establishing exponential tightness. 
\end{itemize}

Step (I) is settled by the following result. For $\tilde h\in D((\tilde{\mathscr{W}}, \delta_\square), [0,T])$, let 
\begin{equation}
    J(\tilde h) = \inf_{F \in D(\mathcal{M}([0,1]),[0,T]) : \tilde g^{[F]}=\tilde h} K(F).
\end{equation}

\begin{proposition}
\label{pr:point}
If Assumptions \ref{ass:CP} and \ref{ass:DP} hold, then the sequence $((\tilde h^{ G_n(t)})_{t \geq 0})_{n \in \mathbb{N}}$ satisfies the LDP in the pointwise topology with rate $\ell(n)$ and with rate function $J(\tilde h)$.
\end{proposition}

Note that Proposition \ref{pr:point} does not refer to any edge-switching dynamics. Specifically, if two process $\{ G_n \}_{n \in \mathbb{N}}$ and $\{G^*_n \}_{n \in \mathbb{N}}$ have a common sequence of types $( (X_i(t))_{t \geq 0})_{i \in [n]}$ and a common edge-connection function $H$, then the marginal distributions are equivalent, i.e.,
\begin{equation}
\tilde h^{G_n(t)} \stackrel{\rm d}{=} \tilde h^{G^*_n(t)}, \qquad t\in [0,T].
\end{equation}
However, this does \emph{not} necessarily mean that the joint distributions are equivalent, i.e., we may have
\begin{equation}
\big(\tilde h^{ G_n(t)}\big)_{t \in [0,T]} \stackrel{\rm d}{\neq} \big(\tilde h^{ G^*_n(t)}\big)_{t \in [0,T]},
\end{equation}
because these depend on the specific edge-switching dynamics. Nonetheless, Proposition \ref{pr:point} implies that both  $\{ G_n \}_{n \in \mathbb{N}}$ and $\{  G^*_n \}_{n \in \mathbb{N}}$ satisfy equivalent LDPs in the pointwise topology, i.e., the rate function depends only on the marginal distributions of the process and not on the specific edge-switching dynamics. In Sections \ref{Sec:APP1} and \ref{sec:ExDS1} we provide examples of processes with equivalent marginals and different edge-switching dynamics. 

The specific edge-switching dynamics \emph{do} need to be taken into consideration when we want to perform step (II), i.e., strengthen the topology of the LDP in Proposition \ref{pr:point} by establishing exponential tightness. We next provide a condition that can be used to verify that $\{ \tilde h^{G_n} \}_{n \in \mathbb{N}}$ are exponentially tight. Let 
\begin{equation}
E^{(n)}_{ij}(t) = 
\begin{cases}
1, \quad &\text{if edge $ij$ is active at time $t$,} \\
0, \quad &\text{otherwise},
\end{cases}
\end{equation}
and define
\begin{equation}\label{cndef}
C_n(t, \delta) = \sum_{1 \leq i < j \leq n} \sup_{t \leq u \leq v \leq t+\delta} |E^{(n)}_{ij}(u)-E^{(n)}_{ij}(v)|.
\end{equation}
In other words, $C_n(t,\delta)$ is the number of edges that change (i.e., go from active to inactive or vice versa) at some time between $t$ and $t+\delta$.

\begin{proposition}
\label{lem:TC}
If, for all $t \in [0,T]$ and $\varepsilon>0$,
\begin{align}
\label{eq:Tcon}
\lim_{\delta \downarrow 0} \limsup_{n \to \infty} \frac{1}{\ell(n)} \log \mathbb{P}\left(C_n(t, \delta) > \varepsilon {n \choose 2}\right)= - \infty,
\end{align}
then $((\tilde h^{G_n(t)})_{t \geq 0})_{n \in \mathbb{N}}$ is exponentially tight.
\end{proposition}

Combining the above two propositions, we obtain the following.

\begin{theorem}
\label{thm:LDPmain}
If the conditions of Propositions \ref{pr:point} and \ref{lem:TC} are satisfied, then the sequence of processes $(\tilde h^{ G_n(t)})_{t \geq 0})_{n \in \mathbb{N}}$ satisfies the LDP on $D(\tilde{\mathscr{W}},[0,T])$ with rate $\ell(n)$ and with rate function $J(\tilde h)$.
\end{theorem}

In view of Lemma \ref{lem:LDPDP}, the conditions of Theorem \ref{thm:LDPmain} can be readily verified for the illustrative example. In Theorem \ref{thm:CD} we establish a sample-path LDP for a class of processes that includes the illustrative example.


\subsection{Stochastic process convergence}
\label{sec:GFvcalt}

In the sequel, $\Rightarrow$ denotes convergence in distribution, and $\stackrel{\rm  fdd}{\Rightarrow}$ convergence of the associated finite-dimensional distributions. We assume that the empirical type distribution has a stochastic process limit.

\begin{assumption}
\label{ass:CD}
{\rm Suppose that $F_n \Rightarrow F$ as $n\to\infty$ on $D(\mathcal{M}([0,1]), [0,T])$.}\hfill$\diamondsuit$
\end{assumption}

We establish the stochastic process limit of $(h^{G_n(t)})_{t \in [0,T]}$ as $n\to\infty$ on $D((\mathscr{W},d_{\square}), [0,T])$, i.e., we no longer take the quotient with respect to the equivalence relation $\sim$. To establish a stochastic process limit in this finer topology, as explained below, we need to ensure that the labels of the vertices update dynamically.

\begin{assumption}\label{ass:VL}
At any time $t \in [0,1]$ the labels of the vertices are such that
\begin{equation}
X_1(t) \leq \dots \leq X_n(t).
\end{equation}
\hfill$\diamondsuit$
\end{assumption}

The importance of Assumption \ref{ass:VL} is illustrated in Figure \ref{MComp}, where we consider the illustrative example from Section \ref{sec:example} with $t=1$, $\lambda=6$ and $\gamma=3$. The left panel shows an outcome of $h^{G_{100}(1)}$ with a \emph{static labelling}, where vertices are labelled arbitrarily at time $t=0$ and their labels do not change over time. We observe that $h^{G_{100}(1)}$ has no discernible structure, and so with probability 1 the sequence $(h^{G_n(1)})_{n \in \mathbb{N}}$ \emph{does not converge to a limit in $(\mathscr{W}, d_\square)$}. The center panel of Figure \ref{MComp} illustrates an outcome of $h^{G_{100}(1)}$ under the \emph{dynamic labelling} given in Assumption \ref{ass:VL}. Under this assumption $h^{G_{100}(1)}$ has a discernible structure and $h^{G_n(1)} \to g$ in $(\mathscr{W}, d_\square)$ as $n \to \infty$, where $g$ is the smooth graphon illustrated in the right panel of Figure \ref{MComp}.

\begin{figure}
\begin{center}
\includegraphics[width=4.96cm]{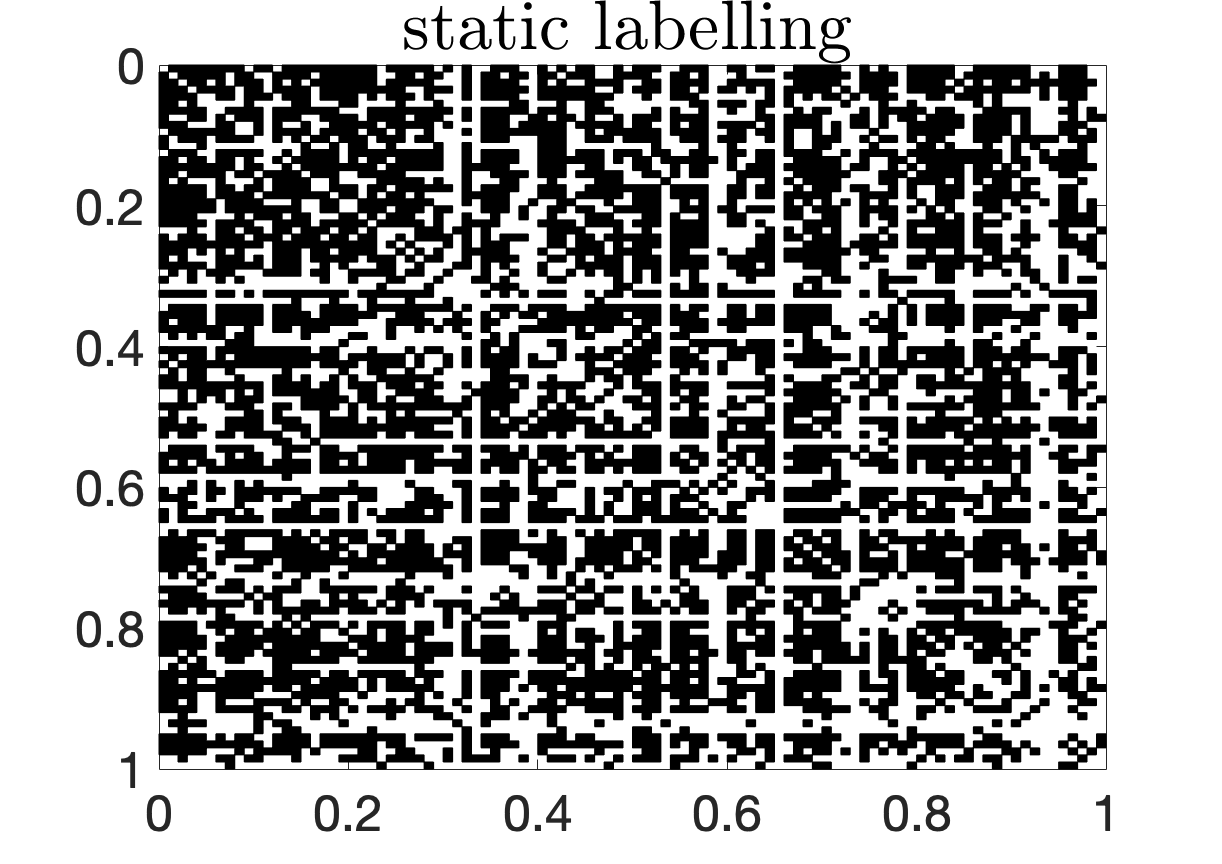}
\includegraphics[width=4.96cm]{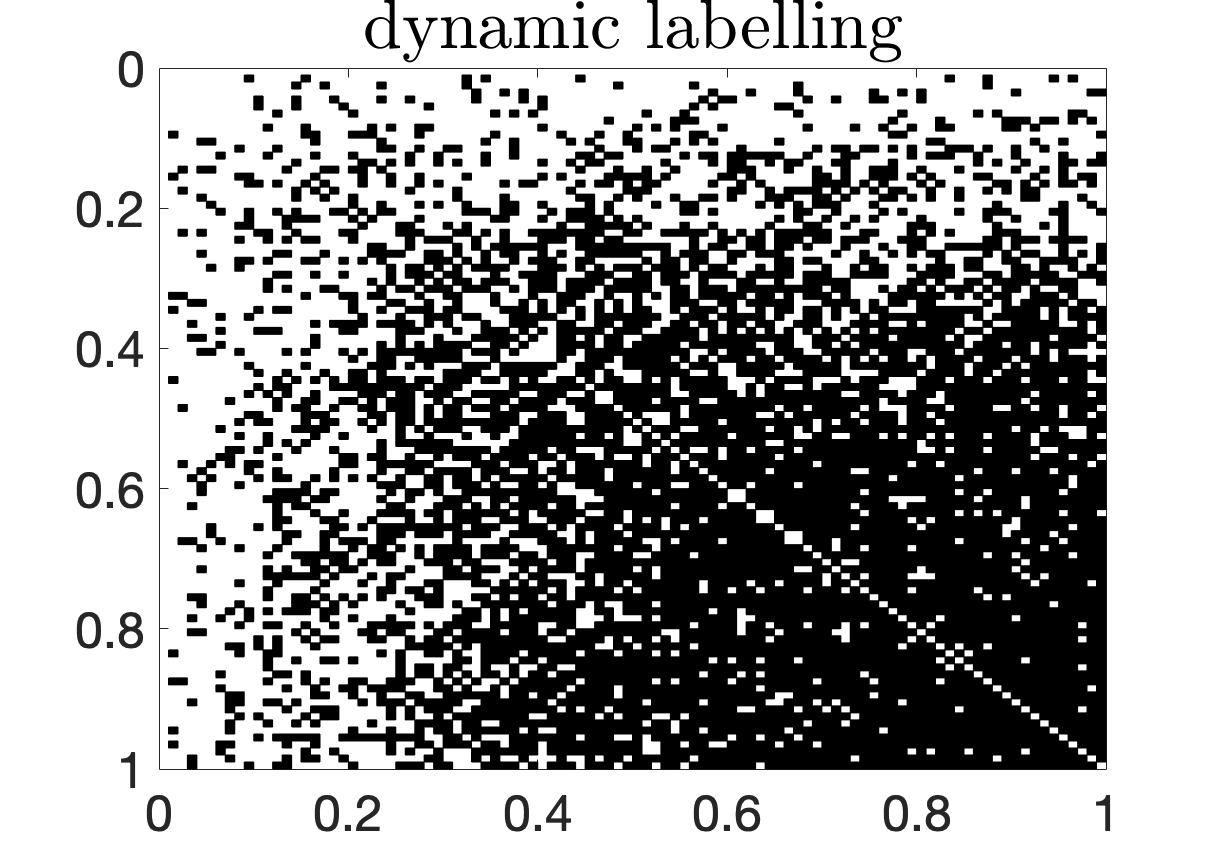}
\includegraphics[width=4.96cm]{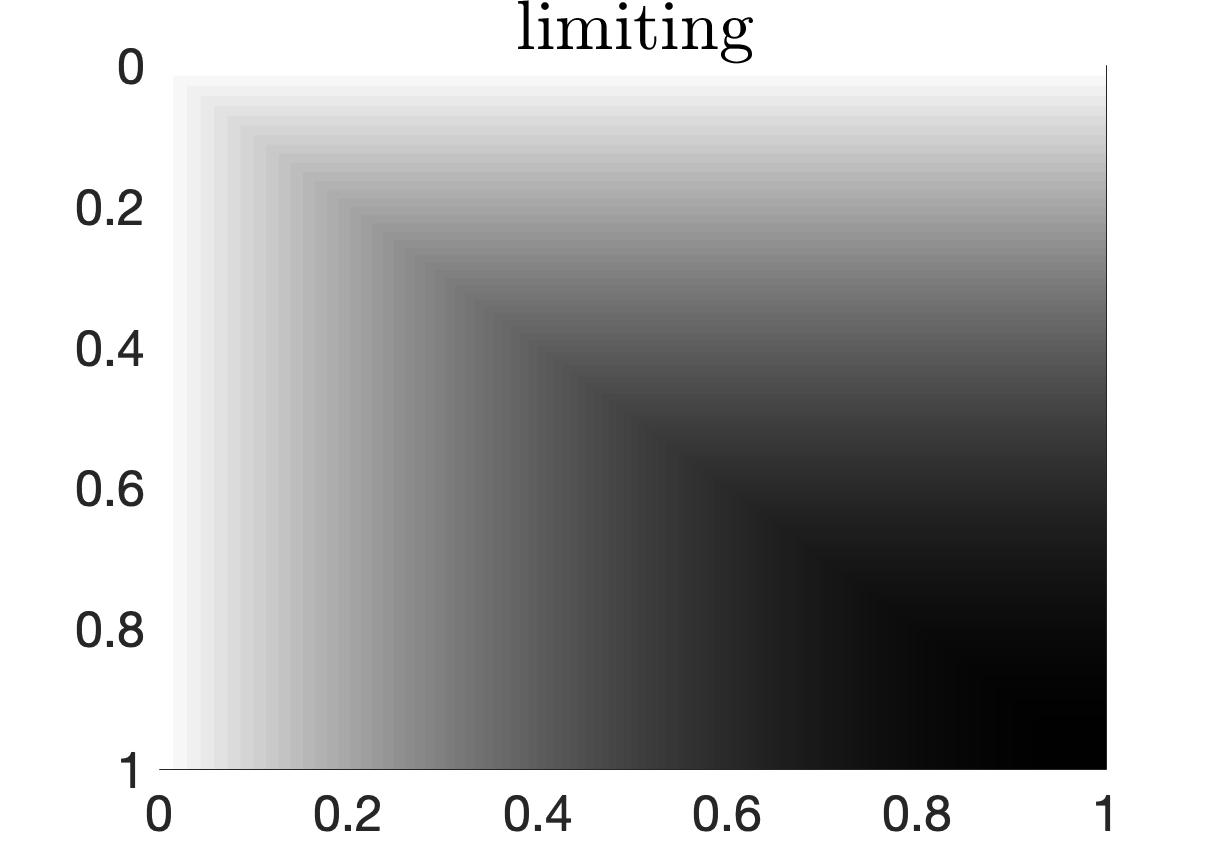}
\end{center}
\caption{\label{MComp} An illustration of $h^{G_{100}(1)}$ with static labelling (left panel) and dynamic labelling (center panel), and the corresponding limit for the dynamic labelling in $(\mathcal{W}, d_\square)$ (right panel). In this figure black corresponds to the value 1 and white corresponds to the value 0.}
\end{figure}

Given the dynamic labelling in Assumption \ref{ass:VL} and the illustrative example, the motivation behind establishing our stochastic process limits on $D((\mathscr{W},d_{\square}), [0,T])$ rather than on $D((\tilde{\mathscr{W}},\delta_{\square}), [0,T])$ is clear: it provides insight into the natural question `Are the older vertices more connected than the younger vertices?' If we establish a limit in $D((\mathscr{W},d_{\square}), [0,T])$, then we have a definitive answer, whereas if we establish a limit in $D((\tilde{\mathscr{W}},\delta_{\square}), [0,T])$, then we do not gain any insight.

\begin{proposition}
\label{prop:fdd}
If Assumptions \ref{ass:CP}, \ref{ass:CD} and \ref{ass:VL} hold, then $h^{G_n} \stackrel{\rm fdd}{\Rightarrow} g^{[F]}$ as $n\to\infty$. 

\end{proposition}

Strengthening the topology to obtain convergence in distribution on $D((\mathscr{W},d_\square), [0,T])$ is more difficult than in Section \ref{sec:GFvc}, because, unlike the space $(\tilde{\mathscr{W}},\delta_\square)$, the space $(\mathscr{W},d_\square)$ is not Polish, and hence we cannot directly apply established sufficient conditions for tightness, such as those stated in \cite[Sections 3.6--3.9]{EK09}. Nonetheless, we are able to establish convergence directly by using \cite[Corollary 3.3]{EK09}.

\begin{assumption}
\label{ass:FLT}
For any $\varepsilon >0$,
\begin{equation}
\lim_{\delta \downarrow 0} \limsup_{n \to \infty} \sup_{t \in [0,T]} \frac{T}{\delta}\, 
\mathbb{P}\left(C_n(t, \delta) > \varepsilon {n \choose 2}\right)=0.
\end{equation}
\end{assumption}

\begin{theorem}
\label{thm:CD}
Subject to Assumptions \ref{ass:CP}, \ref{ass:CD}--\ref{ass:VL}, \ref{ass:FLT}, $h^{G_n} \Rightarrow g^{[F]}$ as $n\to\infty$ in $D((\mathscr{W},d_\square), [0,T])$.
\end{theorem}

In view of Lemma \ref{lem:LDPDP}, the conditions of Theorem \ref{thm:CD} can again be readily verified for the illustrative example. We establish a more general result in Proposition \ref{prop:CDLLN}.


\section{Applications and extensions}
\label{sec:ExSD}

In this section we consider a class of processes that generalise the illustrative example. We use this class of processes to make three points that we believe apply more generally:
\begin{itemize}
\item[(I)] 
An additional layer of dependence between the edges can be introduced that cannot be captured by the types of the vertices (so that \eqref{Ob:spIRG} no longer holds), but still allows to establish limiting results in the spirit of Section \ref{sec:GVP}. Roughly speaking, this is the case when the additional layer of dependence between the edges is of mean-field type (see Section \ref{Sec:APP1}).
\item[(II)] 
The specific edge-switching dynamics rarely affects the limiting path of the process (see Section \ref{sec:ExDS1}).
\item[(III)] 
The dependence between edges in inhomogeneous random graphs with type dependence leads to new behaviour in the corresponding variational problems, even in relatively simple settings (see Section \ref{sec:ExDS3}).
\end{itemize}


\subsection{Beyond conditional independence of edges}
\label{Sec:APP1}


\subsubsection{Model and LDP}

Suppose that $(G_n(t))_{t \in [0, T]}$ is characterised by the following dynamics:
\begin{itemize}
\item 
$G_n(0)$ is the empty graph.
\item
Each vertex is assigned an independent rate-$\gamma$ Poisson clock, and each time the clock associated with vertex $v$ rings all the edges that are adjacent to $v$ become inactive.
Let $\{\tau_k(v)\}_{k \in \mathbb{N}}$ denote the sequence of times at which the Poisson clock attached to vertex $v$ rings, and let, as before,
\begin{equation}
X_v(t) := t - \max_k \{ \tau_k (v)\colon\, \tau_k(v) \leq t \} \vee 0
\end{equation}  
denote the time since the clock last rung (i.e., the `age' of vertex $v$). To ensure that $X_v(t) \in [0,1]$ we take $T=1$. Note that the definition $X_v(t)$ here differs slightly from that in the illustrative example (Section \ref{sec:example}).
\item 
If edge $ij$ is inactive, then it becomes active at rate $\lambda(t,X_i(t),X_j(t),F_n(t; \cdot), \tilde h^{G_n(t)})$.
\item
If edge $ij$ is active, then it becomes inactive at rate $\mu(t,X_i(t),X_j(t),F_n(t; \cdot), \tilde h^{{G_n(t)}})$. 
\end{itemize}
We assume that $\lambda$ and $\mu$ are Lipshitz-continuous functions on $[0,T] \times [0,1]^2 \times \mathcal{M}([0,1]) \times \tilde{\mathscr{W}}$.

If $\lambda(\cdot)$ and $\mu(\cdot)$ do not depend on the current state of the unlabelled graph $\tilde h^{{G_n(t)}}$, i.e.,  if
\begin{equation}
\label{eq:FF}
\lambda(t,u,v,F,\tilde h) \equiv \lambda(t,u,v,F), \qquad \mu(t,u,v,F,\tilde h) \equiv \mu(t,u,v,F),
\end{equation}
then the process fits into the framework of Section \ref{sec:GVP}, otherwise it does not. To understand why, we compute the probability that edge $ij$ is active under \eqref{eq:FF} at time $t$ given $X_i(t)=x_i$, $X_j(t)=x_j$ and $F_n=F$. We have
\begin{equation}
\begin{aligned}
\label{eq:HFF}
H(t; x_i, x_j, F) &= \int_{t - x_i \wedge x_j}^t {\rm d}s\, \lambda(s, x_i-t+s, x_j-t+s, F(s; \cdot)) \\
&\qquad\times \exp \bigg\{ - \int_s^t {\rm d}a\, [\mu(a, x_i-t+a, x_j-t+a, F(a; \cdot)) \\
&\qquad \qquad +  \lambda(a, x_i-t+a, x_j-t+a, F(a; \cdot))]  \bigg\};
\end{aligned}
\end{equation}
an explanation of the expression on the right-hand side is given below, when we discuss a related differential equation version. 
It is also easy to see that two edges $ij$ and $k\ell$ are independent given $t$, $X_i(t)$, $X_j(t)$, $X_k(t)$, $X_\ell(t)$ and $F_n=F$, and hence the process indeed falls into the framework of Section \ref{sec:GVP}. To recover the illustrative example, take $\lambda(t,u,v,F)=\lambda \in \mathbb{R}_+$ and $\mu(t,u,v,F)=0$. Note that through \eqref{eq:Gtdef} and \eqref{eq:HFF}, we see that the empirical reference graphon $g^{[F]}(\cdot)$ is characterised by 
\begin{equation}
\begin{aligned}
\label{eq:HFF2}
g^{[F]}(t; x, F) &= H(t; \bar{F}(t;x), \bar{F}(t;y), F).
\end{aligned}
\end{equation}

An example of a choice for $\lambda(\cdot)$ that \emph{does} depend on the current state of the unlabelled graph $\tilde h^{{G_n(t)}}$ is 
\begin{equation}
\label{eq:lex}
\lambda(t,u,v,F,\tilde h) = 1+ s(G,\tilde h),
\end{equation}
where $G$ is a simple graph (e.g.\ a triangle) and $s(G,\tilde h)$ denotes the homomorphism density of $G$ in $\tilde h$. Note that, by the counting lemma (see, for instance, \cite[Proposition 2.2]{C17}), this particular choice of $\lambda(\cdot)$ is Lipshitz-continuous. In this case, the two edges $ij$ and $k\ell$ are \emph{not} independent given $t$, $X_i(t)$, $X_j(t)$, $X_k(t)$, $X_\ell(t)$ and $F_n=F$.  Indeed, if edge $ij$ is active, then it may participate in additional copies of $G$, which means that edge $k\ell$ is more likely to be active. This dependence is inherent to the model, in that it cannot be removed by changing the definition of the types $X_i(t)$.

We continue by giving some background on the next main result, Theorem \ref{thm:CDalt}, which demonstrates that, despite the above observation, we can still establish a sample-path LDP for the process. In order to express the rate function of this LDP, we need to define a mapping $F \mapsto g^{(F)}$ such that $g^{(F)}(t;x,y)$ can be interpreted as the (approximate) probability that there is an edge between vertices $\lceil n x \rceil$ and $\lceil n y \rceil$ at time $t$. We proceed in a similar manner as above. Given $X_i(t)=x_i$, $X_j(t)=x_j$ and $F_n=F$, the probability that there is an edge between vertices $i$ and $j$ is given by 
\begin{equation}
\label{eq:HD1}
\begin{aligned}
H(t; x_i, x_j, F) &= \int_{t - x_i \wedge x_j}^t {\rm d}s\, \lambda(s, x_i-t+s, x_j-t+s, F(s; \cdot), \tilde h^{G(s)}) \\
&\qquad \times \exp \bigg\{ - \int_s^t {\rm d}a\, [\mu(a, x_i-t+a, x_j-t+a, F(a; \cdot), \tilde h^{G(a)}) \\
&\qquad \qquad+  \lambda(a, x_i-t+a, x_j-t+a, F(a; \cdot), \tilde h^{G(a)})]  \bigg\}.
\end{aligned}
\end{equation}
Because this expression depends on $\tilde h^{G(\cdot)}$ (in addition to $F$ and $t$), it cannot be used to define a mapping $F \mapsto g^{(F)}$ through \eqref{eq:Gtdef}. However, if we tacitly assume that $\tilde h^{G(\cdot)}$ is well approximated by $\tilde g^{(F)}(\cdot)$, then we can use \eqref{eq:HD1} with $\tilde h^{G(t)}$ replaced by $\tilde g^{(F)}$ in combination with \eqref{eq:Gtdef}, to implicitly define the mapping $F \mapsto g^{(F)}$ by
\begin{equation}
\label{eq:HD2}
\begin{aligned}
g^{(F)}(t;x,y)&=\int_{t - \bar F(t;x) \wedge \bar F(t;y)}^t {\rm d}s\, \lambda(s, \bar F(s;x)-t+s, \bar F(s;y)-t+s, F(s; \cdot), \tilde g^{(F)}(s)) \\
&\quad \times \exp \bigg\{ - \int_s^t {\rm d}a\, [\mu(a, \bar F(a;x)-t+a, \bar F(a;y)-t+a, F(a; \cdot), \tilde g^{(F)}(a)) \\
&\quad \quad  +  \lambda(a, \bar F(a;x)-t+a, \bar F(a;y)-t+a, F(a; \cdot), \tilde g^{(F)}(a))]  \bigg\}.
\end{aligned}
\end{equation}

The mapping in \eqref{eq:HD2} is well-defined for all the relevant paths of $F$ (being such that $F \in \bar{\mathcal{M}}_X$, where $\bar{\mathcal{M}}_X$ is defined below), which can be seen by re-expressing it as the solution to a differential equation. For any $i \in [n]$, $X_i(\cdot)$ is a random variable on 
\begin{equation}
D_X=\{s(\cdot): s(t)=t-\max\{\tau_i: \tau_i \leq t\} \text{ for some } 0 \leq \tau_1 < \dots < \tau_k \leq T\}
\end{equation}
and $F_n(\cdot)$ is a random variable on 
\begin{equation}
\mathcal{M}_X=\left\{F: F(t;x)=\frac{1}{k}\sum_{i=1}^k \mathbbm{1}\{s_i(t)\leq x\} \text{ for some } s_1, \dots, s_k \in D_X, \, k \in \mathbb{N} \right\}.
\end{equation}
Let $\widebar{\mathcal{M}}_X$ denote the closure of $\mathcal{M}_X$, and observe that $\widebar{\mathcal{M}}_X$ contains all possible paths of $F_n$ and their limits.
If $F \in \widebar{\mathcal{M}}_X$ and $u' = F(t; \bar F(t +{\rm d}t;u)-{\rm d}t)$ for $u \in [0,1]$, then the mapping in \eqref{eq:HD2} satisfies the differential equation: $g^{(F)}(0;\cdot,\cdot)=0$,
\begin{align}
\begin{split}
\label{eq:GFdef}
g^{(F)}(t+{\rm d}t; x,y) 
&= g^{(F)}(t; x', y') \\
&\qquad + {\rm d}t \, [1-g^{(F)}(t; x', y')] \, \lambda[t,\bar F(t;x'), \bar F(t;y'), F(t; \cdot), \tilde g^{(F)}(t; \cdot)] \\
&\qquad - {\rm d}t \, g^{(F)}(t; x', y') \, \mu[t,\bar F(t;x'), \bar F(t;y'), F(t; \cdot), \tilde g^{(F)}(t ; \cdot)],
\end{split}
\end{align}
if $\bar F(t + {\rm d}t; x) \wedge \bar F(t + {\rm d}t; y) \geq {\rm d}t$ and $0$ otherwise.

The differential equation in \eqref{eq:GFdef} can be understood as follows. Consider the process $(h^{G_n(t)})_{t \in [0,T]}$ under Assumption \ref{ass:VL}, so that the vertices are labelled in order of increasing age. Informally, given $F_n=F$, $g^{(F)}(t + {\rm d}t ;x,y)$ can be thought of as the probability that edge $(\lceil n x \rceil, \lceil n y \rceil)$ is active at time $t + {\rm d}t$. For $F \in \widebar{\mathcal{M}}_X$, under Assumption \ref{ass:VL}, at time $t$ these vertices had the labels $(\lceil n x' \rceil, \lceil n y' \rceil)$, respectively (with $x',y'$ given above \eqref{eq:GFdef}). The first term in the right-hand side of \eqref{eq:GFdef}, $g^{(F)}(t;x',y')$, is the probability that the edge was active at time $t$, the second term accounts for the event that the edge turned on during the time interval $[t, t+ {\rm d}t]$, while the third term accounts for the event that it turned off during the time interval $[t, t + {\rm d}t]$.

\begin{theorem}
\label{thm:CDalt}
The sequence of processes $\{ (\tilde h^{G_n(t)})_{t \geq 0}\}_{n \in \mathbb{N}}$ satisfies the LDP with rate $n$ and with rate function 
\begin{equation}
\label{eq:RFex}
J(\tilde h) = \inf_{F \in D(\mathcal{M}([0,1]),[0,t])\colon\, \tilde g^{(F)}=\tilde h} K(F),
\end{equation}
where $g^{(\cdot)}$ is defined by \eqref{eq:GFdef} and the function $K(\cdot)$ is given in Proposition \ref{lem:LDPDP}.
\end{theorem}

\begin{proposition}
\label{prop:CDLLN}
Suppose Assumption \ref{ass:VL} holds then $(h^{G_n(t)})_{t \geq 0} \Rightarrow g^{(F^*)}$ with $F^*$ characterised by 
\begin{equation}
F^*(x;t) =\begin{cases}
1 - \eee^{-\gamma x}, \quad &\text{if } x <t,\\
1 & \text{otherwise},
\end{cases}
\end{equation}
and $g^{(\cdot)}$ defined by \eqref{eq:GFdef}.
\end{proposition} 

\noindent
Theorem \ref{thm:CDalt} and Proposition \ref{prop:CDLLN} are interesting because they show that we can add a `mean-field type' interaction between the edges and still obtain equivalent results. 


\subsubsection{Numerical illustration}

\begin{figure}
\begin{center}
\includegraphics[width=4.96cm]{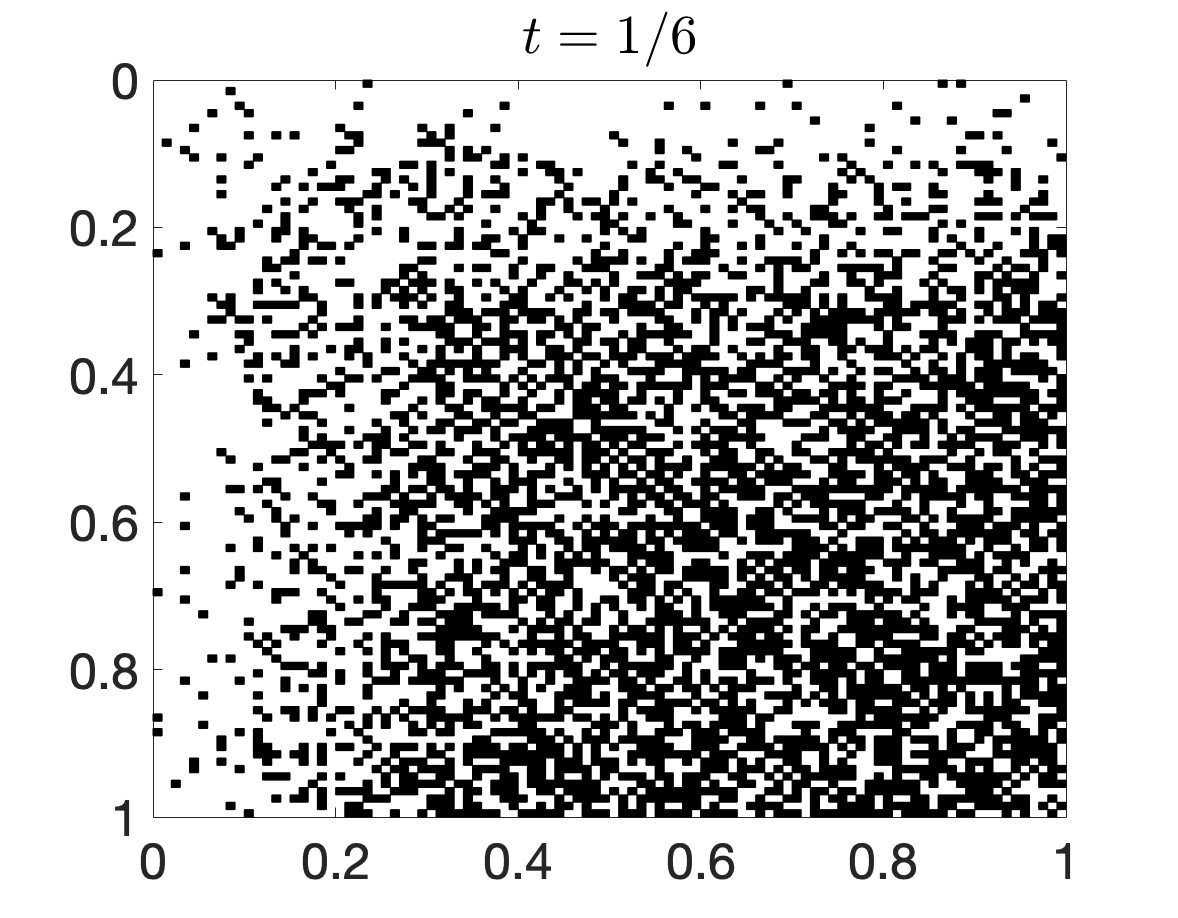}
\includegraphics[width=4.96cm]{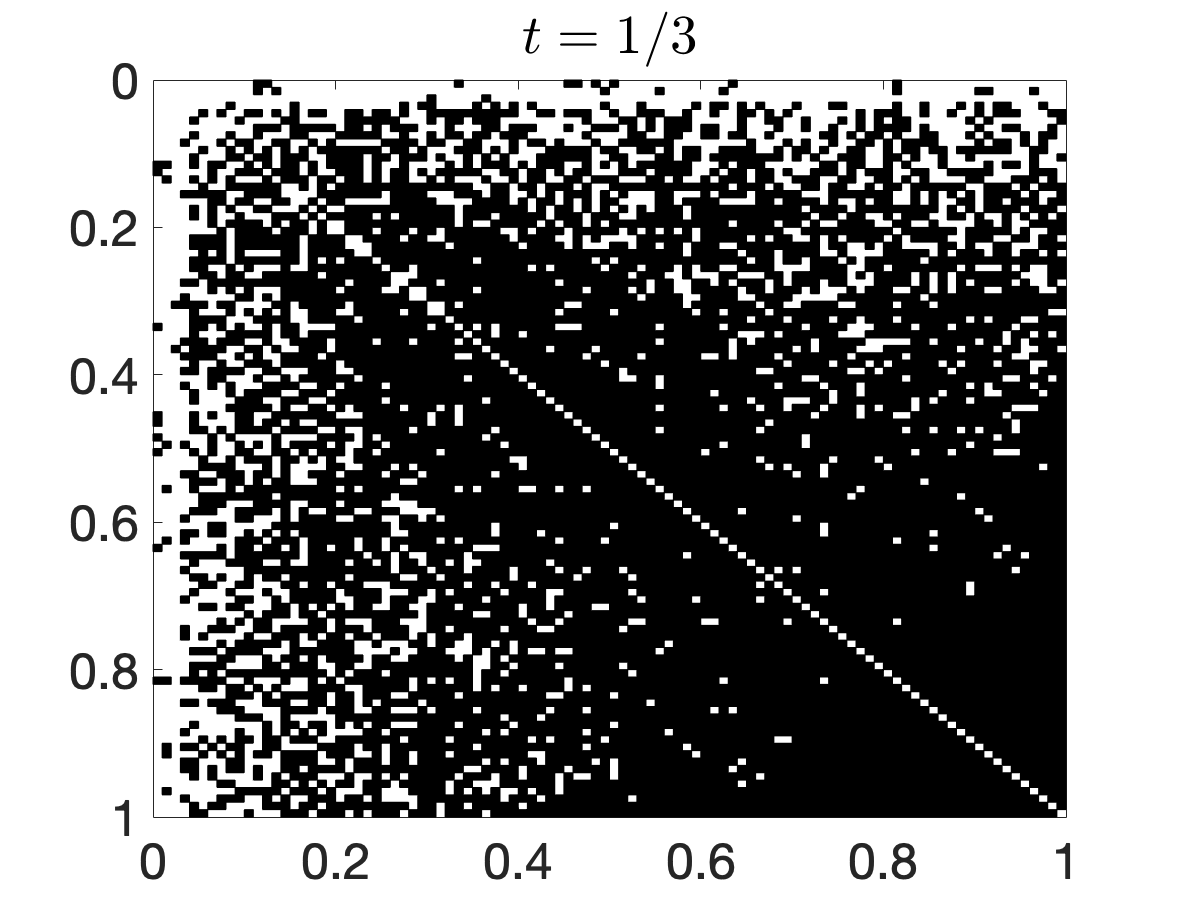}
\includegraphics[width=4.96cm]{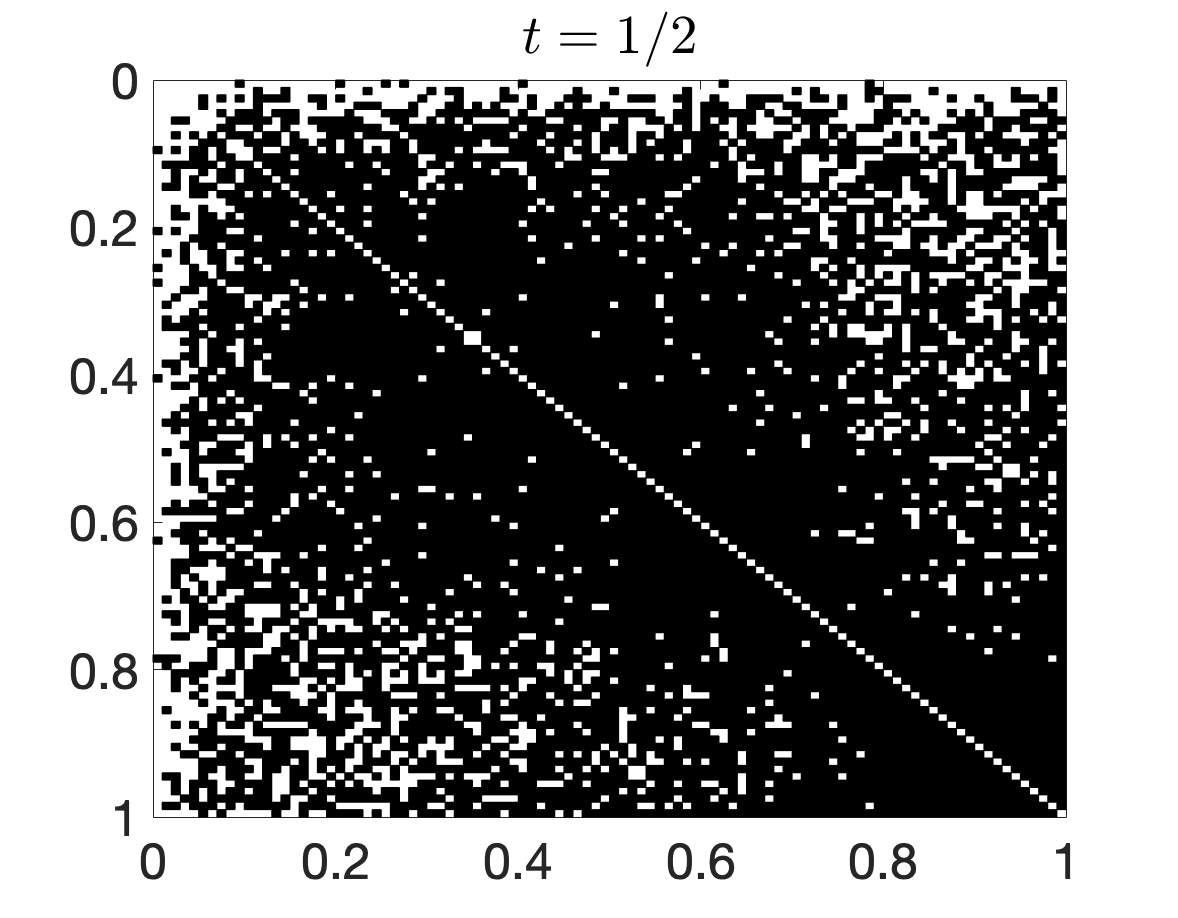}
\includegraphics[width=4.96cm]{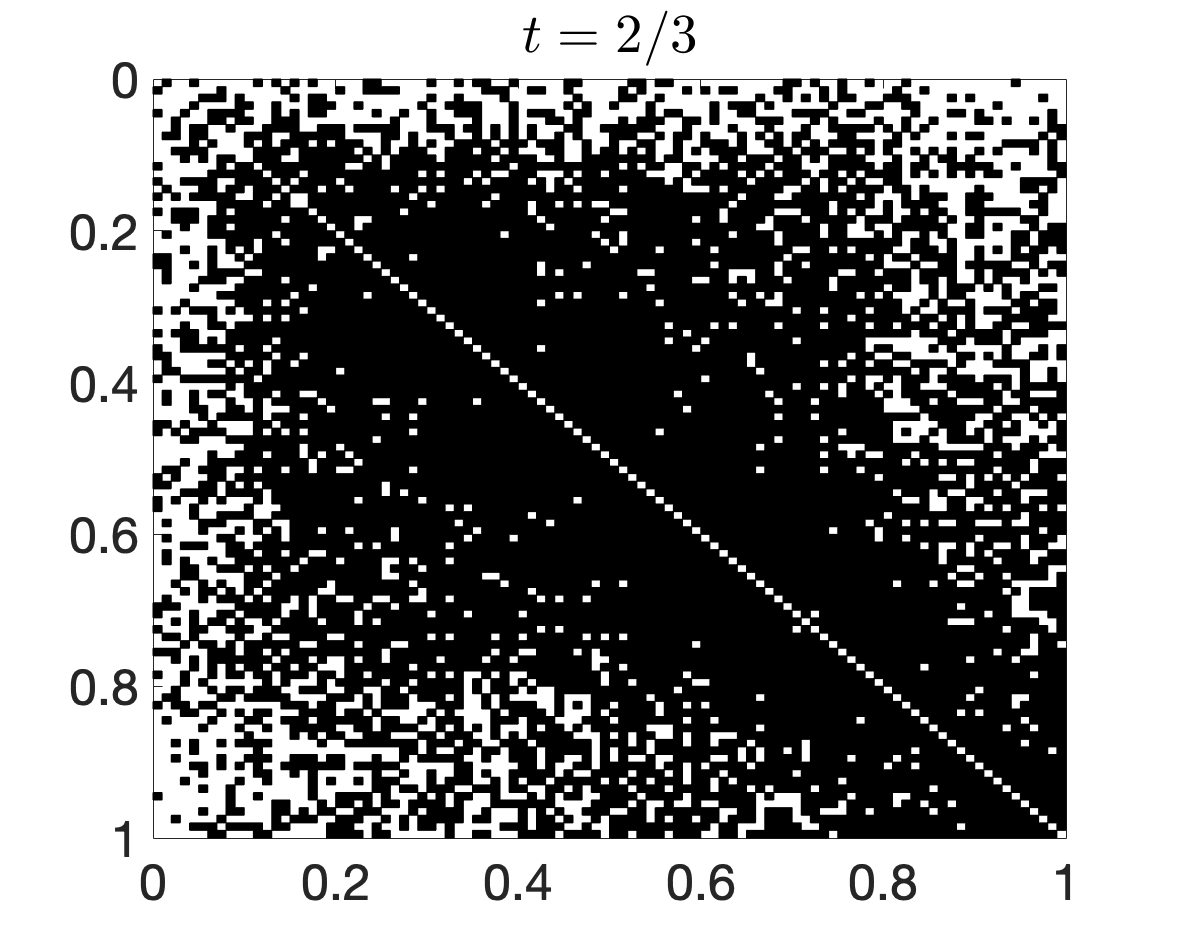}
\includegraphics[width=4.96cm]{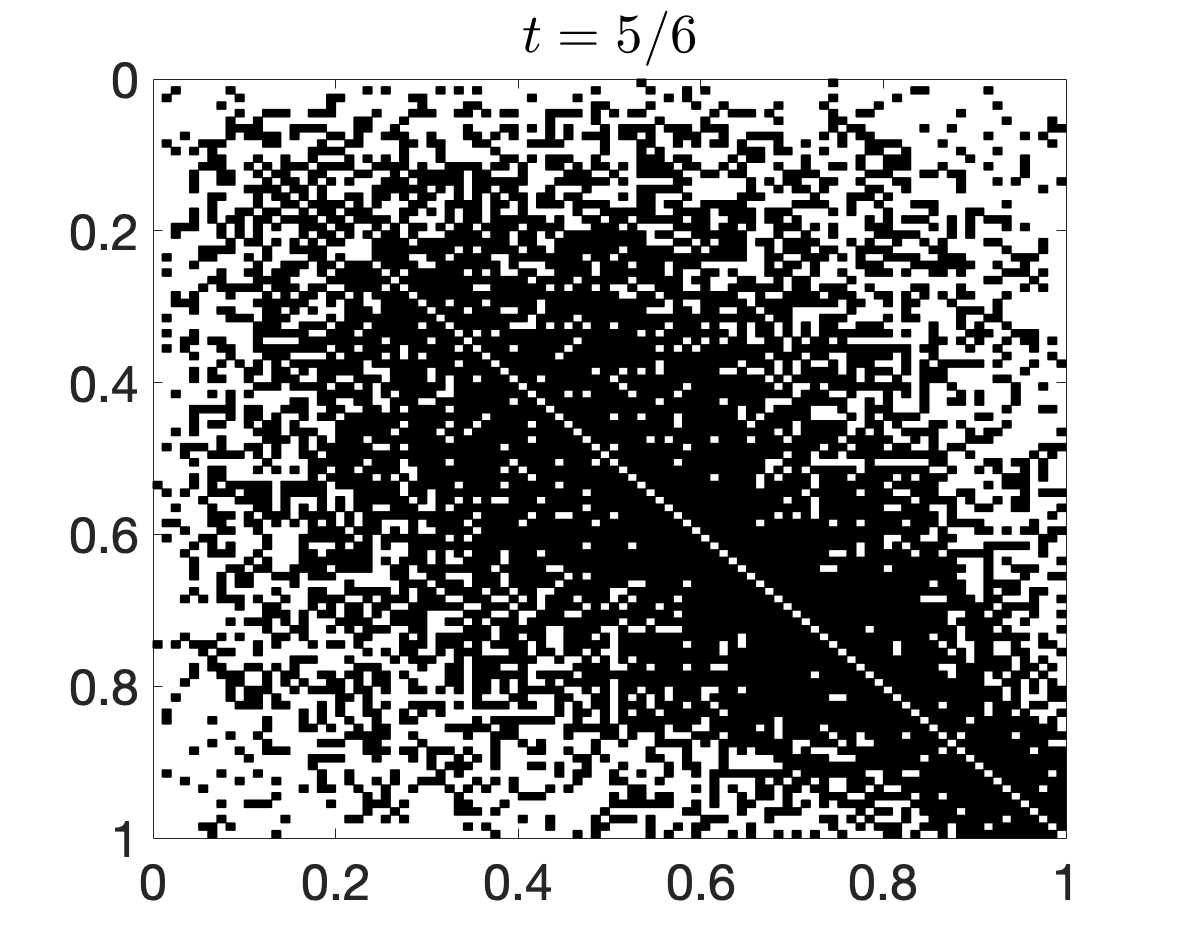}
\includegraphics[width=4.96cm]{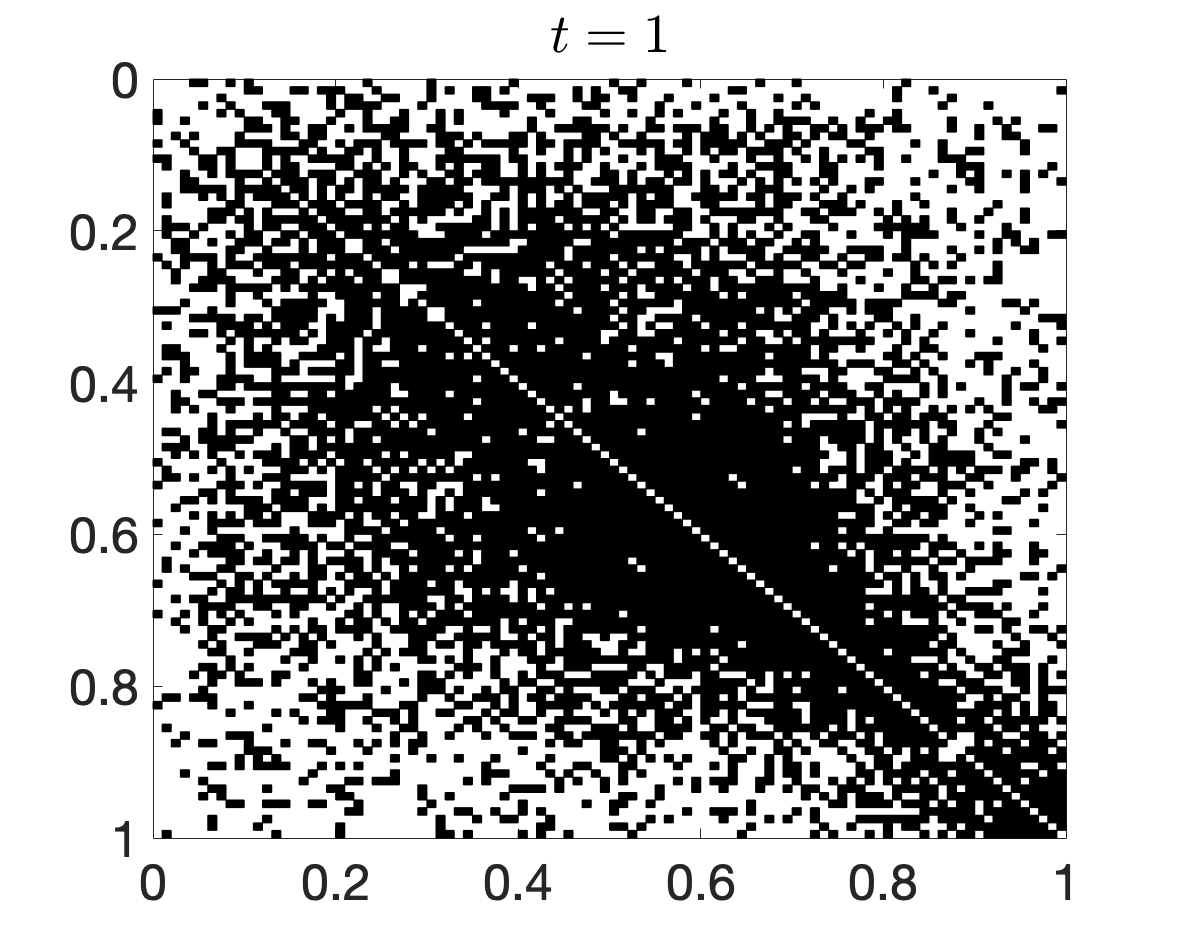}
\includegraphics[width=4.96cm]{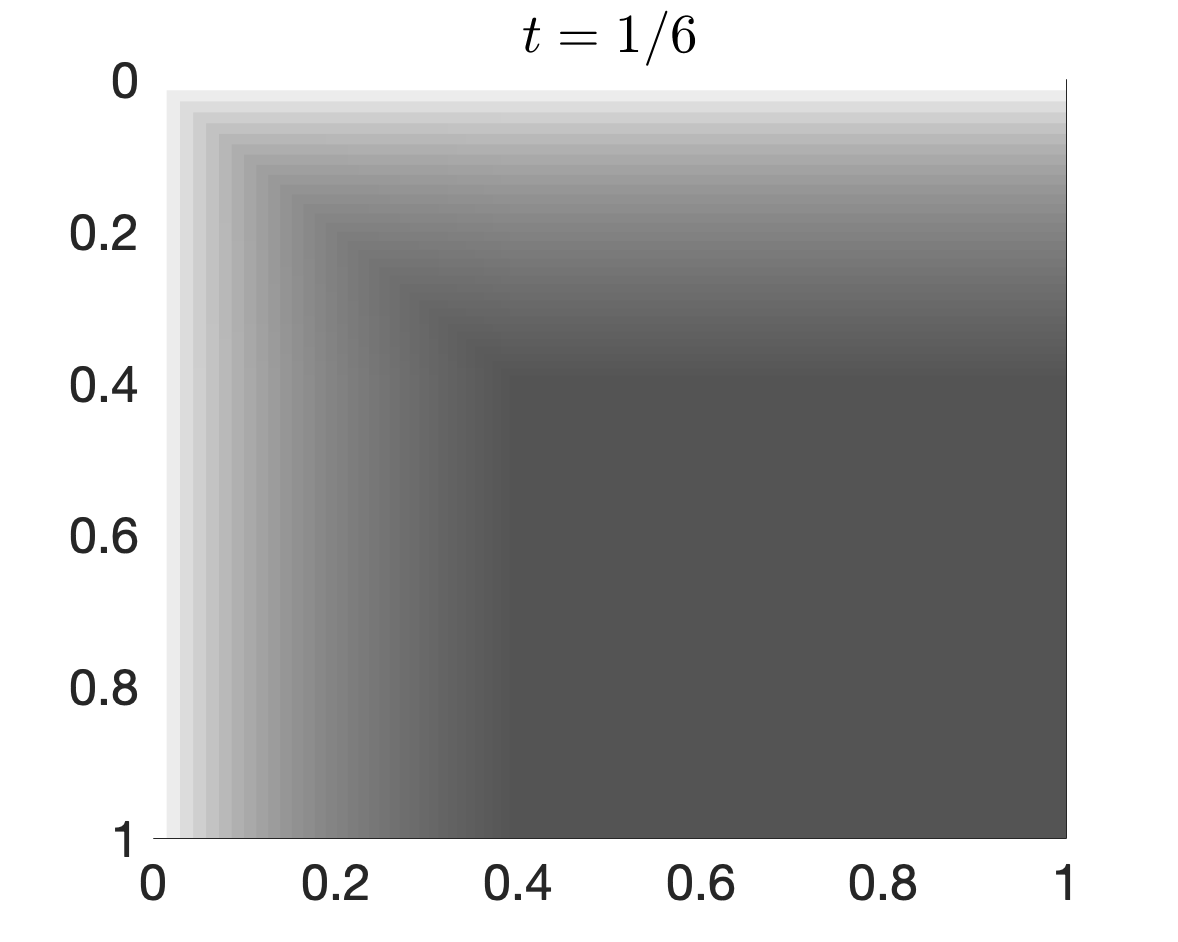}
\includegraphics[width=4.96cm]{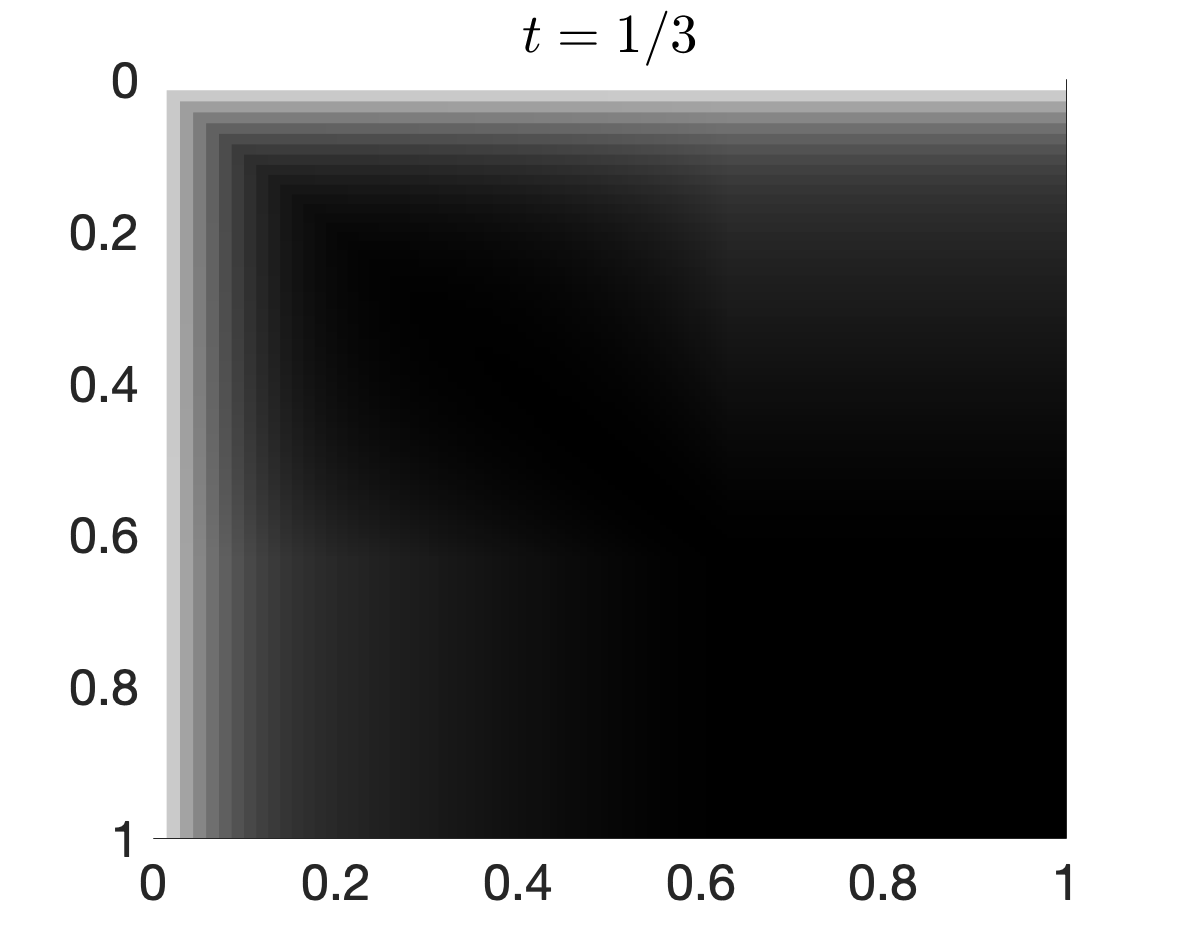}
\includegraphics[width=4.96cm]{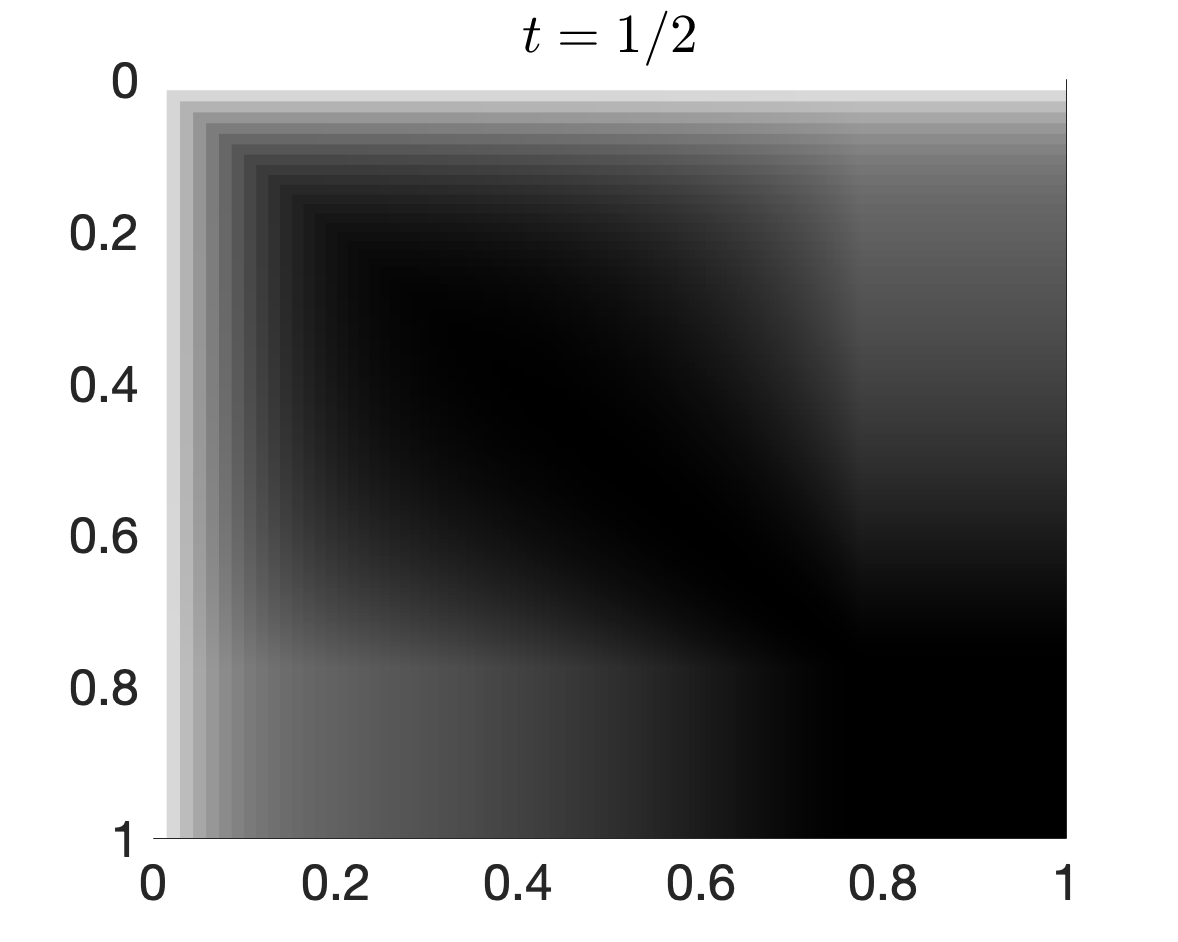}
\includegraphics[width=4.96cm]{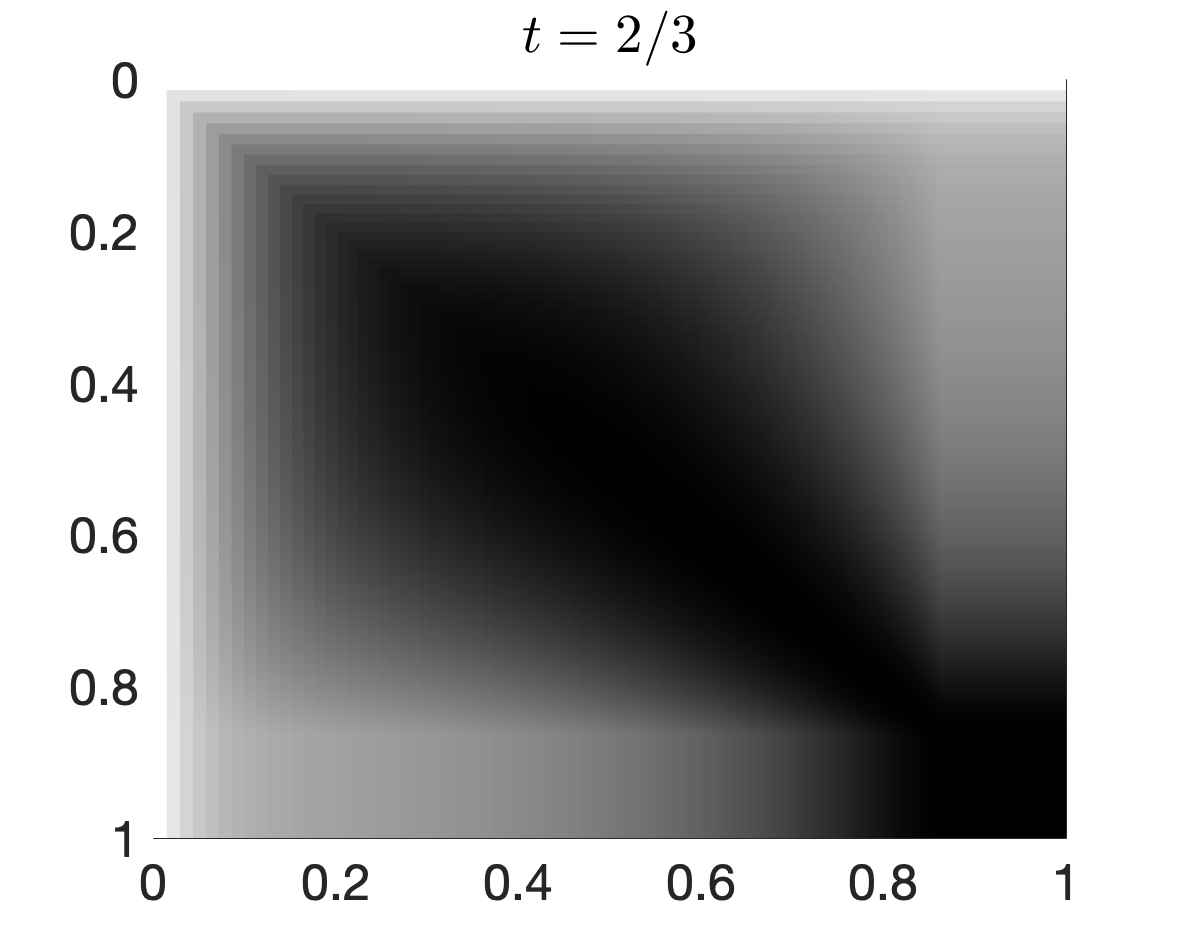}
\includegraphics[width=4.96cm]{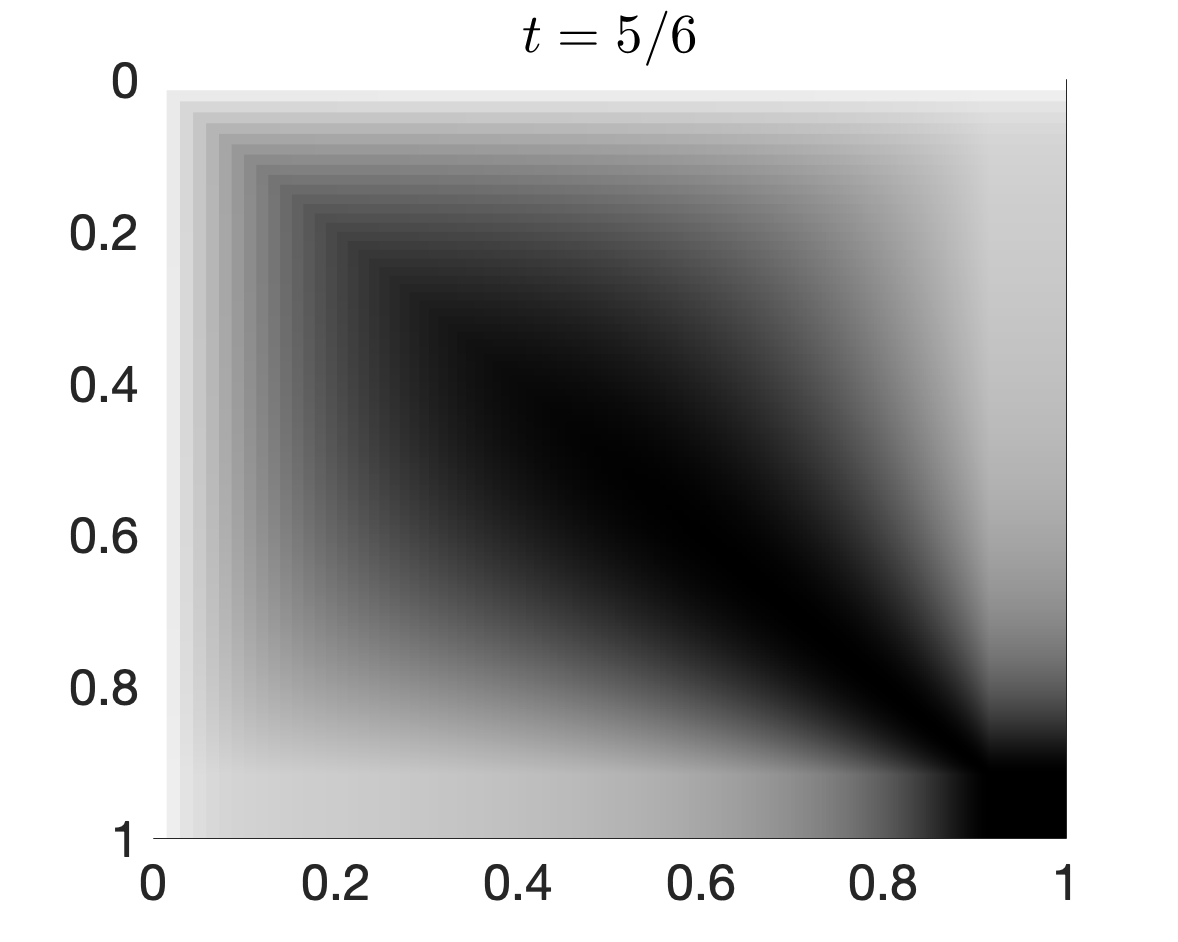}
\includegraphics[width=4.96cm]{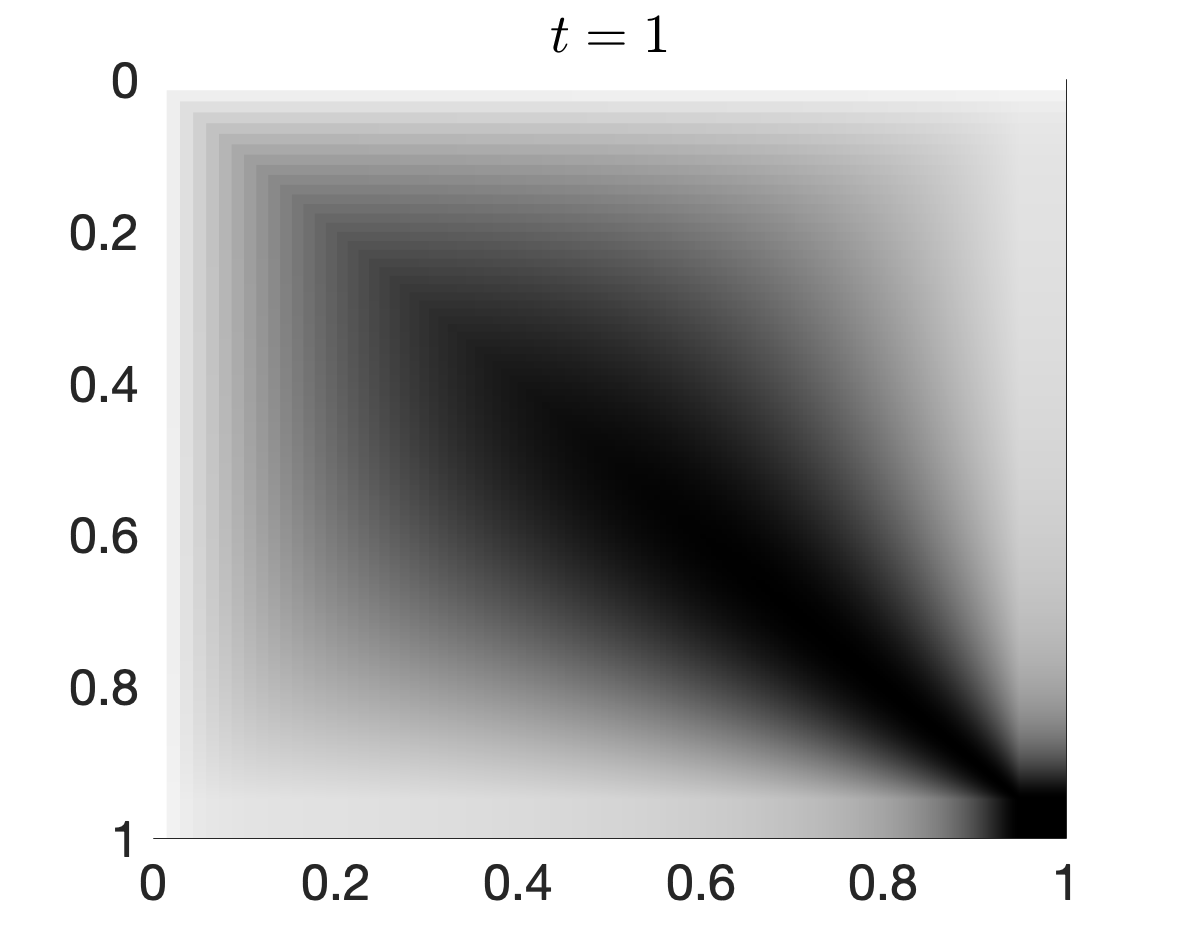}
\end{center}
\caption{\label{ExB} An illustration of $h^{G_{100}}(t)$ (top two rows) and $g^{(F^*)}(t)$ (bottom two rows) for $t=\frac16,\frac13,\frac12,\frac23,\frac56,1$.}
\end{figure}

To illustrate Theorem \ref{thm:CDalt} and Proposition \ref{prop:CDLLN} we suppose that
\begin{equation}
\label{eq:NER}
\gamma=3, \quad \lambda(t,u,v,F,\tilde h)= 4 + 60 \,s(\Delta, \tilde h), \quad \text{and} \quad \mu(t,u,v,F,\tilde h)= 120(u-v)^2,
\end{equation}
where $s(\Delta, \tilde h)$ denotes the triangle density in $\tilde h$. In Figure \ref{ExB} we plot, under Assumption \ref{ass:VL}, $h^{G_n}(\cdot)$ for $n=100$ and the corresponding fluid limit $g^{(F^*)}$ given in Proposition \ref{prop:CDLLN} for $t=\frac{1}{6},\frac{1}{3},\frac{1}{2},\frac{2}{3},\frac{5}{6},1$. 

The evolution of the graphon valued processes illustrated in Figure~\ref{ExB} can be understood by analysing the rates in Equation \eqref{eq:NER}. When $t$ is small (i.e., $t \leq \frac16$), the constant term `4' in $\lambda(\cdot)$ is much larger than the terms $120(u-v)^2$ and $60s(\Delta, \tilde h)$. This is because all vertices start with the same type $0$ (so that $120(X_i(t)-X_j(t))^2 = 0$ with high probability) and initially the graph contains no edges (so that $s(\Delta, \tilde h^{G_n(t)}) \approx 0$). Consequently, at $t=\frac16$, the graph is similar to what it would be if $\lambda(\cdot)=4$ and $\mu(\cdot)=0$, as in the illustrative example. As $t$ increases, both the triangle density and the squared difference in the ages $(X_i(t)-X_j(t))^2$ (on average) increase, which means that the terms $60 s(\Delta, \tilde g^{G_n(t)})$ and $120(X_i(t)-X_j(t))^2$ have an increasing influence on the dynamics of the process. Initially, the term $60 s(\Delta, \tilde g^{G_n(t)})$ grows faster (on average) than $120(X_i(t)-X_j(t))^2$ and, consequently, when $t=\frac13, \frac12, \frac23$, the process $g^{G_n(t)}$ has a high edge density. However, as $t$ increases further ($t=\frac56,1$) the squared differences in the ages $120(X_i(t)-X_j(t))^2$ (on average) become increasingly dominant and, consequently, pairs of vertices with significantly different ages are much less likely to be connected by an edge. Due to Assumption \ref{ass:VL} these vertices lie away from the diagonal connecting $(0,0)$ to $(1,1)$, and therefore we observe fewer edges in these regions of the plots.

\begin{figure}
\centering
\includegraphics[width=14cm]{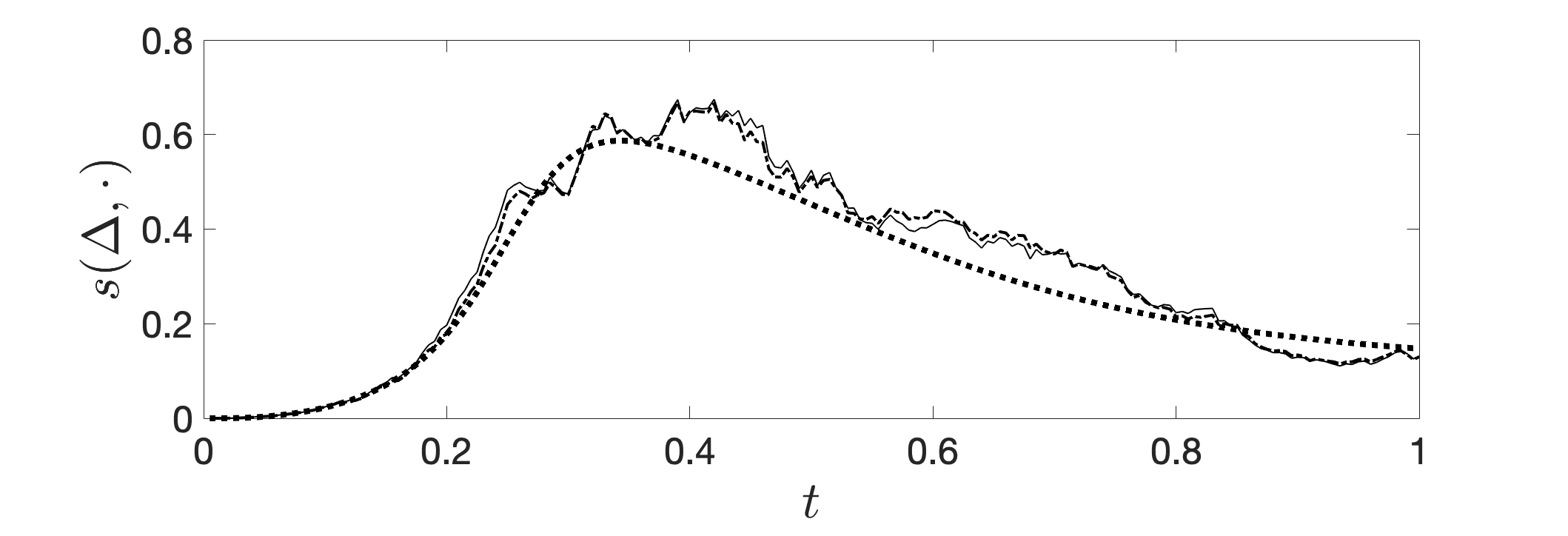}
\includegraphics[width=14cm]{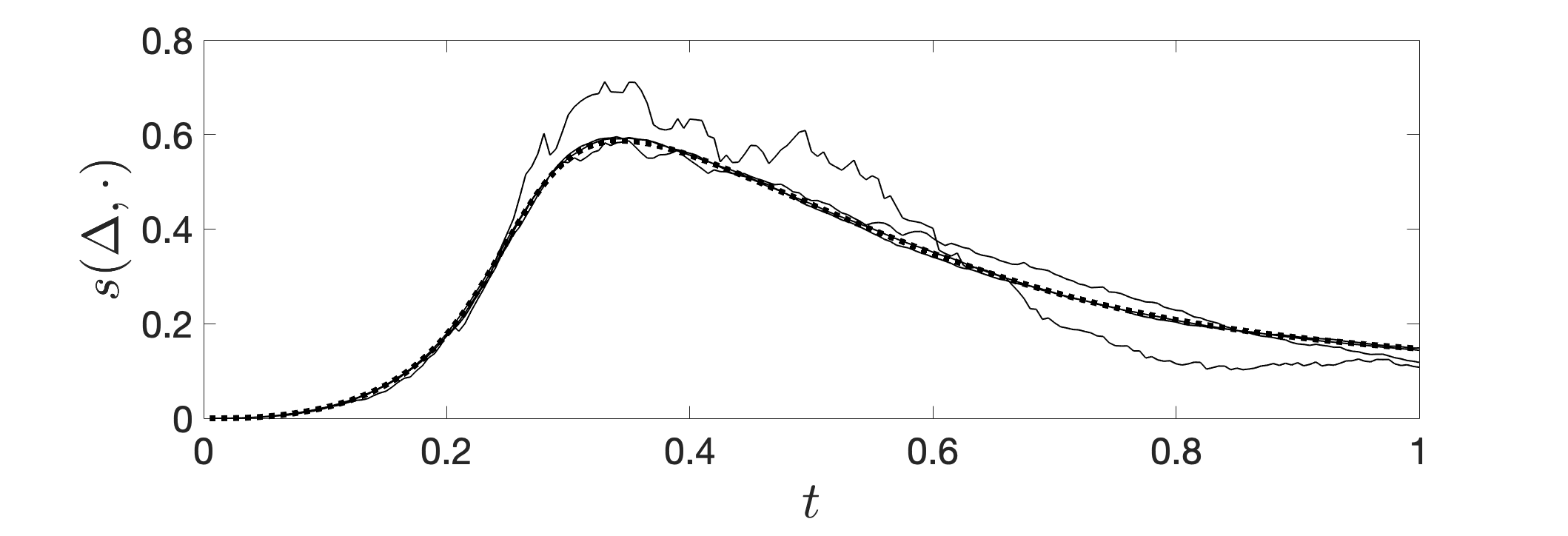}
\caption{Top panel: $s(\Delta,\tilde h^{(G_{100})}(t))$ (solid line), $s(\Delta,\tilde g^{(F_{100})}(t))$ (dashed line), and $s(\Delta,\tilde g^{(\bar F)}(t))$ (dotted line) for $t \in [0,1]$. Bottom panel: $s(\Delta,\tilde g^{(F^*)}(t))$ (dotted line) and $s(\Delta,\tilde g^{(F_{n})}(t))$ for $n=10^2,10^3,10^4,10^5$ (solid lines) and $t\in [0,1]$.}
\label{ExBTD}
\end{figure}

In the top panel of Figure \ref{ExBTD} we plot the triangle density of $h^{G_{100}}(t)$ (solid line), $g^{(F_{100})}(t)$ (dashed line), and $g^{(F^*)}$ (dotted line). To obtain the outcome of $g^{(F_{100})}(t)$, we take the \emph{same} outcome of the empirical type processes $F_{100}(\cdot)$ that was used to generate the outcome of $h^{G_{100}}(t)$ and insert it into the differential equation in \eqref{eq:GFdef}. We can therefore view the path of $h^{G_{100}}(t)$ as having randomness in both the types of the vertices and the outcomes of the edges, and view the path of $g^{(F_{100})}(t)$ as the same path but with the randomness in the outcomes of the edges removed. Because there are only $n=100$ vertices and ${n \choose 2}=4950$ edges, it is much more computationally efficient to simulate $g^{(F_{100})}(t)$ than it is to simulate $h^{G_{100}}(t)$. We note that in the top panel of Figure \ref{ExBTD} the solid and dashed curves follow each other closely. This is because, as Theorem \ref{thm:CDalt} suggests, fluctuations in the process are considerably more likely to be caused by fluctuations in the types of the vertices than in the specific outcomes of the edges given these types. This means for the purposes of simulation, when $n$ is sufficiently large ($n \geq 100$) it suffices to simulate $g^{(F_{n})}(t)$. In the bottom panel of Figure \ref{ExBTD} we simulate the triangle density of $g^{(F_n)}(t)$ for $n=10^2,10^3,10^4,10^5$ and observe convergence to the fluid limit.


\subsection{Different edge-switching dynamics, equivalent sample-path LDP}
\label{sec:ExDS1}

Next suppose that the process $\tilde h^{G_n(\cdot)}$ has the same underlying driving process (i.e., the same $\{(X_i(t))_{t \geq 0}\}_{i \in \mathbb{N}}$) as in Section \ref{Sec:APP1}, but different edge-switching dynamics. In particular, we consider edge switching dynamics inspired by \cite{AdHR19}. Let $(U_{ij})_{1\leq i < j \leq n}$ be a sequence of independent uniform variables on $[0,1]$, and suppose that edge $ij$ is active if 
\begin{equation}
U_{ij} \leq H(t; X_i(t), X_j(t), F_n(t;)),
\end{equation} 
where $H$ is given by \eqref{eq:FF} and \eqref{eq:HFF}. Observe that for any edge $ij$ the random number $U_{ij}$ is sampled just once, rendering this process not Markov. Nonetheless, by Proposition \ref{pr:point} we immediately have that, in the pointwise topology, the sequence of processes satisfies the same LDP as the processes described in Section \ref{sec:ExDS1}. To strengthen the topology it remains to verify establish exponential tightness.  This can be done by using Proposition \ref{lem:TC}, which leads to the following result.

\begin{proposition}
\label{prop:DED}
The sequence of processes $\{ (\tilde h^{G_n(t)})_{t \geq 0} \}_{n \in \mathbb{N}}$ described in Section \ref{sec:ExDS1} satisfies the LDP with rate $n$ and with rate function \eqref{eq:RFex}.
\end{proposition}


\subsection{The most likely path to an unusually small edge density}
\label{sec:ExDS3}

Consider the illustrative example described in Section \ref{sec:example} with $\lambda = \gamma$. Observe that in this case \eqref{eq:IEH} simplifies to 
\begin{equation}
\label{eq:43S}
H(u,v,F) = u \wedge v.
\end{equation} 
Suppose that we would like to determine the most likely path the process takes to a prescribed edge density $e^*$ at time $T$. The standard method involves two steps: first compute the most likely state of the process at time $T$ given this edge density, and afterwards use this computation to obtain the most likely trajectory of the process. Below we focus only on the first step and demonstrate that it is far simpler to numerically compute the most likely state of the process at time $T$ when $e^*$ is below the expected edge density than when it is above the expected edge density. Note that, given \eqref{eq:43S}, the results in this section apply to the models introduced in Sections \ref{Sec:APP1} and \ref{sec:ExDS1} (with the above simplification).


\subsubsection{Most likely state of the process at time $T$}
\label{Sec:ASG}

Let $Q$ denote the distribution of $X_v(t)$, where $X_v(t)$ is defined in \eqref{eq:XvD}. Note that 
\begin{equation}
\label{eq:Qex}
Q({\rm d}x) = \begin{cases}
{\rm d}x, \qquad & \text{if }  x < F^{{\rm exp}}(T), \\
1-F^{{\rm exp}}(T), \qquad &\text{if }  x =F^{{\rm exp}}(T), \\
0, \qquad &\text{otherwise}.
\end{cases}
\end{equation} 
However, for the moment we will assume that $Q$ is a general measure on $[0,1]$. We first consider the event that $G_n$ has an unusually small edge density $e^*$ at time $T$. By Theorem~\ref{thm:LDPTP}, the corresponding variational problem is 
\begin{equation}
\label{eq:OP}
\begin{array}{lll}
&{\rm minimize}\qquad  &\,\int_{0}^1 \log \left( \frac{{\rm d} P}{{\rm d}Q}  \right) {\rm d} P \\[0.2cm]
&{\rm subject\, to} \qquad &\,2 \int_0^1 \int_0^y x\, P({\rm d}x)\, P({\rm d} y) \leq e^* \\[0.2cm]
&{\rm over} &\,P \in \mathcal{M}([0,1]).
\end{array}
\end{equation}

\begin{proposition}
\label{prop:convex}
The feasible region of the variational problem described in \eqref{eq:OP} is convex.
\end{proposition}

\noindent
Because the objective function in \eqref{eq:OP} is strictly convex, Proposition \ref{prop:convex} implies that the variational problem has a unique local minimum which is the \emph{global} minumium. Consequently, there are several numerical methods that we can apply to find the global minimum.

If instead we consider the probability that $G_n(T)$ has an unusually high edge density, then we must solve the same variational problem as described in \eqref{eq:OP} with `$\leq$' replaced by `$\geq$'. Now the feasible region is no longer convex. Consequently, as we illustrate with a numerical example, the corresponding variational problem may have multiple local maxima and multiple local minima.

Let 
\begin{equation}
\label{eq:NumQ}
Q(x)=
\begin{cases}
\frac{4}{5}-\frac{1}{1000}, \quad &\text{if } x=0, \\
\frac{1}{5}, \quad &\text{if } x=\frac{1}{10}, \\
\frac{1}{1000}, \quad &\text{if } x=1.
\end{cases}
\end{equation}
In Figure \ref{Edgedensity1} we plot the rate (the value of the objective function evaluated at the maxima) against the edge density $e^*$. When $e^* \approx 0.085$ there are two distinct optimal solutions corresponding to 
\[
P^*_1(x) = \begin{cases}
0.0782, \quad &\text{if } x=0,\\
0.9159, \quad &\text{if } x=\frac{1}{10}, \\
0.0059, \quad &\text{if } x=1,
\end{cases} 
\qquad \qquad  
P^*_2(x) = \begin{cases}
0.3728, \quad &\text{if } x=0,\\
0.4020, \quad &\text{if } x=\frac{1}{10}, \\
0.2252, \quad &\text{if } x=1.
\end{cases} 
\]
For $e^*\approx 0.085$, solutions near $P^*_1$ and $P^*_2$ are local minima. These local minima are illustrated by the dotted curve in Figure \ref{Edgedensity1}: values above $0.085$ correspond to solutions that are close to $P^*_1$ and values below $0.085$ correspond to solutions that are close to $P^*_2$. Observe that we can restrict our search of an optimal measure $P$ to measures that are absolutely continuous with respect to $Q$. For $Q$ given by \eqref{eq:NumQ}, these measures live on the 2-dimensional simplex. Consequently, there can be at most two local minima. If $Q$ has a continuous component (as it does in \eqref{eq:Qex}), then there is, in principle, no bound on the number of local minima. This makes it difficult to determine if a global minimum has been reached using numerical methods.

\begin{figure}
\begin{center}
\includegraphics[width=11cm]{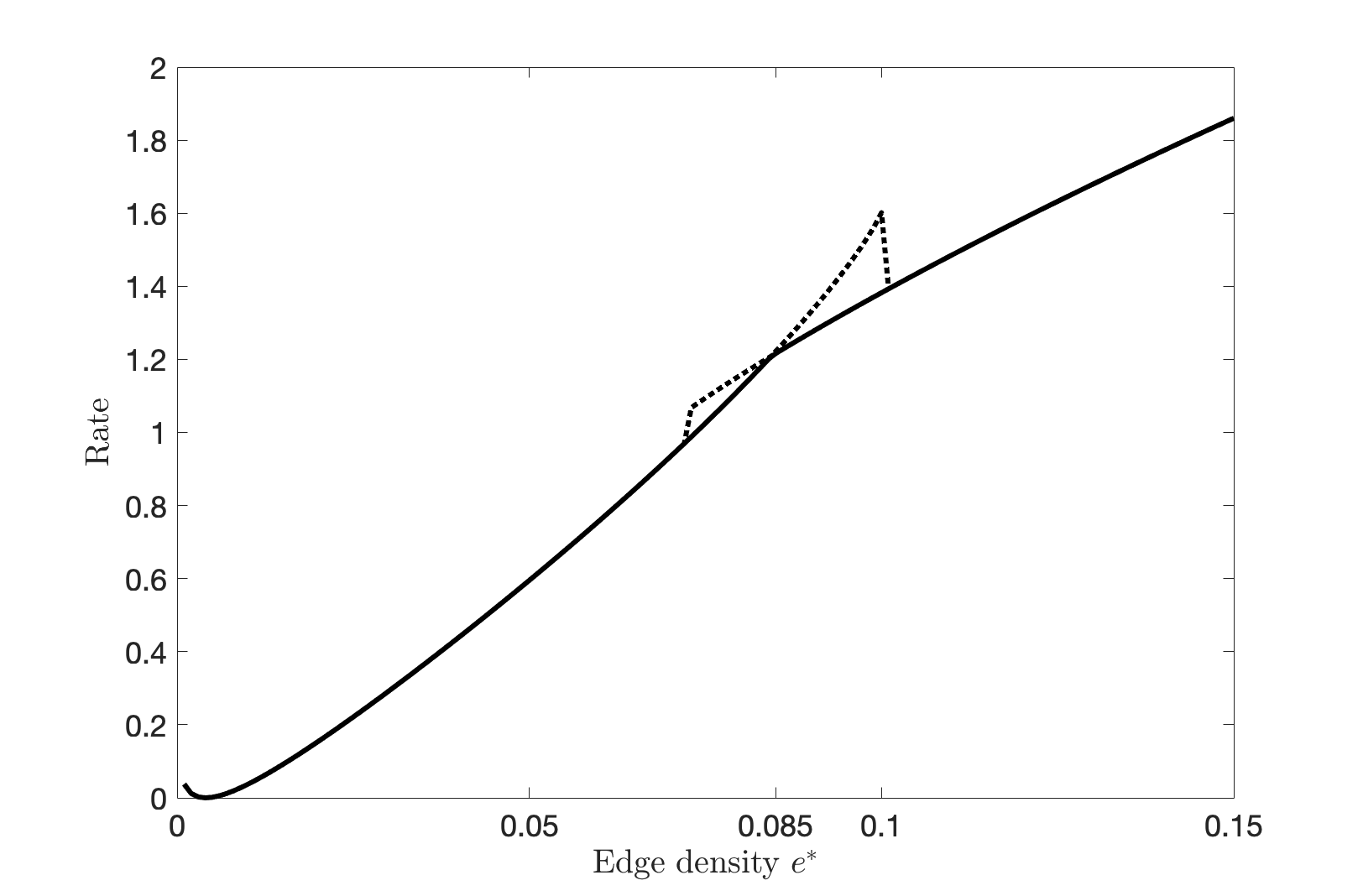}
\end{center}
\caption{\label{Edgedensity1} The rates of the local minima for different prescribed edge densities $e^*$ when $Q$ is given by \eqref{eq:NumQ}.}
\end{figure}


\section{Proofs}
\label{sec:proof}

Sections \ref{sec:IRGpr}--\ref{sec:ExSDpr} contain the proofs of the theorems in Sections \ref{sec:IRG}--\ref{sec:ExSD}.


\subsection{Proofs of the results in Section \ref{sec:IRG}}
\label{sec:IRGpr}

We prove Theorem \ref{thm:LDPTP} via a sequence of lemmas. Recall that we use $\hat G_n$ to denote an inhomogeneous Erd\H{o}s--R\'enyi random graph (IRG) and $G_n$ to denote a inhomogeneous random graph with type dependence (IRGT). 


\subsubsection{Inhomogeneous Erd\H{o}s--R\'enyi random graphs}

The first lemma is similar to results presented in \cite{DS19} and to \cite[Theorem 4.1]{DM20}, the primary difference being that in \cite{DS19,DM20} there is a single reference graphon $r$, i.e., $r_n = r$ for all $n \geq 0$. The generalisation comes at the cost of the addition of Assumption \ref{ass:AB} in the lower bound (which is not made in \cite[Theorem 4.1]{DM20}, but is in \cite{DS19}).

\begin{lemma}
\label{lem:gIHR}
Let $r_n$ denote the reference graphon for $\hat G_n$ and suppose that $r_n \to r$ in $L^1$. Then
\begin{equation}
\limsup_{n \to \infty} \frac{1}{{n \choose 2}} \log \mathbb{P}(\tilde h^{\hat G_n} \in \mathcal{C}) 
\leq - \inf_{\tilde h \in \mathcal{C}} \tilde I_{r}(\tilde h), \quad \forall \, \mathcal{C} \text{ closed},
\end{equation}
and, subject to Assumption \ref{ass:AB},
\begin{equation}
\liminf_{n \to \infty} \frac{1}{{n \choose 2}} \log \mathbb{P} (\tilde h^{\hat G_n} \in \mathcal{O}) 
\geq - \inf_{\tilde h \in \mathcal{O}} \tilde I_{r}(\tilde h), \quad \forall \, \mathcal{O} \text{ open}.
\end{equation}
\end{lemma}

\begin{proof}
The upper bound follows from the same arguments as used in \cite[Theorem 4.1]{DM20} (by noting that the specific requirement that $r_n=r$ is not used there). The lower bound again follows from similar arguments. However, now instead of applying Jensen's inequality we apply the dominated convergence theorem, which is possible because of Assumption \ref{ass:AB}. 
\end{proof}

The previous lemma can be used to obtain the following concentration type result for inhomogeneous Erd\H{o}s--R\'enyi random graphs.
For $r \in \mathscr{W}$, denote $\mathbb{B}_\square(\tilde r, \varepsilon)=\{\tilde h \in \mathscr{W}: \delta_\square(\tilde r, \tilde h) < \varepsilon \}$.

\begin{lemma}
\label{lem:EFa}
Let $\hat G_n$ be an IRG with reference graphon $r_n \in \mathscr{W}_n$. If $\lVert r_n - r \rVert_{L_1} \to 0$, then, for any $r \in \mathscr{W}$,
\begin{equation}
\limsup_{n \to \infty} \frac{1}{{n \choose 2}} \log \mathbb{P}(\tilde h^{\hat G_n} \notin \mathbb{B}_\square(\tilde r, \varepsilon)) 
\leq -\varepsilon^2.
\end{equation} 
\end{lemma}

\begin{proof}
Suppose that $\tilde h \notin \mathbb{B}_\square(r, \varepsilon)$, and let $h$ be any member of the equivalence class $\tilde h$. Using a Taylor expansion in the first inequality and Jensen's inequality in the second inequality, we have 
\begin{equation}
\begin{aligned}
I_r(h) &= \frac{1}{2} \int_{[0,1]^2} {\rm d}x\, {\rm d}y \left[ h(x,y) \log \left( \frac{h(x,y)}{r(x,y)} \right) 
+ (1-h(x,y)) \log \left( \frac{1-g_n(x,y)}{1-r(x,y)} \right) \right] \\
&\geq \int_{[0,1]^2} {\rm d}x\, {\rm d}y\,(h(x,y) - r(x,y))^2 \geq \lVert h - r \rVert_{L_1}^2
\geq d_\square(h, r)^2 \geq \varepsilon^2.
\end{aligned}
\end{equation}
Since $\mathbb{B}_\square(\tilde r, \varepsilon)$ is open, its complement is closed, which implies that we can apply the upper bound in Lemma \ref{lem:gIHR}, from which the result follows.
\end{proof}


\subsubsection{Inhomogeneous random graphs with type dependence}

We next turn our attention to inhomogeneous random graphs with type dependence $G_n$. We first use the previous lemma to show that $G_n$ is close to the induced reference graphon $g^{[F_n]}$ with high probability.

\begin{lemma}
\label{lem:EF}
If Assumption \ref{ass:RG} holds, then 
\begin{equation}
\label{eq:LemEF}
\limsup_{n \to \infty} \frac{1}{{n \choose 2}} \log 
\mathbb{P}( \tilde h^{G_n} \notin \tilde{\mathbb{B}}_\square(\tilde g^{[F_n]}, \varepsilon)) \leq -  \varepsilon^2.
\end{equation}
\end{lemma}

\begin{proof}
Suppose that Assumption \ref{ass:RG} holds. We proceed by contradiction. Suppose that \eqref{eq:LemEF} does not hold. Then there necessarily exist sequences $(n_k)_{k \in \mathbb{N}} \subseteq \mathbb{N}$ and $(F^\star_{n_k})_{k \in \mathbb{N}} \subset \mathcal{M}([0,1])$ such that 
\begin{equation}
\liminf_{k \to \infty} \frac{1}{{n_k \choose 2}} \log 
\mathbb{P} \left( \tilde h^{ G_{n_k}} \notin \tilde{\mathbb{B}}_\square(\tilde g^{[F^\star_{n_k}]}, \varepsilon) 
\mid F_{n_k}= F^\star_{n_k} \right) >- \varepsilon^2,
\end{equation}
where, for each $k\in \mathbb{N}$, $F^\star_{n_k}$ is an empirical distribution function with $n_k$ data points. Since $\mathcal{M}([0,T])$ is compact, there exists a convergent subsequence of $(F^\star_{n_k})_{k \in \mathbb{N}}$. Consequently, without loss of generality we may assume that there exists $F^\star$ such that $F^\star_{n_k} \to F^\star$ in $\mathcal{M}([0,T])$ as $k \to \infty$. Under Assumption \ref{ass:RG} we therefore have 
\begin{equation}
\lVert r^{[F^\star_{n_k}]} - r^{[F^\star]}\rVert_{L_1} \to 0, \qquad \text{as } k \to \infty.
\end{equation}
Recalling that, due to \eqref{ob:SS} (i.e., conditional on the induced reference graphon the graph has the distribution of an inhomogeneous random graph), we can apply Lemma \ref{lem:EFa} to obtain a contradiction.
\end{proof}

The next lemma establishes a large deviation principle for the sequence of induced reference graphons $g^{[F_n]}$.

\begin{lemma}
\label{Cor:Contr}
Subject to Assumptions \ref{ass:LDPtypes} and \ref{ass:RG}, $\{ \tilde g^{[F_n]} \}_{n \in \mathbb{N}}$ satisfies the LDP on $(\tilde{\mathscr{W}}, \delta_\square)$ with rate $\ell(n)$ and with rate function 
\begin{equation}
\label{RateLDPT}
J(\tilde h) = \inf_{F \in \mathcal{M}([0,1])\,:\, \tilde g^{[F]}=\tilde h} K(F). 
\end{equation}
\end{lemma}

\begin{proof}
For any $g,f \in \mathscr{W}$, $\lVert f-g \rVert_{L_1} \geq \delta_\square(\tilde g, \tilde f)$. Thus, under Assumption \ref{ass:RG}, the map $F \mapsto \tilde g^{[F]}$ is continuous. The result therefore follows by using Assumption \ref{ass:LDPtypes} and the contraction principle \cite[Theorem III.20]{dH00}.
\end{proof}

In the remainder of this subsection we make use of the above auxiliary results to prove Theorem \ref{thm:LDPTP}.

\smallskip\noindent
\textit{Proof of Theorem \ref{thm:LDPTP}:} We prove (i), (ii), and (iii) separately. Recall that these correspond to the three cases $\lim_{n \to \infty}\ell(n)/{n \choose 2}=0$, $\lim_{n \to \infty}\ell(n)/{n \choose 2}=c$, and $\lim_{n \to \infty}\ell(n)/{n \choose 2}=\infty$ respectively. 

\smallskip\noindent
\underline{(i) lower bound:} 
Let $\mathcal{O}$ be an open subset of $\tilde{\mathscr{W}}$, and let $\mathcal{O}^{(-\varepsilon)}$ denote the largest open set whose $\varepsilon$-neighbourhood is contained in $\mathcal{O}$. We have 
\begin{equation}
\label{eqn:LBi}
\mathbb{P}( \tilde h^{G_n} \in \mathcal{O}) \geq \mathbb{P}( \tilde g^{[F_n]} \in \mathcal{O}^{(-\varepsilon)}) 
\big(1- \mathbb{P}( \delta_\square( \tilde h^{G_n}, \tilde g^{[F_n]})>\varepsilon)\big).
\end{equation}
Applying Lemmas \ref{Cor:Contr} and \ref{lem:EF} to the first and second terms on the right-hand-side of \eqref{eqn:LBi}, respectively, we obtain
\begin{equation}
\limsup_{n \to \infty} \frac{1}{\ell(n)} \log \mathbb{P}( \tilde h^{G_n} \in \mathcal{O}) 
\geq - \lim_{\varepsilon \downarrow 0} \inf_{\tilde h \in \mathcal{O}^{(-\varepsilon)}} J(\tilde h) 
= - \inf_{\tilde h \in \mathcal{O}} J(\tilde h).
\end{equation}

\medskip\noindent
\underline{(i) upper bound:} 
Let $\mathcal{C}$ be a closed subset of $\tilde{\mathscr{W}}$, and let $\mathcal{C}^{(+\varepsilon)}$ denote the largest closed set that contains the $\varepsilon$-neighbourhoods of all the points in $\mathcal{C}$. We have 
\begin{equation}
\label{eqn:UBi}
\mathbb{P}( \tilde h^{G_n} \in \mathcal{C}) \leq \mathbb{P}( \tilde g^{[F_n]} \in \mathcal{C}^{(+\varepsilon)}) 
+ \mathbb{P}( \delta_\square( \tilde h^{G_n}, \tilde g^{[F_n]})>\varepsilon).
\end{equation}
Again applying Lemmas \ref{Cor:Contr} and \ref{lem:EF} to the first and second terms on the right-hand-side of \eqref{eqn:LBi}, respectively, we obtain
\begin{equation}
\limsup_{n \to \infty} \frac{1}{\ell(n)} \log \mathbb{P}( \tilde h^{G_n} \in \mathcal{C}) 
\leq - \lim_{\varepsilon \downarrow 0} \inf_{\tilde h \in \mathcal{C}^{(+\varepsilon)}} J(\tilde h) 
= - \inf_{\tilde h \in \mathcal{C}} J(\tilde h),
\end{equation}
where in the first step we use the fact that $\ell(n) = o({n \choose 2})$.

\medskip\noindent
\underline{(ii) lower bound:} 
For $r \in \mathscr{W}$, let
\begin{equation}
F(r,\varepsilon) = \{F \in \mathcal{M}([0,1]): \lVert g^{[F]} -r \rVert_{L_1}< \varepsilon) \},
\end{equation} 
and observe that, by Assumption \ref{ass:RG}, the set $F(r, \varepsilon)$ is measurable. Under Assumption \ref{ass:AB} we therefore have 
\begin{equation}
\begin{aligned}
&\liminf_{n \to \infty} \frac{1}{{n \choose 2}} \log \mathbb{P}(\tilde h^{ G_n} \in \mathcal{O})\\ 
&\geq \lim_{\varepsilon \downarrow 0} \liminf_{n \to \infty} \frac{1}{{n \choose 2}} \bigg[ \log \mathbb{P}( F_n \in F(r,\varepsilon) )
+ \log \mathbb{P}( \tilde h^{\hat G_n} \in \mathcal{O}| F_n \in F(r,\varepsilon))\bigg] \\
&\geq -\lim_{\varepsilon \downarrow 0} \left[ c K(F(r,\varepsilon)) 
+ \sup_{F \in F(r,\varepsilon)} \inf_{\tilde h \in \mathcal{O}} I_{r^{[F]}}(\tilde h)  \right] \\
&= - [c J(\tilde r) + \inf_{\tilde h \in \mathcal{O}} I_r(\tilde h)].
\end{aligned}
\end{equation}
In the second inequality we use the fact that, since $\mathcal{M}([0,T])$ is compact, for any sequence $(F_n)_{n \in \mathbb{N}}$ in $F(r,\varepsilon)$ there exists a convergent subsequence, which allows us to apply Lemma~\ref{lem:gIHR}. In the final step we use the lower semi-continuity of $K$ and the assumption that $F \mapsto r^{[F]}$ is a continuous mapping from $\mathcal{M}([0,1])$ to $(\mathscr{W}, L_1)$, in combination with the fact that, under Assumption \ref{ass:AB}, if $\lVert r_n - r \rVert_{L_1} \to 0$, then $I_{r_n}(\tilde h) \to I_r(\tilde h)$ uniformly over $h \in \mathscr{W}$ as $n \to \infty$ (see \cite[Lemma 2.3]{DS19}). Because these arguments hold for any $r \in \mathscr{W}$ we can take the sharpest lower bound:
\begin{equation}
\liminf_{n \to \infty} \frac{1}{{n \choose 2}} \log \mathbb{P}(\tilde h^{G_n} \in \mathcal{O}) 
\geq - \inf_{\tilde r \in \tilde{\mathscr{W}}} [c J(\tilde r) + \inf_{\tilde h \in \mathcal{O}} I_r(\tilde h)] 
=- \inf_{\tilde h \in \mathcal{O}} \big\{\inf_{\tilde r \in \tilde{\mathscr{W}}} [ c J(\tilde r) + I_r(\tilde h) ] \big\}.
\end{equation}

\medskip\noindent
\underline{(ii) upper bound:} 
Let $L(\cdot, \cdot)$ be the L\'evy metric, let $B_L(F,\varepsilon) = \{ H \in \mathcal{M}([0,1])\colon\, L(H, F) \leq \varepsilon \}$, and recall that $L(\cdot, \cdot)$ metrises the weak topology (see \cite[Theorem~D.8]{DZ98}). Since $\mathcal{M}([0,1])$ is a compact space, for any $\varepsilon >0$ we can construct a finite set $F[\varepsilon]$ with the property that for any $H \in \mathcal{M}([0,1])$ there exists $F \in F[\varepsilon]$ such that $L(F,H) \leq \varepsilon$. We therefore have 
\begin{equation}
\begin{aligned}
\limsup_{n \to \infty} 
&\frac{1}{{n \choose 2}} \log \mathbb{P} (\tilde h^{\hat G_n} \in \mathcal{C}) \\
&\leq \lim_{\varepsilon \downarrow 0} \limsup_{n \to \infty} \frac{1}{{n \choose 2}} \log \bigg[ \sum_{F \in F[\varepsilon]}  
\mathbb{P}(F_n \in B_L(F,\varepsilon) )\mathbb{P}(\tilde h^{\hat G_n} \in \mathcal{C}\,|\, F_n \in B_L(F,\varepsilon))\bigg] \\
&\leq -\lim_{\varepsilon \downarrow 0} \min_{F \in F[\varepsilon]} [ c K(B_L(F,\varepsilon)) 
+ \inf_{F^\star \in B_L(F,\varepsilon)} \inf_{\tilde h \in \mathcal{C}} I_{g^{[F^\star]}}(\tilde h)]. \\
&\leq - \lim_{\varepsilon\to 0} \min_{F \in \mathcal{M}([0,1])} [c K(B_L(F,\varepsilon)) 
+ \inf_{F^\star \in B_L(F,\varepsilon)} \inf_{\tilde h \in \mathcal{C}} I_{g^{[F^\star]}}(\tilde h)]\\
&= - \min_{F \in \mathcal{M}([0,1])}  [c K(F) + \inf_{\tilde h \in \mathcal{C}} \tilde I_r^{[F]}(\tilde h)] \\
&= \inf_{\tilde h \in \mathcal{C}} \big\{ \min_{F \in \mathcal{M}([0,T])} [ c K(F) + \tilde I_r^{[F]}(\tilde h)] \big\}.
\end{aligned}
\end{equation}
In the second step we apply Assumption \ref{ass:LDPtypes}, Lemma \ref{lem:gIHR} and Laplace's method, using a similar justification as in the lower bound. In the fourth step we use lower semi-continuity of $K$ (Assumption \ref{ass:LDPtypes}) and apply \cite[Lemma 2.3]{DS19}).

\medskip\noindent
\underline{(iii):} 
In this case we can apply similar (albeit simpler) arguments as in case (i).
\qed


\subsection{Proofs of the results in Section \ref{sec:GVP}}
\label{sec:GVPpr}


\subsubsection{Large deviations}

The next lemma will be used to prove Proposition \ref{pr:point}.

\begin{lemma}
\label{lem:CMTFlem}
Subject to Assumption \ref{ass:CP}, $\tilde H(\cdot; \cdot, \cdot, F_n)$ satisfies the LDP with rate $\ell(n)$ and with rate function 
\begin{equation}
\label{eq:Jdef}
J(\tilde h) = \inf_{F \in \mathcal{M} \times [0,T] :  \tilde H(\cdot; \cdot, \cdot, F) = \tilde h} K(F).
\end{equation}
\end{lemma}

\begin{proof}
The claim follows from the contraction principle (cf.\ Lemma \ref{Cor:Contr}).
\end{proof}

\medskip\noindent
\textit{Proof of Proposition \ref{pr:point}.} 
To establish a multi-point LDP, we can follow similar arguments as in the proof of Theorem \ref{thm:LDPTP}.(i). For example, to establish the lower bound, pick $0\leq t_1 \leq t_2 \leq \dots \leq t_k \leq T$,  let $\mathcal{O}_i$ be an open subset of $\tilde{\mathscr{W}}$, and let $\mathcal{O}^{(-\varepsilon)}_i$ be as in the proof of Theorem \ref{thm:LDPTP}. We have 
\begin{equation}
\begin{aligned}
&\liminf_{n \to \infty} \frac{1}{\ell(n)} \log \mathbb{P} \left( \tilde h^{G_n(t_i)} \in \mathcal{O}_i, \, \forall i =1,\dots,k \right) \\
&\geq \lim_{\varepsilon \downarrow 0}  \liminf_{n \to \infty} \frac{1}{\ell(n)} 
\log \bigg[  \mathbb{P} \left( \tilde g^{[F_n]}(t_i) \in \mathcal{O}^{(-\varepsilon)}_i, \, \forall i =1,\dots,k \right) \\
&\qquad \qquad \qquad \qquad \qquad
+ \bigg(1 - \sum_{i=1}^k \mathbb{P}(\delta_\square(\tilde{h}^{G_n(t_i)}, \tilde{g}^{[F_n]}(t_i)) > \varepsilon \bigg) \bigg]  \\
&\geq \inf_{\tilde h: \tilde h(t_i) \in \mathcal{O}_i, \, \forall i=1,\dots, k} J(\tilde h),  
\end{aligned}
\end{equation}
where $J(\tilde h)$ is given by \eqref{eq:Jdef}, and we use a similar justification as in the lower bound of Theorem \ref{thm:LDPTP}.(i) (now applying Lemma \ref{lem:CMTFlem} where we used to apply Lemma \ref{Cor:Contr}). The upper bound is again similar and is therefore omitted. With the multi-point LDP established, we can apply the Dawson-G\"{a}rtner  projective limit theorem \cite[Theorem 4.6.1]{DZ98} to establish an LDP in the pointwise topology, and \cite[Lemma 4.6.5]{DZ98} to obtain the specific form of the rate function in Proposition \ref{pr:point}. 
\qed

\medskip\noindent
\textit{Proof of Proposition \ref{lem:TC}.}
For $\delta>0$ and $T>0$, define the modulus of continuity in $D(\tilde{\mathscr{W}},[0,T])$ by
\begin{equation}
w'(\tilde h^{G_n},\delta,T) = \inf_{t_i} \max_i \sup_{s,t\in[t_{i},t_{i+1})} 
\delta_\square\left(\tilde h^{ G_n(s)}, \tilde h^{G_n(t)} \right),
\end{equation}
where the infimum is over $\{t_i\}$ satisfying 
\begin{equation}
0=t_0 <t_1 < \dots < t_{m-1} <T \leq t_m
\end{equation}
and $\min_{1 \leq i \leq n} (t_i-t_{i-1})\geq \delta$. By \cite[Theorem 4.1]{FK06} (in combination with the compactness of $\tilde{\mathscr{W}}$), the sequence of processes $\{(h^{ G_n(t)})_{t \in [0,T]}\}_{n \in \mathbb{N}}$ is exponentially tight if 
\begin{equation}
\label{eq:Tgoal}
\lim_{\delta \downarrow 0} \limsup_{n \to \infty} \frac{1}{\ell(n)} \log \mathbb{P} \left( w'(\tilde h^{G_n},\delta,T) > \varepsilon \right) = -\infty
\end{equation}
for all $\varepsilon >0$. Suppose that \eqref{eq:Tcon} holds, with $C_n(t,\delta)$ given by \eqref{cndef}. We will show that this entails that \eqref{eq:Tgoal} holds with $t_i= i \delta$ for $i \in \{0, \dots, \lceil {T}/{\delta} \rceil \}$. Indeed, observe that 
\begin{equation}
\sup_{s,t \in [t_i, t_{i+1})} \delta_{\square} \left(\tilde h^{G_n(s)}, \tilde h^{ G_n(t)} \right) 
\leq \sup_{s,t \in [t_i, t_{i+1})} \lVert  h^{ G_n(s)} -  h^{G_n(t)} \rVert_{L_1} \leq {n \choose 2}^{-1} C_n (t_i, \delta).
\end{equation}
Consequently, 
\begin{equation}
\label{eq:wdid}
\mathbb{P} \left( w'(\tilde h^{G_n}, \delta, T) > \varepsilon \right) \leq \lceil T/\delta \rceil  
\sup_{t \in [0,T]} \mathbb{P} \left(C_n (t,\delta) > \varepsilon {n \choose 2} \right). 
\end{equation}
Hence 
\begin{align}
\begin{split}
\lim_{\delta \downarrow 0} 
&\limsup_{n \to \infty} \frac{1}{\ell(n)} \log \mathbb{P} \left( w'(\tilde h^{G_n},\delta,T) > \varepsilon \right) \\
&\leq \lim_{\delta \downarrow 0} \limsup_{n \to \infty} \left[\frac{1}{\ell(n)} \log  \lceil T/\delta \rceil 
+ \frac{1}{\ell(n)} \log \left( \sup_{t \in [0,T]} \mathbb{P} \left(C_n (t,\delta) > \varepsilon {n \choose 2} \right) \right)\right]\\
&= -\infty,
\end{split}
\end{align}
where in the final step we apply \eqref{eq:Tcon} and use the fact that $\ell(n) \to \infty$.
\qed


\subsubsection{Weak convergence} 

The next lemma is needed in the proofs of Proposition \ref{prop:fdd} and Theorem \ref{thm:CD}.

\begin{lemma}
\label{lem:CM}
Subject to Assumptions \ref{ass:CP} and \ref{ass:CD}, $g^{[F_n]} \Rightarrow g^{[F]}$ on $D((\mathscr{W}, d_\square), [0,T])$.
\end{lemma}

\begin{proof}
By Assumptions \ref{ass:CP} and \ref{ass:CD}, we can apply the continuous mapping theorem to establish that
\begin{equation}
g^{[F_n]} \Rightarrow g^{[F]} \quad \text{on } D((\mathscr{W}, \lVert \cdot \rVert_{L_1}), [0,T]).
\end{equation}
Because $(\mathscr{W}, \lVert \cdot \rVert_{L_1})$ is a stronger topology than $(\mathscr{W}, d_{\square})$, this implies the claim.
\end{proof}

\medskip\noindent
\textit{Proof of Proposition \ref{prop:fdd}.} 
By Lemma \ref{lem:CM}, we have $g^{[F_n]} \stackrel{\rm fdd}{\Rightarrow} g^{[F]}$ on $(\mathscr{W}, d_\square)$. From \eqref{Ob:spIRG}, Assumption \ref{ass:VL}, and the uniform bound in \cite[Lemma 5.11]{C17} we know that, for any $t \in [0,T]$, 
\begin{equation}
\lim_{n \to \infty} \mathbb{P}(d_\square( h^{G_n(t)} ,g^{[F_n]}(t)) \leq  \varepsilon) = 1,
\end{equation}
which implies that 
\begin{equation}
\lim_{n \to \infty} \mathbb{P}( h^{G_n(t_i)} \in \mathbb{B}_\square(g^{[F_n]}(t_i), \varepsilon), \, \forall i) = 1.
\end{equation}
The claim therefore follows from \cite[Corollary 3.3]{EK09}.
\qed

\medskip\noindent
\emph{Proof of Theorem \ref{thm:CD}.}
By Lemma \ref{lem:CM} and \cite[Corollary 3.3]{EK09}, it suffices to prove that 
\begin{equation}
d(g^{[F_n]}(\cdot), h^{G_n(\cdot)}) \to 0, \qquad \text{with probability 1},
\end{equation}
where $d(\cdot, \cdot)$ is a metric that generates $D(\mathscr{W}, [0,T])$ (see \cite[Ch. 3.5]{EK09} for an expression for $d(\cdot, \cdot)$). Define $g_\delta^{[F_n]}$ such that, $g_\delta^{[F_n]}(0)=g^{[F_n]}(0)$ and
\begin{equation}
g^{[F_n]}_\delta(t) = g^{[F_n]}(\delta i), \qquad \text{for } t \in (\delta i, \delta(i+1)], \; \; i=1, \dots \lfloor T/\delta \rfloor.
\end{equation}
Because, for any $\delta>0$,
\begin{equation}
\begin{aligned}
d(g^{[F_n]}(\cdot), h^{G_n(\cdot)}) &\leq d\left(g^{[F_n]}(\cdot),g_\delta^{[F_n]}(\cdot)\right) 
+ d\left(g_\delta^{[F_n]}(\cdot), h^{G_n(\cdot)}\right)\\ 
&\leq d\left(g^{[F_n]}(\cdot),g_\delta^{[F_n]}(\cdot)\right) 
+ \sup_{t \in [0,T]} d_\square\left(g_\delta^{[F_n]}(t), h^{G_n(t)}\right) 
\end{aligned}
\end{equation}
it suffices to show that 
\begin{equation}
\label{eqn:NGl}
\lim_{\delta \downarrow 0} \lim_{n \to \infty} \left[d\left(g^{[F_n]}(\cdot),g_\delta^{[F_n]}(\cdot)\right) 
+ \sup_{t \in [0,T]} d_\square\left(g_\delta^{[F_n]}(t), h^{G_n(t)}\right)  \right] = 0
\end{equation}
with probability 1. We first deal with the term $\lim_{\delta \downarrow 0} \lim_{n \to \infty} \sup_{t \in [0,T]} d_\square(g_\delta^{[F_n]}(t), h^{G_n(t)})$. By the same arguments as in the proof of Proposition \ref{prop:fdd}, for any $\delta, \varepsilon>0$,
\begin{equation}
\label{eqn:CLLN}
\lim_{n \to \infty} \mathbb{P} \left(h^{G_n(\delta i)} \in \mathbb{B}_\square (g^{[F_n]}(\delta i), \varepsilon) 
\,\,\forall\,\ i = 0,1,\dots , \lfloor T/\delta \rfloor \right)  = 1.
\end{equation}
Thus,
\begin{equation}
\begin{aligned}
\lim_{\delta \downarrow 0} 
&\lim_{n \to \infty} \mathbb{P}\left(\sup_{t \in [0,T]} d_\square\left(g_\delta^{[F_n]}(t), h^{G_n(t)}\right) > 2 \varepsilon \right) \\
&\leq\lim_{\delta \downarrow 0} \lim_{n \to \infty} \sum^{T/\delta}_{i=1} 
\left\{ \mathbb{P}\left(d_\square\left(g_\delta^{[F_n]}(\delta i), h^{G_n(\delta i)}\right) > \varepsilon \right) 
+ \mathbb{P} \left( C_n(t, \delta) > \varepsilon {n \choose 2}\right) \right\} \\
&= 0,
\end{aligned}
\end{equation}
where in the final step we apply \eqref{eqn:CLLN} in combination with Assumption \ref{ass:FLT}. The fact that, with probability 1,
\[
\lim_{\delta \downarrow 0} \lim_{n \to \infty}d\left(g^{[F_n]}(\cdot),g_\delta^{[F_n]}(\cdot)\right)=0
\]
follows because, due to Assumptions \ref{ass:CD} and \ref{ass:CP}, the limiting object $g^{[F]}$ of $g^{[F_n]}$ is a random variable on $D(\mathscr{W}, [0,T])$ (i.e., its trajectories are c\`adl\`ag paths). 
\qed


\subsection{Proofs of the results in Section \ref{sec:ExSD}}
\label{sec:ExSDpr}


\subsubsection{Proofs of the results in Section \ref{Sec:APP1}}

To prove Theorem \ref{thm:CD}, we construct a graphon-valued process, $(\tilde h^{ G^*_n(t)})_{t \geq 0}$, that mimics the behaviour of $(\tilde h^{G_n(t)})_{t \geq 0}$ while still falling into the framework of Section \ref{sec:GVP}. We couple the two processes and demonstrate that, under the coupling, the probability that the two processes deviate from each other significantly (in a way defined below) is $o(\eee^{-n(1+o(1))})$.

\smallskip\noindent
\textit{Constructing a mimicking process:}
Suppose that the process $(G^*_n(t))_{t \geq 0}$ is characterised by the following dynamics:
\begin{itemize}
\item $G^*_n(0)$ is the empty graph.
\item Each vertex $v$ is assigned an independent rate-$\gamma$ Poisson clock. Each time the clock rings, all the edges that are adjacent to $v$ become inactive.
\item 
If edge $ij$ is inactive, then it becomes active at rate 
$$
\lambda\left(t,Y_i(t),Y_j(t),F_n(t; \cdot), \tilde g^{[F_n]}(t; \cdot, \cdot)\right).
$$
\item 
If edge $ij$ is active, then it becomes inactive at rate 
$$
\mu\left(t,Y_i(t),Y_j(t),F_n(t; \cdot), \tilde g^{[F_n]}(t; \cdot, \cdot)\right).
$$
\end{itemize}
Here, $g^{[\cdot]}(t; \cdot, \cdot)$ is defined in \eqref{eq:GFdef}. We point out that the induced reference graphon process of $(G^*(t))_{t \geq 0}$ is indeed $g^{[F]}$. Note that the only difference between the transition rates of $(G_n(t))_{t \geq 0}$ and $(G^*_n(t))_{t \geq 0}$ is that in the transition rate functions $\lambda(\cdot)$ and $\mu(\cdot)$ we have replaced $\tilde h^{G_n(t)}$ by $\tilde g^{[F_n]}(t; \cdot, \cdot))$.

Theorem \ref{thm:CD} follows by verifying that we can apply Theorem \ref{thm:LDPmain} to $\{(G^*_n(t))_{t \geq 0}\}_{n \in \mathbb{N}}$, and using the following lemma.

\begin{lemma}
\label{lem:couple}
There exists a coupling of $\{(G_n(t))_{t \geq 0}\}_{n \in \mathbb{N}}$ and $\{(G^*_n(t))_{t \geq 0}\}_{n \in \mathbb{N}}$ such that
\begin{equation}
\lim_{n \to \infty} \frac{1}{n} \log \mathbb{P} \left(\lVert\tilde h^{G_n(t)} 
- \tilde h^{ G^*_n(t)} \rVert_{L_1} > \eta, \, \text{for some } t\in [0,T] \right) = -\infty.
\end{equation}
\end{lemma}

\begin{proof} 
The claim is proved in three steps.

\medskip\noindent
\emph{Step 1: description of the coupling.}  
Let $C_{\rm max}$ be the maximal value that the rates $\lambda(\cdot)$ and $\mu(\cdot)$ can take, i.e.,
\begin{equation}
C_{\rm max} = \max_{t \in [0,T], u,v\in[0,1], F \in \mathcal{M}([0,1]), \tilde h \in \tilde{\mathscr{W}}} 
\lambda(t, u,v,F,\tilde h) \vee \mu(t, u,v,F,\tilde h),
\end{equation}
and observe that $C_{\rm max} < \infty$ because $\lambda(\cdot)$ and $\mu(\cdot)$ are Lipshitz continuous functions with a compact domain. Suppose that outcomes of $(G_n(t))_{t \geq 0}$ and $(G^*_n(t))_{t \geq 0}$ are generated in the following manner.
\begin{itemize}
\item 
For each $i \in [n]$, vertex $i$ is assigned the same (coupled) rate-$\gamma$ Poisson clock in both processes, so that if the clock associated with vertex $i$ rings in $(G_n(t))_{t \geq 0}$ at time $s$, then the clock associated with vertex $i$ also rings in $(G^*_n(t))_{t \geq 0}$ at time $s$ (and vice-versa). When the clock associated to vertex $i$ then all edges adjacent to vertex~$i$ become inactive (in both $(G_n(t))_{t \geq 0}$ and $(G^*_n(t))_{t \geq 0}$ simultaneously).
\item 
Assign each edge the same (coupled) Poisson rate-$C_{\rm max}$ clock in both processes. When the Poisson clock associated with edge $ij$ rings, generate an outcome $u$ of a Unif(0,1) distribution.

\begin{itemize}
\item if $u \leq \lambda(t,Y_i(t),Y_j(t),F_{n},\tilde h^{G_n(t)})/C_{\rm max}$, then
the edge $ij$ becomes active in $( G_n(t))_{t \geq 0}$.
\item if $u \leq \lambda(t,Y_i(t),Y_j(t),F_{n},\tilde g^{[F_n]}(t; \cdot, \cdot))/C_{\rm max}$, then the
edge $ij$ becomes active in $( G^*_n(t))_{t \geq 0}$. 
\end{itemize}
(If it was already active, then it remains active.)
\item 
Assign each edge a second (coupled) Poisson rate-$C_{\rm max}$ clock in both processes. When the Poisson clock associated with edge $ij$ rings, generate an outcome $u$ of a Unif(0,1) distribution.
\begin{itemize}
\item
if $u \leq \mu(t,Y_i(t),Y_j(t),F_{n},\tilde h^{G_n(t)})/C_{\rm max}$, then the
edge $ij$ becomes inactive in $( G_n(t))_{t \geq 0}$.
\item 
if $u \leq \mu(t,Y_i(t),Y_j(t),F_{n},\tilde g^{[F_n]}(t; \cdot, \cdot))/C_{\rm max}$, then the
edge $ij$ becomes inactive in $(G^*_n(t))_{t \geq 0}$.
\end{itemize}
(If it was already inactive, then it remains inactive.)
\end{itemize}

\medskip\noindent
\emph{Step 2: majorization of the $L_1$ distance.} 
Observe that if edge $ij$ is inactive in both models and the clock associated with edge $ij$ rings, then a \emph{difference} is formed (i.e., edge $ij$ is active in one process and inactive in the other) with probability 
\begin{equation}
\frac{1}{C_{\rm max}}\,\Big| \lambda\left(t,X_i(t),X_j(t),F_{n}(t; \cdot),\tilde h^{G_n(t)}\right) - \lambda\left(t,X_i(t),X_j(t),F_{n}(t; \cdot),\tilde g^{[F_n]}(t; \cdot, \cdot)\right)\Big|.
\end{equation}
At any time $t$, by the Lipshitz continuity of $\lambda(\cdot)$,
\begin{align}
\begin{split}\label{eqn:rBg}
&\left| \lambda\left(t,X_i(t),X_j(t),F_{n}(t; \cdot),\tilde h^{G_n(t)}\right) - \lambda\left(t,X_i(t),X_j(t),F_{n}(t; \cdot),\tilde g^{[F_n]}(t; \cdot, \cdot)\right)\right| \\
&\qquad \leq c \left[ \delta_\square(\tilde g^{(F_n)}(t; \cdot, \cdot), \tilde h^{G^*_n(t)}) 
+ \delta_\square( \tilde h^{G^*_n}(t),\tilde h^{G_n(t)}) \right] \\
&\qquad \leq c\left[ \delta_\square(\tilde g^{(F_n)}(t; \cdot, \cdot), \tilde h^{G^*_n(t)}) 
+  \left\lVert h^{G^*_n(t)} - h^{ G_n(t)} \right\rVert_{L_1} \right],
\end{split}
\end{align}
where $c$ is the Lipshitz constant. Observe that an equivalent bound holds when $\lambda(\cdot)$ is replaced by $\mu(\cdot)$.

Let $D_i(t)$ denote the number of differences between $G^*(t)$ and $G(t)$, so that \begin{equation}D_i(t) = \frac{n^2}{2} \left\lVert h^{G^*_n(t)}- h^{G_n(t)} \right\rVert_{L_1}.\end{equation}
Using a standard property of the superposition of Poisson processes and \eqref{eqn:rBg}, we see that, at any time $t$, $D_n(t)$ increases by 1 at a rate that is bounded above by 
\begin{equation}
{n^2} C_{\rm max}\,c\left[ \delta_\square(\tilde g^{(F_n)}(t; \cdot, \cdot), \tilde h^{G^*_n(t)}) +  \left\lVert h^{G^*_n(t)} - h^{ G_n(t)} \right\rVert_{L_1}\right].
\end{equation}
In addition, observe that if a clock associated to a vertex rings it can only decrease the number of differences because then, in both processes, all the edges adjacent to the vertex are inactive. Combining the above, we see that, for any $\beta>0$, $D_n(t)$ is stochastically dominated by the random quantity
\begin{equation}
Z_n^{(\beta)}(t) + n^2 \mathbbm{1}\left\{\exists s \leq t : \delta_\square\left(\tilde g^{(F_n)}(s; \cdot, \cdot), 
\tilde h^{G^*_n(s)}\right) > \beta\right\},
\end{equation}
where $(Z^{(\beta)}_n(t))_{t \geq 0}$ is a pure-birth process, with initial state $Z^{(\beta)}_n(0)=0$, characterized by the transition rate
\begin{equation}
\label{eq:dp}
C_{\rm max}c \left[ \beta n^2 + 2 i \right].
\end{equation}
from state $i$ to state $i+1$.
Consequently, for any $\beta>0$ we have
\begin{equation}
\label{eqn:rBB}
\begin{aligned}
\mathbb{P}& \left(\left\lVert\tilde h^{G_n(t)} 
- \tilde h^{ G^*_n(t)} \right\rVert_{L_1} > \eta, \, \text{for some } t\in [0,T] \right) \\
&\leq  \mathbb{P}(Z^{(\beta)}_n(T) \leq \eta n^2/2 ) + \mathbb{P} \left(\delta_\square\left(\tilde g^{(F_n)}(t; \cdot, \cdot), 
\tilde h^{G^*_n(t)}\right) > \beta, \, \text{for some } t\in [0,T] \right),
\end{aligned}
\end{equation}
where we have used the fact that $(Z^{(\beta)}_n(t))_{t \geq 0}$ is an increasing process.

\medskip\noindent
\emph{Step 3: bounding the dominating process.} 
So as to bound the second term on the right-hand-side of \eqref{eqn:rBB}, we observe that, for any $\beta>0$, 
\begin{equation}\label{eq:CPL}
\limsup_{n \to \infty} \frac{1}{n}\log \mathbb{P} \left(\delta_\square\left(\tilde g^{(F_n)}(t; \cdot, \cdot), 
\tilde h^{G^*_n(t)}\right) > \beta, \, \text{for some } t\in [0,T] \right) = - \infty.
\end{equation}
We can establish \eqref{eq:CPL} following similar arguments as in the proof of Theorem \ref{thm:CD}, and hence these arguments are omitted.

For ease of notation below we write $Z_n(t)$ to denote $Z_n^{(\beta)}(t)$. It suffices to show that, for any $\eta >0$, there exists $\beta$ sufficiently small so that 
\begin{equation}
\limsup_{n \to \infty} \frac{1}{n^2/2} \log \mathbb{P}\left( Z_n (T) > \eta n^2/2  \right) \geq - C(\beta,\eta) >0
\end{equation}
for some $C(\beta,\eta)>0$. Observe that the Markov chain $(Z_n(t))_{t \geq 0}$ described by the transition rates \eqref{eq:dp} is a continuous-time branching process with immigration. The initial population size is 0, immigrants arrive at rate $c\,C_{\rm max}\,\beta n^2$, while individuals in the population give birth at rate $2 c\,C_{\rm max}$ and die at rate $0$. Let $X(t)$ denote the number of descendants that are alive at time $T$ of an individual that immigrated to the population at time $t<T$. Let $C^*:=2c\,C_{\rm max}$. It was shown by Yule (cf.\ \cite[Chapter V.8]{H63}) that 
\begin{equation}
\mathbb{P}(X(t)=i) = \eee^{-C^*(T-t)} (1- \eee^{-C^*(T-t)})^{i -1}, \quad i \in \mathbb{N},
\end{equation}
i.e., $X(t) -1$ has a geometric distribution with success probability $\eee^{-C^*(T-t)}$. Note that (since the death rate of individuals is zero) $X_0 $ stochastically dominates $X_t$ for all $t \geq 0$. In addition, the total number of immigrants has a Poisson distribution with mean $C^* \beta Tn^2/2$. Thus, if $\{ X^{(k,\ell)}_0 \}_{k,\ell \in \mathbb{N}}$ are i.i.d.\ copies of $X_0$ and $Y = \sum_{k=1}^{n^2/2} Y^{(k)}$, where $Y^{(k)} \sim {\rm Poi}(C^* \beta T)$ are independent of everything else, then
\begin{equation}
Z_n(T) \stackrel{\rm st}{\leq} Z_n^*(T) := \sum_{i=1}^Y  X^{(i,1)}_0 \stackrel{\rm d}{=} 
\sum_{i=1}^{n^2/2} \sum_{k=1}^{Y^{(i)}} X_0^{(i,k)}.
\end{equation}
With
\begin{equation}
\begin{aligned}
\varphi(s) &:=\mathbb{E} \left( \eee^{s \sum_{k=1}^{Y^{(i)}} X_0^{(i,k)}} \right) 
= \mathbb{E}\left( \mathbb{E} \left(\eee^{s X^{(i,k)}_0} \right)^{Y^{(i)}}  \right)\\ 
&= \exp\left\{ C^* \beta T \left( \frac{\eee^{-C^*T+s}}{1-(1-\eee^{-C^*T})\eee^{s}}-1 \right)  \right\}
\end{aligned} 
\end{equation}
and $I(z)= \sup_{s \in \mathbb{R}} \left[ zs - \log \varphi(s) \right]$, and applying the Chernoff bound, we therefore have 
\begin{equation}\label{eqn:ePF}
\mathbb{P}\left(Z^*_n(T) \geq \eta n^2/2 \right) \leq \eee^{-n^2(I(\eta)+ o(1))/2}  
\qquad \forall\,\eta \geq \mathbb{E}\left( \sum_{k=1}^{Y^{(i)}} X_0^{(i,k)} \right)=C^* \beta T e^{C^*T}.  
\end{equation}
The claim now follows by observing that $\beta$ can be selected sufficiently small to make the inequality on the right-hand-side of \eqref{eqn:ePF} strict, in which case $I(\eta)>0$.
\end{proof}


\begin{lemma}
\label{lem:couple2}
The process $\{(G^*_n(t))_{t \geq 0}\}_{n \in \mathbb{N}}$ satisfies the conditions of Theorem \ref{thm:LDPmain}.
\end{lemma}

\begin{proof}
Assumption \ref{ass:DP} follows from Proposition \ref{lem:LDPDP}. The conditions of Proposition \ref{lem:TC} can be verified using similar (albeit simpler) arguments as in the proof of Proposition \ref{prop:DED} below. To verify Assumption \ref{ass:CP} we need to show that if $F_n \to F$ in $D(\mathcal{M}([0,1]), [0,T])$ then $g^{[F_n]} \to g^{[F]}$ in $D([\mathscr{W}, [0,T])$. Recall that, for any $u \in [0,1]$, $u' = F(t; \bar F(t +{\rm d}t;u)-{\rm d}t)$, and let 
\begin{equation}
u'_n=F_n(t; \bar F_n(t +{\rm d}t;u)-{\rm d}t).
\end{equation}
Below we assume that $t \in [0,T]$ is a continuity point of $F$. Applying \eqref{eq:GFdef} and the triangle inequality, we have
\begin{equation}
\label{eq:CPP1}
\begin{aligned}
\Delta_n&(t +{\rm d}t):=\left\lVert g^{[F]}(t+{\rm d}t; \cdot, \cdot) - g^{[F_n]}(t+{\rm d}t; \cdot, \cdot))\right\rVert_{L_1} \\
&\leq \left\lVert g^{[F]}(t; \cdot, \cdot) - g^{[F_n]}(t; \cdot, \cdot) \right\rVert_{L_1} \\
&+ {\rm d}t \int_{[0,1]^2}{\rm d}x {\rm d}y\, \bigg|\lambda\left(t, \bar F(t; x'),\bar F(t ; y'), F(t; \cdot) , \tilde g^{(F)}(t; \cdot)\right)(1-g^{(F)}(t;x',y')) \\
&\quad\hspace{8mm}- \lambda\left(t, \bar F_n(t ; x'_n),\bar F_n(t ; y'_n), F_n(t; \cdot) , \tilde g^{(F_n)}(t; \cdot)\right)(1-g^{(F_n)}(t;x'_n,y'_n)) \bigg| \\
&+ {\rm d}t \int_{[0,1]^2} {\rm d}x {\rm d}y\,\bigg|\mu\left(t, \bar F(t ; x'),\bar F(t ; y'), F(t; \cdot) , \tilde g^{(F)}(t; \cdot)\right)g^{(F)}(t;x',y') \\
&\quad\hspace{8mm}- \mu\left(t, \bar F_n(t; x'_n),\bar F_n(t ; y'_n), F_n(t; \cdot) , \tilde g^{(F_n)}(t; \cdot)\right)g^{(F_n)}(t;x'_n,y'_n) \bigg|.
\end{aligned}
\end{equation}
Using $F_n \to F$, the Lipschitz continuity of $\mu(\cdot)$ and $\lambda(\cdot)$ and the fact that for any $g,f \in \mathscr{W}$, $\delta_\square(\tilde g, \tilde f)\leq \lVert g-f \rVert_{L_1}$, we see that there exists $\bar K<\infty$ such that
\begin{equation}
\label{eq:CPP2}
\begin{aligned}
&\bigg| \lambda\left(t, \bar F(t ; x'),\bar F(t; y'), F(t; \cdot) , \tilde g^{(F)}(t; \cdot)\right) \\
&\qquad - \lambda\left(t, \bar F_n(t ; x'_n),\bar F_n(t ; y'_n), F_n(t; \cdot) , \tilde g^{(F_n)}(t; \cdot)\right) \bigg| \leq \bar K(\Delta_n(t) + o(1)),
\end{aligned}
\end{equation}
for almost all $(x,y) \in [0,1]^2$, with the same inequality holding for $\mu(\cdot)$. In addition because $\mu(\cdot)$ and $\lambda(\cdot)$ are Lipschitz continuous functions with compact support they are uniformly bounded by some constant $\widehat K<\infty$. Combining this with \eqref{eq:CPP1} and \eqref{eq:CPP2}, we obtain
\begin{equation}
\begin{aligned}
\label{eq:CPP3}
&\Delta_n(t+{\rm d}t) \leq \\
&\Delta_n(t) + 2 {\rm d}t \bigg[\bar K (\Delta_n(t) +o(1)) + \widehat K \int_{[0,1]^2} {\rm d}x\, {\rm d}y\, 
\left| g^{(F_n)}(t; x_n',y'_n) - g^{(F)}(t; x',y') \right| \bigg].
\end{aligned}
\end{equation}
Using the fact that $\Delta_n(0)=0$ and repeatedly applying \eqref{eq:CPP3}, we obtain
\begin{equation}
\begin{aligned}
\label{eq:CPP4}
&\Delta_n(t+{\rm d}t) \leq \\
&2 \int_{0}^{t + {\rm dt}} {\rm d}s \bigg[  \bar K (\Delta_n(s) +o(1)) + \widehat K \int_{[0,1]^2} {\rm d}x\, {\rm d}y \,\left| g^{(F_n)}(s; x_n',y'_n) - g^{(F)}(s; x',y') \right| \bigg].
\end{aligned}
\end{equation}
Because $F(\cdot;\cdot) \in D(\mathcal{M}([0,1]),[0,T])$ we can replace $x'$ and $y'$ by $x$ and $y$ in \eqref{eq:CPP4} almost everywhere. Since $F_n \to F$, we also have $x'_n \to x$ and $y'_n \to y$ almost everywhere. Consequently, from \eqref{eq:CPP4} we obtain
\begin{equation}
\begin{aligned}\label{eq:CPP5}
&\Delta_n(t+{\rm d}t) \leq \\
&2 \int_{0}^{t + {\rm dt}} {\rm d}s \bigg[  \bar K (\Delta_n(s) +o(1)) 
+ (\widehat K+o(1)) \int_{[0,1]^2} {\rm d}x\, {\rm d}y \,\left| g^{(F_n)}(s; x,y) - g^{(F)}(s; x,y) \right| \bigg] \\
&=2 \int_{0}^{t + {\rm dt}} {\rm d}s \Delta_n(s)(\bar K+\widehat K+o(1)).
\end{aligned}
\end{equation}
Because $\Delta_n(0)=0$, \eqref{eq:CPP5} implies that $\Delta_n(t) \to 0$ for all $t$, which completes the proof.
\end{proof}

\smallskip
\noindent
\emph{Proof of Theorem \ref{thm:CD}.}
Theorem \ref{thm:CD} now follows from Lemmas \ref{lem:couple} and \ref{lem:couple2}, in combination with Theorem \ref{thm:LDPmain}. 
\qed

\smallskip
\noindent
\emph{Proof of Proposition \ref{prop:DED}.}
Once we observe that $F^*$ is the unique zero of the rate function in Proposition \ref{lem:LDPDP}, we see that Proposition \ref{prop:CDLLN} follows from Lemmas \ref{lem:couple} and \ref{lem:couple2}, in combination with Theorem \ref{thm:CD}. \qed


\subsubsection{Proofs of the results in Section \ref{sec:ExDS1}}

We proceed by giving the proof of Proposition \ref{prop:DED}.

\smallskip
\noindent
\emph{Proof of Proposition \ref{prop:DED}.} It remains to establish exponential tightness using Proposition \ref{lem:TC}. Recall that $\lambda(\cdot)$ and $\mu(\cdot)$ are bounded functions and let $K<\infty$ denote their maximum. It directly follows that
\begin{equation}
\begin{aligned}
&\lim_{\delta \downarrow 0} \limsup_{n \to \infty}  \frac{1}{n} \log \mathbb{P} \left( \sup_{t \in [0,T]} C_n(t,\delta) > \varepsilon {n \choose 2} \right) \\
&\qquad \leq \lim_{\delta \downarrow 0} \limsup_{n \to \infty}  
\frac{1}{n}\log \mathbb{P} \left( \max_{i \in \{0, \dots, \lfloor \frac{T}{\delta} \rfloor \} } C_n(i\delta,2 \delta) > \varepsilon {n \choose 2} \right) \\
&\qquad \leq \lim_{\delta \downarrow 0} \limsup_{n \to \infty} 
\frac{1}{n} \log \left\{ \sum_{i=0}^{\lfloor \frac{T}{\delta} \rfloor} \max_{i \in \{0, \dots, \lfloor \frac{T}{\delta} \rfloor} 
\mathbb{P}\left(  C_n(i \delta, 2 \delta)> \varepsilon {n \choose 2} \right) \right\}.
\end{aligned}
\end{equation}
Note that there are two ways that edges can change during the interval $[i\delta, (i+2)\delta]$. In the first place, the Poisson clock of one of the adjacent vertices of the rings during $[i\delta, (i+2)\delta]$. We let $C^{(v)}_n(t,\delta)$ denote the number of edges (active or inactive) that are adjacent to such a vertex. Because the probability  each vertex that each vertex rings during $[i\delta, 2 \delta]$ is $1-e^{-\gamma t}$, $C^{(v)}_n(t,2\delta)$ is stochastically dominated by a random variable $X^{(v)}_n(\delta) \sim  n \cdot {\rm Bin}(n, 1-e^{-\gamma \delta})$. In the second place, the value of $U_{ij}$ falls within the interval 
\begin{equation}
\begin{aligned}
\bigg[ \min_{u\in[0,2\delta]} H(t + u, X_i(t)+u, &X_j(t)+u, F_n(t+u)),\\
&\qquad\max_{u \in [0,2\delta]} H(t+u, X_i(t)+u, X_j(t)+u, F_n(t+u) \bigg].
\end{aligned}
\end{equation}
We let $C_n^{(e)}(i\delta,2 \delta)$ denote the number of such edges. Because both $\lambda(\cdot)$ and $\mu(\cdot)$ are bounded by a constant $K$, it follows that 
\begin{equation}
\begin{aligned}
\max_{u \in [0, 2\delta]} H(t + u, &X_i(t)+u, X_j(t)+u, F_n(t+u))\\
&-\min_{u \in [0,2\delta]} H(t+u, X_i(t)+u, X_j(t)+u, F_n(t+u)) \leq 4K\delta.
\end{aligned}
\end{equation}
Consequently, $C_n^{(e)}(i\delta, 2 \delta)$ is stochastically dominated by a random variable $X^{(e)}_n(\delta) \sim {\rm Bin}({n \choose 2}, 4 K \delta)$.
We thus have 
\begin{equation}
\begin{aligned}
&\lim_{\delta \downarrow 0} \limsup_{n \to \infty} \frac{1}{n} \log \left\{ \sum_{i=0}^{\lfloor \frac{T}{\delta} \rfloor} 
\max_{i \in \{0, \dots, \lfloor \frac{T}{\delta} \rfloor} \mathbb{P}\left(  C_n(i \delta, 2 \delta)> \varepsilon {n \choose 2} \right) \right\} \\
&\leq  \lim_{\delta \downarrow 0} \limsup_{n \to \infty} \frac{1}{n} \log \left\{ \sum_{i=0}^{\lfloor \frac{T}{\delta} \rfloor} 
\max_{i \in \{0, \dots, \lfloor \frac{T}{\delta} \rfloor} \mathbb{P}\left(  C^{(v)}_n(i \delta, 2 \delta)> \frac{\varepsilon}{2} {n \choose 2} 
\right)+\mathbb{P}\left(  C^{(e)}_n(i \delta, 2 \delta)> \frac{\varepsilon}{2} {n \choose 2} \right) \right\} \\
&\leq  \lim_{\delta \downarrow 0} \limsup_{n \to \infty} \frac{1}{n} \log 
\frac{T}{\delta}\left\{ \mathbb{P}\left(  X_n^{(v)}(\delta)> \frac{\varepsilon}{2} {n \choose 2} \right)
+\mathbb{P}\left(  X_n^{(e)}(\delta)> \frac{\varepsilon}{2} {n \choose 2} \right) \right\} \\
&\leq  \lim_{\delta \downarrow 0} \limsup_{n \to \infty} \frac{1}{n} \log 
\frac{T}{\delta}\left\{ \exp\left({-n\frac{\varepsilon}{2}\left(\log \frac{\varepsilon}{1-\eee^{-\gamma \delta}}-1\right)}\right)
+\exp\left({-{n \choose 2}\frac{\varepsilon}{2}\left(\log \frac{\varepsilon}{4K\delta}-1\right)}\right) \right\} \\
&= - \infty,
\end{aligned}
\end{equation}
where in the last inequality we apply \cite[Theorem 2.3]{{M98}} (which is essentially a Chernoff bound). \qed


\subsubsection{Proofs of the results in Section \ref{sec:ExDS3}}

The proof of Proposition \ref{prop:convex} relies on the following direct computation.

\smallskip\noindent
\emph{Proof of Proposition \ref{prop:convex}.}
Pick $P_1, P_2 \in \mathcal{M}([0,T])$ and suppose that 
\begin{equation}
\label{eq:EDA}
2 \int_0^1 \int_0^y x \,P_i({\rm d}x) \,P_i({\rm d} y) \leq e^*, \qquad i=1,2.
\end{equation}
Observe that if $X_i^{(k)}$ are independent random variables with distribution $P_i$, then 
\begin{equation}
2 \int_0^1 \int_0^y x \,P_i({\rm d}x) \,P_i({\rm d} y) = \mathbb{E}(X_i^{(1)} \wedge X_i^{(2)}).
\end{equation}
Let $P_3 = cP_1+(1-c)P_2$ with $c \in [0,1]$. We have 
\begin{align}
\begin{split}
2 \int_0^1 \int_0^y x \,
&P_3({\rm d}x) \,P_3({\rm d} y) = \mathbb{E}(X_3^{(1)} \wedge X_3^{(2)}) \\
&= c^2 \mathbb{E}(X_1^{(1)} \wedge X_1^{(2)}) + (1-c)^2 \,\mathbb{E}(X_2^{(1)} \wedge X_2^{(2)})
+ 2c(1-c) \, \mathbb{E}(X_1^{(1)} \wedge X_2^{(1)}) \\
&\leq e^*(c^2+(1-c)^2) + 2c(1-c) \, \mathbb{E}(X_1^{(1)} \wedge X_2^{(1)}).
\end{split}
\end{align}
Hence it remains to show that $\mathbb{E}(X_1^{(1)} \wedge X_2^{(1)}) \leq e^*$. We have 
\begin{align}
\begin{split}
\mathbb{E}(X_1^{(1)} \wedge X_2^{(1)}) &= \int^1_0  {\rm d}x\, 
\mathbb{P} (X_1 \geq x)\, \mathbb{P} (X_2 \geq x) \\
&\leq \left( \int^1_0  {\rm d}x\, \mathbb{P} (X_1 \geq x)^2 \int^1_0  {\rm d}x\, \mathbb{P} (X_2 \geq x)^2 \right)^{1/2} \leq e^*,
\end{split}
\end{align}
where in the second step we apply the Cauchy-Schwarz inequality, and in the final step use \eqref{eq:EDA}.
\qed


\appendix


\section{Rate function for the driving process}
\label{appA}

To establish an LDP for the illustrative example in Section \ref{sec:example} and the examples in Sections \ref{Sec:APP1} and \ref{sec:ExDS1}, we need to verify Assumption \ref{ass:LDPtypes}, i.e., we need to establish an LDP for the driving process. The latter are dealt with in Appendix \ref{appA1}, the former in Appendix \ref{appA2}.


\subsection{LDP for driving process in Section \ref{sec:ExSD}}
\label{appA1}

Recall that in Sections \ref{Sec:APP1} and \ref{sec:ExDS1} each vertex $v$ is assigned a Poisson clock that rings at times $\{\tau_k(v)\}_{k \in \mathbb{N}}$, and 
\begin{equation}
X_v(t) := t - \max_k \{ \tau_k(v) : \tau_k(t) \leq t \},
\end{equation}
with $\mu_n(t)= \frac{1}{n} \sum^n_{i=1} \delta_{X_i(t)}$. Here we drop the restriction that $X_i(t) \leq 1$, and suppose that $\mu_n(t) \in \mathcal{M}(\mathbb{R}_+)$. We also suppose that  
\begin{equation}
\label{DGa}
\mu_n(0) = \frac{1}{n}\sum_{i=1}^n \delta(X_i(0)) \to v, \qquad n \to \infty,
\end{equation}
in $\mathcal{M}(\mathbb{R}_+)$. 
For $A \subseteq \mathbb{R}_+$ and $t>0$, we let 
\begin{equation}
\label{D1def}
D_{\bs 1} \mu_t(A) = \lim_{h \to 0} \frac{\mu_{t+h}(A+h) - \mu_{t}(A) }{h}.
\end{equation}

\begin{proposition}
\label{lem:LDPDP}
The sequence of processes $(\mu_n)_{n \in \mathbb{N}}$ satisfies the LDP on $D([\mathcal{M}(\mathbb{R}_+),[0,T])$ with rate $n$ and with rate function
\begin{align}
\begin{split}
\label{eq:rateex1}
K(\mu) 
&=\int^T_0 {\rm d}t \int_0^\infty 
[\gamma \mu_i({\rm d}x)- D_{\bs 1} \mu_t ({\rm d}x)]
+ \int_0^T {\rm d} t \int^{\infty}_0  
D_{\bs 1} \mu_t( {\rm d}x) \log \left( \frac{D_{\bs 1} 
\mu_t( {\rm d}x)}{f_t(0)\mu_i({\rm d}x)} \right)\\
&\qquad\qquad\qquad 
+ \int^T_0 {\rm d}t\, f_t(0) \log (f_t(0)/\gamma),
\end{split}
\end{align}
where $f_t(0) := \lim_{h \to 0} \mu_t([0,h])/h$ if $\mu$ is absolutely continuous and $\infty$ otherwise.
\end{proposition}

\begin{proof}
We prove Proposition \ref{lem:LDPDP} by applying the standard method of proving sample-path LDPs: 
\begin{itemize}
\item[(I)] 
Establish a finite-dimensional LDP for times $0<t_1<\dots < t_k<T$.
\item[(II)] 
Write down a limiting expression for the rate function when the gaps between the times shrink to zero.
\item[(III)] 
Strengthen the topology by establishing exponential tightness.
\end{itemize}

\medskip
\noindent{\underline{Step (I):}} Let 
\begin{equation}
P_x^{(\bs t)}({\rm d} y^{(1)}, \dots,  {\rm d} y^{(r)}) = \mathbb{P}(X_i(t_1) \in {\rm d}y^{(1)}, \dots, X_i(t_r) \in  {\rm d} y^{(r)} | X_i(0)=x),
\end{equation}
where $\bs t = (t_1, \dots, t_r)$ with $0 < t_1 < t_2 < \dots < t_r < T$. We apply the following result, which is an immediate consequence of \cite[Theorem 3.5]{DG87}: \textit{If \eqref{DGa} holds, then the sequence of measures $(\mathbb{P}((\mu_n(t_1), \dots, \mu_n(t_r)) \in \cdot))_{n\in\mathbb{N}}$ satisfies the LDP with rate $n$ and with rate function
\begin{equation}\label{DGtheorem}
\begin{aligned}
&I^{(\bs t)}_v(\mu_1, \dots, \mu_r) \\
&= \sup_{f_1, \dots, f_r \in C_b(\mathbb{R}_+)^r} \bigg[ \sum_{i=1}^r \int_{\mathbb{R}_+} \mu_i({\rm d}z) f_i(z)\\ 
&\qquad \qquad - \int_{\mathbb{R}_+} v({\rm d}x) \log \int_{\mathbb{R}_+^r} P^{(\bs t)}_x({\rm d}y^{(1)}, \dots ,{\rm d}y^{(r)}) 
\exp \left( \sum_{i=1}^r f_i(y^{(i)}) \right) \bigg],
\end{aligned}
\end{equation} 
where $(\mu_1, \dots, \mu_r) \in \mathcal{M}(\mathbb{R}_+)^r$.}

To apply \eqref{DGtheorem}, we need to write down a formula for $P_x^{(\bs t)}({\rm d} y^{(1)}, \dots,  {\rm d} y^{(r)})$. By the Markov property, this is essentially equivalent to writing down an expression for $P_x^{(t)}({\rm d}y)$, i.e., for a single time step. If $X_i(0)=x$, then the probability that $X_i(t)=x+t$ is $\eee^{-\gamma t}$ (i.e., the probability that the Poisson clock associated with vertex $i$ does not ring in the time interval $[0,t]$). On the other hand, if $y \leq t$, then the probability that $X_i(t) \in {\rm d} y$ is the probability that the Poisson clock associated with vertex $i$ rings in the time interval $[t-{\rm d}y,t]$ (which occurs with probability $\gamma\,{\rm d}y$), and afterwards does not ring again (which occurs with probability $\eee^{-\gamma y}$). We thus have
\begin{equation}
P^{(t)}_x({\rm d}y) 
= \begin{cases}
\eee^{-\gamma t} \qquad &\text{if } y=x+t, \\
\gamma {\rm d} y\, \eee^{-\gamma y} \qquad &\text{if } y \leq t, \\
0  &\text{otherwise}.
\end{cases}
\end{equation}
If we apply \eqref{DGtheorem} for a single time step, then we obtain
\begin{align}
I_v^{(t)}(\mu) 
&= \sup_{f \in C_b([0,\infty))} \left[ \int_0^\infty \mu({\rm d}z) f(z) 
- \int_0^\infty v({\rm d} x) \log \left( \int_0^\infty P^{(t)}_x({\rm d}y) \eee^{f(y)} \right) \right] \nonumber \\
&= \sup_{f \in C_b[0,\infty)} \left[ \int_0^\infty \mu({\rm d} z) f(z) 
- \int_0^\infty v({\rm d}x) \log \left( \eee^{-\gamma t + f(x+t)} 
+ \int_0^t {\rm d}y \,\gamma \eee^{-\gamma y +f(y)} \right) \right].
\label{DZc}
\end{align}
We would like to derive a closed form expression for $I^{(t)}(\mu)$. To do this, we consider the term under the supremum, i.e.,
\begin{equation}
\label{JEq}
\int_0^\infty \mu({\rm d} z) f(z) 
- \int_0^\infty v({\rm d}x) \log \left( \eee^{-\gamma t + f(x+t)} 
+ \int_0^t {\rm d}y \,\gamma \eee^{-\gamma y +f(y)} \right),
\end{equation}
take the derivative with respect to $f(x+t)$ for fixed $x \geq 0$, and set this to zero. This gives
\begin{equation}
\mu( {\rm d}(x+t)) -  v( {\rm d} x)\, \frac{\eee^{-\gamma t + f(x+t)}}{\eee^{-\gamma t + f(x+t)}
+ \int_0^t {\rm d}y\,\gamma\, \eee^{-\gamma y + f(y)}}=0,
\end{equation}
which leads to
\begin{equation}
\label{fxtE}
f(x+t) = \log \left( \frac{\mu( {\rm d}(x+t)) \int_0^t {\rm d}y\, \gamma\, \eee^{-\gamma y 
+ f(y)}}{[v( {\rm d} x)-\mu( {\rm d}(x+t)) ]\,\eee^{-\gamma t }} \right).
\end{equation}
Substituting \eqref{fxtE} into \eqref{JEq}, we obtain
\begin{align}
\begin{split}
\eqref{JEq} 
&= \int_0^{t-} \mu({\rm d}z)\, f(z) + \int_0^\infty \mu({\rm d}(x+t))\log \left( \frac{\mu( {\rm d}(x+t)) 
\int_0^t {\rm d}y\, \gamma\, \eee^{-\gamma y + f(y)}}{[v({\rm d} x)-\mu( {\rm d}(x+t)) ]\,\eee^{-\gamma t }} \right) \\
&\qquad - \int_0^\infty v({\rm d} x)\, \log \left( \eee^{-\gamma t} \frac{\mu( {\rm d}(x+t))
\int_0^t {\rm d}y\, \gamma \eee^{-\gamma y + f(y)}}{[v( {\rm d} x)-\mu( {\rm d}(x+t)) ]\,\eee^{-\gamma t }} 
+ \int^t_0 {\rm d}y\, \gamma \eee^{-\gamma y + f(y)}\right)  \\
&= \int_0^{t-} \mu({\rm d}z) f(z) + \int_0^\infty \mu({\rm d}(x+t))\log \left( \frac{\mu( {\rm d}(x+t)) 
\int_0^t {\rm d}y\, \gamma\, \eee^{-\gamma y + f(y)}}{[v( {\rm d} x)-\mu( {\rm d}(x+t)) ]\,\eee^{-\gamma t }} \right) \\
&\qquad - \int_0^\infty v({\rm d} x) \log \left(  \frac{v( {\rm d}x)}{v( {\rm d} x)-\mu( {\rm d}(x+t))} 
\int^t_0 {\rm d}y\, \gamma\, \eee^{-\gamma y + f(y)}\right)  \\
&=\int_0^\infty \mu({\rm d}(x+t))\log \left( \frac{\mu( {\rm d}(x+t))}{[v( {\rm d} x)-\mu( {\rm d}(x+t)) ]\, \eee^{-\gamma t }} \right) \\
&\qquad - \int_0^\infty v({\rm d} x)\, \log \left(  \frac{v( {\rm d}x)}{v( {\rm d} x)-\mu( {\rm d}(x+t))}\right) \\
&\qquad +\int_0^{t-} \mu({\rm d}z)\, f(z) - \int_0^\infty [v({\rm d}x)- \mu( {\rm d}(x +t))] 
\log \left( \int^t_0 {\rm d} y\, \gamma\, \eee^{-\gamma {y + f(y)}} \right).
\end{split}
\end{align}

We next optimise over $f(z)$, $0 \leq z < t$ (note that previously we optimised over $f(z)$, $t \leq z \leq \infty$). To do this, we consider the last line, which reads
\begin{align}
&\int_0^{t-} \mu({\rm d}z)\, f(z) - \int_0^\infty [v({\rm d}x)- \mu( {\rm d}(x +t))] \log 
\left( \int^t_0 {\rm d} z\, \gamma \eee^{-\gamma z + f(z)} \right) 
\label{fuwrE} \\
&\qquad =\int_0^{t-} \mu({\rm d}z)\, f(z) - \left(1 - \int_{t+}^\infty \mu( {\rm d}x) \right) \log 
\left( \int^t_0 {\rm d} z\, \gamma \eee^{-\gamma z + f(z)} \right) ,
\label{flwrE}
\end{align}
and take the derivative with respect to $f(z)$ for fixed $z \in [0,t]$, and set this to zero. This gives
\begin{equation}
\mu({\rm d}z) -  \left(1 - \int_{t+}^\infty \mu( {\rm d}x) \right) \frac{{\rm d} z\, \gamma 
\eee^{-\gamma z + f(z)}}{\int_0^t {\rm d}z\, \gamma \eee^{\gamma z + f(z)}}=0,
\end{equation}
which implies that 
\begin{equation}
\label{fzLwr}
f(z) = \log \left(  \frac{\mu({\rm d}z) \int_0^t {\rm d}y\, \gamma\, \eee^{-\gamma y + f(y)} }
{\left(1 - \int_{t+}^\infty \mu( {\rm d}x) \right) {\rm d} z\, \gamma\, \eee^{-\gamma z}} \right).
\end{equation}
Substituting \eqref{fzLwr} into \eqref{flwrE}, we obtain
\begin{align}
\begin{split}
&\int_0^{t-} \mu( {\rm d} z)\, \log \left(  \frac{\mu({\rm d}z) \int_0^t {\rm d}y\, \gamma\, \eee^{-\gamma y + f(y)} }
{\left(1 - \int_{t+}^\infty \mu( {\rm d}x) \right) {\rm d} z\, \gamma \eee^{-\gamma z}} \right) 
- \left(1 - \int_{t+}^\infty \mu( {\rm d}x) \right) \log \left( \int^t_0 {\rm d} z\, \gamma\, \eee^{-\gamma z + f(z)} \right) \\
&= \int^{t-}_0 \mu({\rm d} z)\, \log \left( \frac{\mu({\rm d}z)}{{\rm d}z\, \gamma \eee^{-\gamma z}} \right) 
- \int^{t-}_0 \mu({\rm d} z)\, \log \left(\int^{t-}_0 \mu({\rm d} z) \right).
\end{split}
\end{align}
Combining this with \eqref{fuwrE}, we obtain
\begin{align*}
\eqref{JEq}
&=\int_0^\infty \mu({\rm d}(x+t))\log \left( \frac{\mu( {\rm d}(x+t))}{[v( {\rm d} x)-\mu( {\rm d}(x+t))]\,\eee^{-\gamma t }} \right)\\ 
&\quad - \int_0^\infty v({\rm d} x) \log \left(  \frac{v( {\rm d}x)}{v( {\rm d} x)-\mu( {\rm d}(x+t))}\right) \\
&\quad + \int^{t-}_0 \mu({\rm d} z) \log \left( \frac{\mu({\rm d}z)}{{\rm d} z\, \gamma\, \eee^{-\gamma z}} \right) 
- \int^{t-}_0 \mu({\rm d} z)\, \log \left(\int^{t-}_0 \mu({\rm d} z) \right)\\
&=\int_0^\infty [v({\rm d}x) - \mu( {\rm d}(x+t))] \log \left( v({\rm d} x) - \mu( {\rm d} (x +t) )\right)\\ 
&\quad + \int^\infty_0 \mu( {\rm d}(x+t)) \log \left( \frac{\mu( {\rm d}(x+t))}{\eee^{-\gamma t}} \right)\\
&\quad - \int^\infty_0 v({\rm d} x) \log ( v( {\rm d} x)) - \int^{t-}_0 \mu( {\rm d}z)\, \log \left( \int^{t-}_0 \mu( {\rm d} z)  \right) \\
&\quad+\int^{t-}_0 \mu( {\rm d}z)\, \log \left( \frac{\mu( {\rm d} z)}{{\rm d}x\, \gamma\, \eee^{- \gamma z}}  \right) \\
&= \int^\infty_0 v({\rm d}x) \left[ \frac{v({\rm d}x) - \mu({\rm d}(x+t))}{v({\rm d}x)} 
\log \left( \frac{v({\rm d}x) - \mu( {\rm d}(x+t))}{v({\rm d}x)} \right) \right.\\ 
&\quad \quad \left. + \frac{\mu({\rm d}(x+t))}{v({\rm d}x)} \log \left( \frac{\mu({\rm d}(x+t))}
{\eee^{-\gamma t}v({\rm d} x)} \right)\right] \\
&\quad - \left( \int^{t-}_0 \mu( {\rm d}z) \right) \log \left( \int^{t-}_0 \mu( {\rm d} z)  \right) 
+ \int^{t-}_0 \mu( {\rm d}z) \log \left( \frac{\mu( {\rm d} z)}{{\rm d}z\, \gamma\, \eee^{- \gamma z}}  \right).
\end{align*}

Rearranging further, we obtain the following.

\begin{lemma}
\begin{align}
\begin{split}\label{FDRF}
I^{(t)}_v(\mu) &=  \int^\infty_0 \mu({\rm d}(x+t)) \log \left( \frac{\mu({\rm d}(x+t))}{v({\rm d} x)\,\eee^{-\gamma t}} \right) \\
&\qquad + \int^\infty_0 [v({\rm d}x) - \mu({\rm d}(x+t))] \log \left( \frac{v({\rm d}x) - \mu( {\rm d}(x+t))}{v({\rm d}x) 
\int^{t-}_0 \mu( {\rm d} z)} \right) \\
&\qquad + \int^{t-}_0 \mu( {\rm d}z) \log \left( \frac{\mu( {\rm d} z)}{{\rm d}z\, \gamma \eee^{- \gamma z}}  \right).
\end{split}
\end{align}
\end{lemma}

Note that if 
\[
v({\rm d}x) = \begin{cases}
1, \qquad \text{if }x=0, \\
0, \qquad \text{otherwise},
\end{cases}
\]
then for $\mu$ absolutely continuous with respect to $P^{(t)}_v$ we have
\begin{align*}
I_v^{(t)}(\mu) &= \mu({\rm d}t) \log \left( \frac{\mu({\rm d}t)}{\eee^{-\gamma t}} \right) +(1-\mu({\rm d}t)) \log\left( \frac{1 - \mu({\rm d}t)}{1 - \mu({\rm d}t)} \right) + \int^{t-}_0 \mu({\rm d}z) \log \left( \frac{\mu({\rm d}z}{{\rm d}z\, \gamma \eee^{-\gamma z}} \right) \\
&=\mu({\rm d}t) \log \left( \frac{\mu({\rm d}t)}{\eee^{-\gamma t}} \right) + \int^{t-}_0 \mu({\rm d}z) \log \left( \frac{\mu({\rm d}z)}{{\rm d}z\, \gamma \eee^{-\gamma z}} \right), 
\end{align*}
which is the relative entropy from $P^{(t)}_v$ to $\mu$.

The above arguments extend naturally to establish the following finite-dimensional large deviation principle.

\begin{lemma}
If \eqref{DGa} holds, then the sequence of measures $(\mathbb{P}((\mu_n(t_1), \dots, \mu_n(t_r)) \in \cdot))_{n\in\mathbb{N}}$ satisfis the LDP with rate $n$ and with rate function $I^{(\bs t)}_v(\mu) = \sum_{i=1}^r I^{(t_i-t_{i-1})}_{\mu_{i-1}}(\mu_i)$, where $t_0=0$ and $\mu_0=v$.
\end{lemma}

\medskip 
\noindent
\underline{Step (II):}
By the Dawson-G\"{a}rtner projective limit LDP (\cite[Theorem 4.6.1]{DZ98}), the sequence of measures $(\mathbb{P}(\mu_n \in \cdot))_{n\in\mathbb{N}}$ satisfies the LDP in the pointwise topology with rate function 
\begin{equation}
I_v^{(pw)}(\mu)=\sup_{0=t_0<t_1<\dots <t_k=T} I^{(t_0, \dots, t_k)}_v(\mu(t_0), \dots, \mu(t_k)). 
\end{equation}
The next step is therefore to prove the following lemma.

\begin{lemma}
$K(\mu)=I_v^{(pw)}(\mu)$ for all $\mu \in D( \mathcal{M}(\mathbb{R}),[0,T])$ with $\mu(0)=v$.
\end{lemma}

\begin{proof}
The proof comes in two steps.

\medskip\noindent
$\bullet$ 
We start by demonstrating that $I^{(pw)}_v(\mu) \geq K(\mu)$ for any $\mu \in \mathcal{M}( \mathbb{R}^+) \times [0,T]$. Consider the times $0<t_1=\Delta<t_2=2\Delta<\dots<t_k=k\Delta = T$, and let $\mu_i = \mu(t_i) \in \mathcal{M}( \mathbb{R}_+)$. The rate function of the finite-dimensional LDP is 
\begin{align}
I^{(\bs t)}_v(\bs \mu) &= \sum^k_{i=1} I_{\mu_{i-1}}^{(t_i-t_{i-1})}(\mu_i) \\
&= \sum_{i=1}^k \bigg[ \int^\infty_0 \mu_i({\rm d}(x+\Delta)) \log 
\left( \frac{\mu_i({\rm d}(x+\Delta))}{\mu_{i-1}({\rm d} x)e^{-\gamma \Delta}} \right) \label{LT1}\\
&\qquad +\int^\infty_0 [\mu_{i-1}({\rm d}x) - \mu_i({\rm d}(x+\Delta))] \log \left( \frac{\mu_{i-1}({\rm d}x) 
- \mu_i( {\rm d}(x+\Delta))}{\mu_{i-1}({\rm d}x) \int^{\Delta-}_0 \mu_i( {\rm d} z)} \right) \label{LT2} \\
&\qquad + \int^{\Delta-}_0 \mu_i( {\rm d}z) \log \left( \frac{\mu_i( {\rm d} z)}{{\rm d}z\, \gamma e^{- \gamma z}}  \right) \bigg]. 
\label{LT3}
\end{align}
We will deal with \eqref{LT1}, \eqref{LT2}, \eqref{LT3} separately. 

Recall the definition of $D_{\bs 1}$ from \eqref{D1def}. For the first term we have 
\begin{align}
\begin{split}
\eqref{LT1}&= \sum_{i=0}^{k-1}  \int^\infty_0 \mu_{i+1}({\rm d}(x+\Delta)) \log \left( \frac{\mu_{i+1}({\rm d}(x+\Delta))}{\mu_{i}({\rm d} x)\,\eee^{-\gamma \Delta}} \right) \\
&= \sum_{i=0}^{k-1}  \int^\infty_0 \mu_{i+1} ( {\rm d}(x+\Delta)) \log \left( 1 + \frac{\mu_{i+1}({\rm d}(x+\Delta)) 
- \mu_i( {\rm d}x) \eee^{-\gamma \Delta}}{\mu_i({\rm d}x)\, \eee^{-\gamma \Delta}} \right) \\
&= o(1)+  \sum_{i=0}^{k-1}  \int^\infty_0 \frac{\mu_{i+1} ( {\rm d}(x+\Delta))}{\mu_i({\rm d}x)\,\eee^{-\gamma \Delta}} [\mu_{i+1}({\rm d}(x+\Delta)) - \mu_i( {\rm d}x) (1-\gamma \Delta)] \\
&= o(1) + \sum_{i=0}^{k-1}  \sum_{i=0}^{k-1}  \int^\infty_0 \Delta [\gamma \mu_i({\rm d}x)- D_{\bs 1} \mu_t ({\rm d}x)] \\
&\to \int^T_0 {\rm d}t \int_0^\infty [\gamma \mu_i({\rm d}x)- D_{\bs 1} \mu_t ({\rm d}x)].
\end{split}
\end{align}
For the second term we have
\begin{align}
\begin{split}
\eqref{LT2} &=\sum_{i=0}^{k-1} \int^\infty_0 [\mu_{i}({\rm d}x) - \mu_{i+1}({\rm d}(x+\Delta))] 
\log \left( \frac{\mu_{i}({\rm d}x) - \mu_{i+1}( {\rm d}(x+\Delta))}{\mu_{i}({\rm d}x) \int^{\Delta-}_0 \mu_{i+1}({\rm d} z)} \right) \\
&\to \int_0^T {\rm d} t \int^{\infty}_0 D_{\bs 1} \mu_t( {\rm d}x) \log\left( \frac{D_{\bs 1} \mu_t( {\rm d}x)}{\mu_i({\rm d}x) f_t(0)} \right) ,
\end{split}
\end{align}
where $f_t(0) := \lim_{h \to 0} \mu_t([0,h])/h$. Finally, for the third term we have 
\begin{align}
\eqref{LT3} &= \sum_{i=1}^k \int^{\Delta-}_0 \mu_i({\rm d}z)\, \log \left( \frac{\mu_i( {\rm d} z)}{{\rm d}z\, \gamma \eee^{- \gamma z}}  \right) \to \int^T_0 {\rm d}t f_t(0) \log (f_t(0)/\gamma).
\end{align}
Combining the three terms, we get the expression in \eqref{eq:rateex1}.

\medskip\noindent
$\bullet$
To prove that $I_v^{(pw)}(\mu) \leq K(\mu)$, suppose to the contrary that $I_v^{(pw)}(\mu) > K(\mu)$ for some $\mu \in \widebar{\mathcal{M}}_X$. In that case there must exist a vector $\bs{t}=(t_0, \dots, t_k)$ such that 
\begin{equation}
\label{eq:Cont1}
I^{(\bs{t})}_v(\mu(t_0), \dots, \mu(t_k))> K(\mu). 
\end{equation}
For $\ell \in \mathbb{N}$, let $\bs{s}^\ell = (s_1^{[\ell]}, \dots, s_{k_\ell}^{[\ell]}$) be such that 
\begin{itemize}
\item for any $i \in \{0, \dots, k\}$ there exists $j \in \{0, \dots, k_\ell\}$ with $t_i=s^{[\ell]}_j$,
\item $\lim_{\ell\to\infty} \max_{j \in \{1, \dots, k_\ell\}} | s_k^{[\ell]}- s_{j-1}^{[\ell]}| = 0$.
\end{itemize}
By the contraction principle, for any $\ell \geq 1$, we have 
\begin{equation}
\label{eq:Cont2}
I^{(\bs{t})}_v(\mu(t_0), \dots, \mu(t_k)) \leq I^{(\bs{s}^{[\ell]})}_v(\mu(s^{[\ell]}_0), \dots, \mu(s^{[\ell]}_{k_\ell})),
\end{equation}
On the other hand, following very similar arguments as those above, we get
\begin{equation}
\lim_{\ell \to \infty} I^{(\bs{s}^{[\ell]})}_v(\mu(s^{[\ell]}_0), \dots, \mu(s^{[\ell]}_{k_\ell})) = K(\mu),
\end{equation}
which in combination with \eqref{eq:Cont2} contradicts \eqref{eq:Cont1}. The proof that $I^{(pw)}_v(\mu)=\infty$ when $\mu$ is not absolutely continuous follows from standard arguments and is therefore omitted.
\end{proof}

\medskip
\noindent
\underline{Step (III):} We now establish exponential tightness.

\begin{lemma}
The sequence of measures $\mathbb{P}(\mu_n(\cdot) \in \cdot)$ is exponentially tight in $D(\mathcal{M}(\mathbb{R}_+),[0,T])$.
\end{lemma}
\begin{proof}
By \cite{FK06} it suffices to show that
\begin{equation}
\lim_{\delta \to 0} \limsup_{n \to \infty} \frac{1}{n} \log \mathbb{P} (w'(\mu_n, \delta, T) > \varepsilon) = - \infty,
\end{equation}
where
\begin{equation}
w'(x,\delta, T) = \inf_{\{t_i\}} \max_i \sup_{s,t \in [t_{i-1}, t_i)} r(x(s),x(t))
\end{equation}
and $r$ is a metric which generates the weak topology. We equip the space $\mathcal{M}(\mathbb{R}_+)$ with the L\'{e}vy metric, so that for two distribution functions $F,G$ we have
\begin{equation}
r(F,G)= \inf \{ \varepsilon > 0 | F(x-\varepsilon) \leq G(x) \leq F(x+ \varepsilon)+ \varepsilon, \, \forall x \in \mathbb{R} \}.
\end{equation}
Let $C^{(v)}(s,t)$ denote the number of vertices that turn off at some point in the interval $[s,t]$.
Observe that 
\begin{equation}\label{eq:obmet}
\sup_{s,t \in [t_{i-1}, t_i} r(F_n(s), F_n(t)) \leq \frac{C^{(v)}(s,t)}{n} + t_{i}-t_{i-1}.
\end{equation}
Given \eqref{eq:obmet} the result then follow from very similar arguments to those used in the proof of Proposition \ref{prop:DED}.
\end{proof}

\end{proof}


\subsection{LDP for driving process in Section \ref{sec:example}}
\label{appA2}

Define $\mu^*$ by letting 
\begin{equation}
\label{eq:ETPT}
\mu^*_t(A) = \mu_t(F^{\exp}(A))
\end{equation}
for all $t >0$ and $A \subset [0,1]$, where $F^{\exp}(\cdot)$ is the cdf of and exponential random variable with rate $\gamma$. Note that with this transformation we recover the empirical type process in Section \ref{sec:example}. 

\begin{proposition}
The sequence of processes $(\mu^*_n)_{n \in \mathbb{N}}$ satisfies the LDP on $D(\mathcal{M}([0,1]),[0,T])$ with rate $n$ and with rate function
\begin{equation}
K^*(\mu^*)=\inf_{\mu\in D(\mathcal{M}(\mathbb{R}_+),[0,T]): F^{\exp}(\mu)=\mu^*} K(\mu),
\end{equation}
where $K(\cdot)$ is defined in Proposition \ref{lem:LDPDP}.
\end{proposition}

\begin{proof}
Because the transformation in \eqref{eq:ETPT} is continuous, the result is a direct application of Proposition \ref{lem:LDPDP} and the contraction principle (see \cite{dH00}).
\end{proof}


\end{document}